\documentclass[10pt]{article}
\usepackage[utf8]{inputenc} 
\usepackage[T1]{fontenc}
\usepackage{lmodern}

\usepackage{amsthm,amsmath,amsfonts,amssymb,stmaryrd,bbm}
\usepackage{mathtools}
\usepackage{enumitem}
\usepackage{mathrsfs} % for mathscr font

\usepackage{subcaption} % for captioning subfigures

\usepackage{graphicx}
\usepackage[dvipsnames]{xcolor}
\usepackage{caption,tikz}

\usepackage{anyfontsize}

\usepackage[english]{babel}
\usepackage[colorlinks=true, linkcolor=blue, urlcolor=black, citecolor=blue,pdfstartview=FitH]{hyperref}

\numberwithin{equation}{section}

\let\originalleft\left
\let\originalright\right
\renewcommand{\left}{\mathopen{}\mathclose\bgroup\originalleft}
\renewcommand{\right}{\aftergroup\egroup\originalright}

\newlength{\bibitemsep}
\newlength{\bibparskip}
\let\oldthebibliography\thebibliography
\renewcommand\thebibliography[1]{\oldthebibliography{#1}
  \setlength{\parskip}{\bibitemsep}
  \setlength{\itemsep}{\bibparskip}}

\DeclareMathOperator{\re}{Re}
\DeclareMathOperator{\im}{Im}
\DeclareMathOperator{\Tr}{Tr}

\DeclareMathOperator{\Crt}{Crt}
\DeclareMathOperator{\Id}{Id}
\DeclareMathOperator{\diag}{diag}
\DeclareMathOperator{\supp}{supp}
\DeclareMathOperator{\Spec}{Spec}

\newcommand{\mc}[1]{\mathcal{#1}}
\newcommand{\mf}[1]{\mathfrak{#1}}
\newcommand{\ms}[1]{\mathscr{#1}}

\newcommand{\rd}{{\rm d}}

\newcommand{\ii}{\mathrm{i}}

\renewcommand{\epsilon}{\varepsilon}

\renewcommand{\leq}{\leqslant}
\renewcommand{\geq}{\geqslant}

\renewcommand{\P}{\mathbb{P}}
\newcommand{\E}{\mathbb{E}}
\newcommand{\R}{\mathbb{R}}
\newcommand{\C}{\mathbb{C}}
\newcommand{\N}{\mathbb{N}}
\newcommand{\Z}{\mathbb{Z}}

\newcommand{\abs}[1]{\left\lvert #1 \right\rvert}

\newcommand{\vertiii}[1]{{\left\vert\kern-0.25ex\left\vert\kern-0.25ex\left\vert #1 
    \right\vert\kern-0.25ex\right\vert\kern-0.25ex\right\vert}}
\newcommand{\ip}[1]{\left\langle #1 \right\rangle}
\newcommand{\diff}{\mathop{}\!\mathrm{d}}

\theoremstyle{plain} %plain, definition, remark
\newtheorem{thm}{Theorem}[section]
\newtheorem{mthm}[thm]{Metatheorem}
\newtheorem{lem}[thm]{Lemma}

\newtheorem{prop}[thm]{Proposition}

\newtheorem{defn}[thm]{Definition}
\newtheorem{example}[thm]{Example}
\newtheorem{rem}[thm]{Remark}

\makeatletter
\renewcommand{\section}{\@startsection
{section}%                   % the name
{1}%                         % the level
{0mm}%                       % the indent
{-2\baselineskip}%            % the before skip
{1\baselineskip}%          % the after skip
{\normalfont\large\scshape\centering}} % the style
\makeatother

\makeatletter
\renewcommand{\subsection}{\@startsection
{subsection}%                   % the name
{2}%                         % the level
{0mm}%                       % the indent
{-\baselineskip}%            % the before skip
{0 \baselineskip}%          % the after skip
{\normalfont\bf\itshape}} % the style
\makeatother

\makeatletter
\renewcommand{\subsubsection}{\@startsection
{subsubsection}%                   % the name
{3}%                         % the level
{0mm}%                       % the indent
{-\baselineskip}%            % the before skip
{0 \baselineskip}%          % the after skip
{\normalfont\bf\itshape}} % the style
\makeatother

\def\author#1{\par
    {\centering{\authorfont#1}\par\vspace*{0.05in}}
}

\def\titlefont{\fontsize{13}{15}\bfseries\boldmath\selectfont\centering{}}
\def\authorfont{\fontsize{13}{15}}

\let\affiliationfont\rhfont

\def\address#1{\par
    {\centering{\affiliationfont#1\par}}\par\vspace*{11pt}
}

\def\title#1{
    \thispagestyle{plain}
    \vspace*{-14pt}
    \vskip 79pt
    {\centering{\titlefont #1\par}}%
    \vskip 1em
}

\begin{document}
%%%%%%%%%%%%%%%%%%%%%%%%%%%%%%%%%%%%%%%%%%%%%%%%%%%%%%%%%%%%%%%%%%%%%%%%%%%	

~\vspace{-1.4cm}

\title{Landscape Complexity Beyond Invariance and the Elastic Manifold}

\vspace{1cm}
\noindent

\begin{minipage}[c]{0.33\textwidth}
 \author{G\'{e}rard Ben Arous}%
\address{Courant Institute\\%
   New York University\\%
   E-mail: benarous@cims.nyu.edu}%
 \end{minipage}%
 \begin{minipage}[c]{0.33\textwidth}
 \author{Paul Bourgade}%
\address{Courant Institute\\%
   New York University\\%
   E-mail: bourgade@cims.nyu.edu}%
 \end{minipage}%
\begin{minipage}[c]{0.33\textwidth}
 \author{Benjamin M\textsuperscript{c}Kenna}%
\address{Courant Institute\\%
   New York University\\%
   E-mail: mckenna@cims.nyu.edu}%
 \end{minipage}%

\begin{abstract}
This paper characterizes the annealed, topological complexity (both of total critical points and of local minima) of the elastic manifold. This classical model of a disordered elastic system captures point configurations with self-interactions in a random medium. We establish the simple-vs.-glassy phase diagram in the model parameters, with these phases separated by a physical boundary known as the Larkin mass, confirming formulas of Fyodorov and Le Doussal. 

One essential, dynamical, step of the proof also applies to  a general signal-to-noise model of soft spins in an anisotropic well, for which we prove a negative-second-moment threshold distinguishing positive from zero complexity. A universal near-critical behavior appears within this phase portrait, namely quadratic near-critical vanishing of the complexity of total critical points, and cubic near-critical vanishing of the complexity of local minima.

These two models serve as a paradigm of complexity calculations for Gaussian landscapes exhibiting few distributional symmetries, i.e. beyond the invariant setting. The two main inputs for the proof are determinant asymptotics for non-invariant random matrices from our companion paper \cite{BenBouMcK2021I}, and  the atypical convexity and integrability of the limiting variational problems.
\end{abstract}

{
	\hypersetup{linkcolor=black}
	\tableofcontents
}

%%%%%%%%%%%%%%%%%%%%%%%%%%%%%%%%%%%%%%%%%%%%%%%%%%%%%%%%%%%%
%%%%%%%%%%%%%%%%%%%%%%%%%%%%%%%%%%%%%%%%%%%%%%%%%%%%%%%%%%%%
%%%%%%%%
%%%%%%%%             Section: Introduction
%%%%%%%%
%%%%%%%%%%%%%%%%%%%%%%%%%%%%%%%%%%%%%%%%%%%%%%%%%%%%%%%%%%%%
%%%%%%%%%%%%%%%%%%%%%%%%%%%%%%%%%%%%%%%%%%%%%%%%%%%%%%%%%%%%

\section{Introduction}

%%%%%%%%%%%%%%%%%%%%%%%%%%%%%%%%%%%%%%%%%%%%%%%%%%%%%%%%%%%%
%%%%%%%%             Subsection: Complexity of the landscape of disordered elastic systems
%%%%%%%%%%%%%%%%%%%%%%%%%%%%%%%%%%%%%%%%%%%%%%%%%%%%%%%%%%%%

\subsection{Complexity of the landscape of disordered elastic systems.}\ \label{sec:1.1}
The elastic manifold is a paradigmatic representative of the class of disordered elastic systems. These are surfaces with rugged shapes resulting from a competition between random spatial impurities (preferring disordered configurations), on the one hand, and elastic self-interactions (preferring ordered configurations), on the other. The model is defined through its Hamiltonian \eqref{eqn:fld_hamiltonian}; for example, a one-dimensional such surface is a polymer; a $d$-dimensional such surface could describe the interface between ordered phases with opposite signs in a $(d+1)$-dimensional Ising model.  Among other motivations, the elastic manifold is interesting because it displays a (de)pinning phase transition,  which is a certain nonlinear response to a driving force: if one applies an external force to the surface at zero-temperature equilibrium, then the surface moves if and only if the force is above the depinning threshold.
The elastic manifold also has a long history as a testing ground for new approaches, for example for fixed $d$ by Fisher using functional renormalization group methods \cite{Fis1986}, and  in the high-dimensional limit by M{\'e}zard and Parisi using the replica method \cite{MezPar1991}. 

In the same diverging dimension regime, we study the energy landscape of this model,  through the expected number of configurations that locally minimize the Hamiltonian against small perturbations.  We also count the expected number of critical configurations.   Our main result,  Theorem \ref{thm:betterFLDcomplexity}, gives the phase diagram in the model parameters, and identifies the boundary between simple and glassy phases as a physical parameter known as the Larkin mass, which appears in the (de)pinning theory, confirming recent formulas by Fyodorov and Le Doussal \cite{FyoLeD2020}. 

The proof proceeds by dimension reduction and naturally leads to analyzing
a generalization of the zero-dimensional elastic manifold. The original zero-dimensional elastic manifold is
\begin{equation}
\label{eqn:fyodorov2004model}
    \mc{H}_N(x) = V_N(x) + \frac{\mu}{2}\|x\|^2,
\end{equation}
where $V_N : \R^N \to \R$ is an isotropic Gaussian field and $\mu > 0$. This has been studied by Fyodorov as a toy model of a disordered system; it admits a continuous phase transition between order for large $\mu$ and disorder for small $\mu$ \cite{Fyo2004}.  We replace the parabolic well confinement $\frac{\mu}{2}\|x\|^2$ with any positive definite quadratic form $\frac{1}{2}\ip{x,D_Nx}$, to see how different signal strengths in different directions affect the complexity; this defines the model of {\it  soft spins in an anisotropic well}.  Theorem \ref{thm:softspins_threshold} identifies a simple scalar parameter distinguishing between positive and zero complexity in high dimension, namely the negative second moment of the limiting empirical measure of $D_N$. We also find that the near-critical decay of complexity is described by universal exponents: quadratic for total critical points, and cubic for minima. 

Our work is part of the landscape complexity research program, which was initially developed for a variety of functions which are invariant under large classes of isometries (see Section \ref{sec:rotationallyinvariantmodels}). 
We address landscapes lacking this property, which we call ``non-invariant.'' The elastic manifold model is a proof of concept for our general approach,  which relies on
the Kac-Rice formula to reduce complexity to the calculation of the determinant of random matrices,  and on our  companion paper \cite{BenBouMcK2021I} for such determinant asymptotics for random matrix ensembles which are not invariant under orthogonal conjugacy.
This gives variational formulas for the annealed complexity such as Theorem \ref{thm:FLDcomplexity} for the elastic manifold.

Such variational problems associated to high dimensional Gaussian fields are not solvable in general (see e.g.  the companion paper \cite{McK2021} about bipartite spherical spin glasses).  However, for the elastic manifold,  a key convexity property inherited from the associated Matrix Dyson Equation (see Proposition \ref{prop:FLDconcavity}) reduces the dimension of the relevant variational formula,  mapping the problem to the complexity of the soft spins in an anisotropic well model for a specific $D_N$. We then find integrable dynamics to analyze the variational problems associated to the general soft  spins in an anisotropic well model,  and obtain the complexity thresholds mentioned above.

%%%%%%%%%%%%%%%%%%%%%%%%%%%%%%%%%%%%%%%%%%%%%%%%%%%%%%%%%%%%
%%%%%%%%             Subsection: Landscapes and determinants
%%%%%%%%%%%%%%%%%%%%%%%%%%%%%%%%%%%%%%%%%%%%%%%%%%%%%%%%%%%%

\subsection{Determinants and the Kac-Rice formula.}\
As mentioned  in the previous section, the Kac-Rice formula provides a bridge between random geometry and random matrix theory. If $f$ is a Gaussian field with enough regularity on a nice compact manifold $\mc{M}$, and if $\Crt_f(t,k)$ denotes the number of critical points of $f$ of index $k$ at which $f \leq t$, then this formula reads
\[
    \E[\Crt_f(t,k)] = \int_{\mc{M}} \E\left[ \left. \abs{\det(\nabla^2 f(\sigma)} \mathbf{1}\{f(\sigma) \leq t, i(\nabla^2f(\sigma)) = k\} \right| \nabla f(\sigma) = 0 \right] \phi_\sigma(0) \diff \sigma.
\]
Here $i(\cdot)$ is the index and $\phi_\sigma(0)$ is the density of $\nabla f(\sigma)$ at $0$. In the models of this paper, we will always take $\mc{M}$ to be the whole Euclidean space (with the necessary arguments to account for non-compactness). Thus the Kac-Rice formula transforms questions about critical points into questions about the (conditional) determinant of the random matrix $\nabla^2 f(\sigma)$. For an introduction to the Kac-Rice formula, we direct the reader to \cite{AdlTay2007, AzaWsc2009}. In a digestible special case, if $\Crt_f$ is the total number of critical points of $f$, then 
\begin{equation}
\label{eqn:kacrice}
    \E[\Crt_f] = \int_{\mc{M}} \E\left[ \left. \abs{\det(\nabla^2 f(\sigma)} \right| \nabla f(\sigma) = 0 \right] \phi_\sigma(0) \diff \sigma.
\end{equation}
In one dimension, this formula dates back to the 1940s \cite{Kac1943, Ric1944}. For many years it was used for small, fixed dimension in applications such as signal processing \cite{Ric1945} and oceanography \cite{Lon1957}. For more modern results in fixed dimension, we refer the reader to \cite{AzaDel2019}.

%%%%%%%%%%%%%%%%%%%%%%%%%%%%%%%%%%%%%%%%%%%%%%%%%%%%%%%%%%%%
%%%%%%%%             Subsection: Rotationally invariant models
%%%%%%%%%%%%%%%%%%%%%%%%%%%%%%%%%%%%%%%%%%%%%%%%%%%%%%%%%%%%

\subsection{Rotationally invariant models.}\
\label{sec:rotationallyinvariantmodels}
In a breakthrough insight, \cite{Fyo2004} used the Kac-Rice formula in \emph{diverging} dimension, to study \emph{asymptotic} counts of critical points via \emph{asymptotics} of random determinants. For example, if $f = f_N$ in the above discussion is defined on an $N$-dimensional manifold, one attempts to compute $\lim_{N \to \infty} \frac{1}{N}\log\E[\Crt_{f_N}]$. The papers \cite{Fyo2004} and \cite{FyoWil2007} studied isotropic Gaussian fields in radially symmetric confining potentials; the centered isotropic case without confining potentials (but in finite volume) was treated in \cite{BraDea2007}. Work has been done in the mathematics and physics literature on complexity for spherical $p$-spin models, starting with \cite{AufBenCer2013} (for pure models) and \cite{AufBen2013} (for mixtures). Similar techniques were used to understand the spiked-tensor model in \cite{BenMeiMonNic2019}. Intricate questions, such as the number of critical points with fixed index at given overlap from a minimum, are considered for pure $p$-spin models in \cite{Ros2020}. We also mention \cite{FanMeiMon2021} for an upper bound on the number of critical points of the TAP free energy of the Sherrington-Kirkpatrick model, and the recent works \cite{BasKeaMezNaj2020, BasKeaMezNaj2021} on neural networks, \cite{AufZen2020} on Gaussian fields with isotropic increments, \cite{BenFyoKho2020} on stable/unstable equilibria in systems of non-linear differential equations, and \cite{BelCerNakSch2021} on mixed spherical spin glasses with a deterministic external field. In most of these models, the conditioned Hessian is closely related to the Gaussian Orthogonal Ensemble (GOE), a consequence of distributional symmetries of the landscapes.

The above results handle the average number of critical points. It is another question entirely to prove concentration, i.e. to show that the \emph{average} (annealed) number of points is also \emph{typical} (quenched). Proving concentration typically involves intricate second-moment computations, which are also possible via the Kac-Rice formula, but which involve determinant asymptotics for a pair of (usually correlated) random matrices. To our knowledge this has only been carried out for $p$-spin models, both for pure models \cite{Sub2017, AufGol2020} and for certain mixtures which are close to pure \cite{BenSubZei2020}. The quenched asymptotics are not always expected to match the annealed ones; for more intricate questions in pure $p$-spin, physical computations based on the replica trick suggest a qualitative picture of this failure \cite{RosBenBirCam2019, RosBirCam2019}. 

%%%%%%%%%%%%%%%%%%%%%%%%%%%%%%%%%%%%%%%%%%%%%%%%%%%%%%%%%%%%
%%%%%%%%             Subsection: Non-invariant models
%%%%%%%%%%%%%%%%%%%%%%%%%%%%%%%%%%%%%%%%%%%%%%%%%%%%%%%%%%%%

\subsection{Non-invariant models.}\
In many models of interest, it happens that the law of the conditioned Hessian in \eqref{eqn:kacrice} does not depend on $\sigma$, and that it has long-range correlations induced by a fixed (not depending on $N$) number $m$ of independent Gaussian random variables. For example, this law might match that of $W_N + \xi \Id$, where $W_N$ is symmetric with independent Gaussian entries with a variance profile or large zero blocks, and $\xi \sim \mc{N}(0,\tfrac{1}{N})$ is independent of $W_N$; the resulting matrix has ``long-range correlations'' because the diagonal entries are all correlated with each other, and $m = 1$ because these correlations are induced by $\xi \in \R^1$. In these models, by integrating over this small number of variables last, the difficult term in the Kac-Rice formula \eqref{eqn:kacrice} takes the form
\begin{equation}
\label{eqn:overallproblem}
    \int_{\R^m} e^{-N\frac{\|u\|^2}{2}} \E[\abs{\det(H_N(u))}] \diff u
\end{equation}
for some Gaussian random matrices $H_N(u)$ which may be far from GOE. (In the example above, $H_N(u) = W_N + u\Id$.) 

The problem then reduces to the exponential asymptotics of \eqref{eqn:overallproblem}.
In the companion paper \cite{BenBouMcK2021I}, we establish two types of results about \eqref{eqn:overallproblem}. First, we show asymptotics for a single matrix of the form
\begin{equation}
\label{eqn:informal_detcon}
    \E[\abs{\det(H_N(u))}] = \exp\left(N \int_\R \log\abs{\lambda} \mu_N(u,\diff \lambda) + o(N) \right).
\end{equation}
Here the deterministic probability measures $\mu_N(u) = \mu_N(u,\cdot)$ come from the theory of the \emph{Matrix Dyson Equation} (MDE), developed in the random-matrix literature by Erd\H{o}s and co-authors in the last several years. Second, after this identification, \eqref{eqn:overallproblem} looks like a Laplace-type integral (with error terms), but the measures $\mu_N$ depend on $N$, meaning \eqref{eqn:overallproblem} may take the form $\int_{\R^m} e^{Nf_N(u)} \diff u$ instead of the more-desirable $\int_{\R^m} e^{Nf(u)} \diff u$. In \cite{BenBouMcK2021I}, we show that -- \emph{assuming} the limits $\mu_N(u) \to \mu_\infty(u)$ exist -- the Laplace method can be carried out on \eqref{eqn:overallproblem}. 

In this paper we discuss how to identify the limits $\mu_N(u) \to \mu_\infty(u)$ for the elastic manifold and soft spins in an anisotropic well (a third model is treated in the companion paper \cite{McK2021}). This is model-dependent, although we identify some common techniques. This leads to the following informal statement:

\begin{mthm}
\label{thm:metathm}
Let $\mc{M}_N$ be a nice sequence of $N$-dimensional manifolds, and let $f_N : \mc{M}_N \to \R$ be a sequence of Gaussian random landscapes with the properties discussed above (namely, the law of the conditioned Hessian is independent of the basepoint on $\mc{M}_N$, and long-range correlations are induced by $m$ independent variables). If the limiting empirical measures $\mu_\infty(u)$ can be identified and some regularity established in $u$ (and we present models where this is possible), then
\begin{equation}
\label{eqn:metathm_variational}
    \lim_{N \to \infty} \frac{1}{N} \E[\Crt_{f_N}] = \sup_{u \in \R^m} \left\{ \int_\R \log\abs{\lambda} \mu_\infty(u,\diff \lambda) - \frac{\|u\|^2}{2}\right\} + \text{simpler non-variational term}.
\end{equation}
The non-variational term comes from the density of the gradient in the Kac-Rice formula: precisely, it is equal to $\lim_{N \to \infty} \frac{1}{N} \log \int_{\mc{M}_N} \phi_\sigma(0) \diff \sigma$, which is typically easy to calculate.
\end{mthm}

We also wish to count local minima, for which the analogue of \eqref{eqn:overallproblem} is
\[
    \int_{\mf{D}} e^{-N\frac{\|u\|^2}{2}} \E[\abs{\det(H_N(u))} \mathbf{1}_{H_N(u) \geq 0}] \diff u.
\]
If we define the set
\[
    \mc{G} = \{u \in \R^m : \mu_\infty(u)((-\infty,0)) = 0\}
\]
of good $u$ values for which $\{H_N(u) \geq 0\}$ is a likely event, then the upshot is that at exponential scale we have
\begin{equation}
\label{eqn:overviewminima}
    \E[\abs{H_N(u)}\mathbf{1}_{H_N(u) \geq 0}] \approx \begin{cases} \E[\abs{H_N(u)}] & \text{if } u \in \mc{G}, \\ 0 & \text{otherwise.} \end{cases}
\end{equation}
(All the matrices $H_N(u)$ we encounter have asymptotically no outliers; otherwise, large-deviations estimates for edge eigenvalues would impact the final result.) This gives an analogue of Metatheorem \ref{thm:metathm} for the complexity of local minima, where the variational problem is restricted to a supremum over $u \in \mc{G}$ instead of $u \in \R^m$. Again, the argument was presented in \cite{BenBouMcK2021I} assuming the existence of limits $\mu_N(u) \to \mu_\infty(u)$; in this paper we verify this assumption.

The goal of this paper is to carry out this program for the elastic manifold and the anisotropic soft spins model, yielding precise versions of Metatheorem \ref{thm:metathm} and its analogue for minima. In fact, for these particular models the variational problem in \eqref{eqn:metathm_variational} turns out to be integrable, as mentioned at the end of Section \ref{sec:1.1}: By introducing a dynamic version of the optimization \eqref{eqn:metathm_variational}, we can distinguish regimes of positive and zero complexity. In addition, we can study near-critical behavior at this phase transition, showing that complexity of total critical points tends to zero quadratically, whereas complexity of local minima tends to zero cubically. These critical exponents were already known for certain models \cite{Fyo2004, FyoWil2007}; we show their universality by extending substantially the class of models exhibiting these quadratic and cubic transitions.

We state our main results in Section \ref{sec:complexity}. Section \ref{sec:stability} provides techniques that will be shared across models, showing how the (well-established) stability theory of the MDE allows one to replace $\mu_N(u)$ by $\mu_\infty(u)$ as discussed above, if one has a candidate $\mu_\infty$. In the remaining sections, we propose candidates for $\mu_\infty$ and carry out this program for each of our models in turn. In the Appendix, we prove a result in free probability  necessary to identify near-critical complexity of our models, and possibly of independent interest: The free convolution of any (compactly supported) measure with the semicircle law decays at least as quickly as a square root at its extremal edges. 

%%%%%%%%%%%%%%%%%%%%%%%%%%%%%%%%%%%%%%%%%%%%%%%%%%%%%%%%%%%%
%%%%%%%%             Subsection: Notations
%%%%%%%%%%%%%%%%%%%%%%%%%%%%%%%%%%%%%%%%%%%%%%%%%%%%%%%%%%%%

\subsection*{Notations.}\
We write $\|\cdot\|$ for the operator norm on elements of $\C^{N \times N}$ induced by Euclidean distance on $\C^N$, and if $\mc{S} : \C^{N \times N} \to \C^{N \times N}$, we write $\|\mc{S}\|$ for the operator norm induced by $\|\cdot\|$. We let
\[
    \|f\|_{\text{Lip}} = \sup_{x \neq y} \abs{\frac{f(x)-f(y)}{x-y}}
\]
for test functions $f : \R \to \R$, and consider the following three distances on probability measures on $\R$ (called bounded-Lipschitz, Wasserstein-$1$, and L\'evy, respectively):
\begin{align*}
    d_{\textup{BL}}(\mu,\nu) &= \sup\left\{\abs{\int_\R f \diff (\mu - \nu)} : \|f\|_{\text{Lip}} + \|f\|_{L^\infty} \leq 1\right\}, \\
    {\rm W}_1(\mu,\nu) &= \sup\left\{\abs{\int_\R f \diff (\mu - \nu)} : \|f\|_{\text{Lip}} \leq 1\right\}, \\
    d_{\textup{L}}(\mu,\nu) &= \inf\{\epsilon > 0 : \mu((-\infty,x-\epsilon]) - \epsilon \leq \nu((-\infty,x]) \leq \mu((-\infty,x+\epsilon])+\epsilon \text{ for all } x\}.
\end{align*}
We will need the semicircle law of variance $t$, which we write as
\[
    \rho_{\text{sc},t}(\diff x) = \frac{\sqrt{4t-x^2}}{2\pi t} \, \mathbf{1}_{x \in [-2\sqrt{t},2\sqrt{t}]} \diff x,
\]
as well as the abbreviation $
    \rho_{\text{sc}} = \rho_{\text{sc},1}$
for the usual semicircle law supported in $[-2,2]$. We write $\mathtt{l}(\mu)$ for the left edge (respectively, $\mathtt{r}(\mu)$ for the right edge) of a compactly supported measure $\mu$. For an $N \times N$ Hermitian matrix $M$, we write $\lambda_{\min{}}(M) = \lambda_1(M) \leq \cdots \leq \lambda_N(M) = \lambda_{\max{}}(M)$ for its eigenvalues and 
\[
    \hat{\mu}_M = \frac{1}{N}\sum_{i=1}^N \delta_{\lambda_i(M)}
\]
for its empirical measure. We write $\odot$ for the entrywise (i.e., Hadamard) product of matrices, and $\boxplus$ for the free (additive) convolution of probability measures. Given a matrix $T$, we write $\diag(T)$ for the diagonal matrix of the same size obtained by setting all off-diagonal entries to zero. In equations, we sometimes identify diagonal matrices with vectors of the same size. We write $B_R$ for the ball of radius $R$ about zero in the relevant Euclidean space. We use $(\cdot)^T$ for the matrix transpose, which should be distinguished both from $(\cdot)^\ast$ for the matrix conjugate transpose, and from $\Tr(\cdot)$ for the matrix trace.

Unless stated otherwise, $z$ will always be a complex number in the upper half-plane $\mathbb{H} = \{z \in \C : \im(z) > 0\}$, and we always write its real and imaginary parts as $z = E+\ii\eta$. 

%%%%%%%%%%%%%%%%%%%%%%%%%%%%%%%%%%%%%%%%%%%%%%%%%%%%%%%%%%%%
%%%%%%%%             Subsection: Acknowledgements
%%%%%%%%%%%%%%%%%%%%%%%%%%%%%%%%%%%%%%%%%%%%%%%%%%%%%%%%%%%%

\subsection*{Acknowledgements.}\
We wish to thank L\'{a}szl\'{o} Erd\H{o}s and Torben Kr\"{u}ger for many helpful discussions about the MDE for block random matrices, in particular for communicating arguments that became the proofs of Lemmas \ref{lem:kroneckermde_Ldbound} and \ref{lem:erdoskruger}. We are also grateful to Yan Fyodorov and Pierre Le Doussal for bringing the elastic-manifold problem to our attention, and for helpful discussions, for which we also thank Krishnan Mody, Eliran Subag, and Ofer Zeitouni. GBA acknowledges support by the Simons Foundation collaboration Cracking the Glass Problem, PB was supported by NSF grant DMS-1812114 and a Poincar\'e chair, and BM was supported by NSF grant DMS-1812114.

%%%%%%%%%%%%%%%%%%%%%%%%%%%%%%%%%%%%%%%%%%%%%%%%%%%%%%%%%%%%
%%%%%%%%%%%%%%%%%%%%%%%%%%%%%%%%%%%%%%%%%%%%%%%%%%%%%%%%%%%%
%%%%%%%%
%%%%%%%%             Section: Main results
%%%%%%%%
%%%%%%%%%%%%%%%%%%%%%%%%%%%%%%%%%%%%%%%%%%%%%%%%%%%%%%%%%%%%
%%%%%%%%%%%%%%%%%%%%%%%%%%%%%%%%%%%%%%%%%%%%%%%%%%%%%%%%%%%%

\section{Main results}
\label{sec:complexity}

%%%%%%%%%%%%%%%%%%%%%%%%%%%%%%%%%%%%%%%%%%%%%%%%%%%%%%%%%%%%
%%%%%%%%             Subsection: Elastic manifold
%%%%%%%%%%%%%%%%%%%%%%%%%%%%%%%%%%%%%%%%%%%%%%%%%%%%%%%%%%%%

\subsection{Elastic manifold.}\
\label{subsec:elasticmanifold}
Fix positive integers $L$ (``length'') and $d$ (``internal dimension''), positive numbers $\mu_0$ (``mass'') and $t_0$ (``interaction strength''), and write $\Omega$ for the lattice $\llbracket 1, L \rrbracket^d \subset \Z^d$, understood periodically. Let $V_N$ be a centered Gaussian field on $\R^N \times \Omega$ with 
\[
    \E[V_N(y_1,x_1)V_N(y_2,x_2)] = N B\left(\frac{\|y_1-y_2\|^2}{N}\right) \delta_{x_1,x_2},
\]
for some function $B : \R_+ \to \R_+$ called the correlator. Schoenberg characterized all possible such correlators \cite[Theorem 2]{Sch1938} (see also \cite{Yag1957}); $B$ must have the representation
\begin{equation}
\label{eqn:schoenberg}
    B(x) = c_0 + \int_0^\infty \exp(-t^2x) \nu(\diff t)
\end{equation}
for some $c_0 \geq 0$ and some finite non-negative measure $\nu$ on $(0,\infty)$. In particular $B$ is infinitely differentiable and non-increasing on $(0,\infty)$. We assume that $B$ is also four times differentiable at zero, which implies via Kolmogorov's criterion that each $V_N(\cdot,x)$ is almost surely twice differentiable. We will also assume
\[
    0 < |B^{(i)}(0)| \qquad \text{for } i = 0, 1, 2,
\]
which should be interpreted as a non-degeneracy condition on the field ($i = 0$), its gradient ($i = 1$), and its Hessian ($i = 2$). This is a very mild assumption; indeed it holds by dominated convergence as soon as the measure $\nu$ in \eqref{eqn:schoenberg} has a finite fourth moment and is not the zero measure.

To each deterministic function $\mathbf{u} : \Omega \to \R^N$ (``point configuration,'' but sometimes ``manifold'' after the continuous analogue) associate the random Hamiltonian
\begin{equation}
\label{eqn:fld_hamiltonian}
    \mc{H}[\mathbf{u}] = \sum_{x,y \in \Omega} (\mu_0 \Id - t_0\Delta)_{xy} \ip{\mathbf{u}(x), \mathbf{u}(y)} + \sum_{x \in \Omega} V_N(\mathbf{u}(x),x).
\end{equation}
Here $\Delta \in \R^{L^d \times L^d}$ is the (periodic) lattice Laplacian on $\Omega$, so the $(x,y)$ entry of $\mu_0 \Id - t_0 \Delta$ is given by
\[
	(\mu_0 \Id - t_0\Delta)_{xy} = \mu_0\delta_{x = y} - t_0(\delta_{x \sim y} - 2d\delta_{x = y}),
\]
where $x \sim y$ means that $x$ and $y$ are lattice neighbors. (Following \cite{FyoLeD2020}, our Laplacian is a negative sign off from the typical mathematical convention.)

Notice that the different energies compete: If the disorder $V_N$ vanished in \eqref{eqn:fld_hamiltonian}, then since $\mu_0\Id$ and $-t_0\Delta$ are both positive semidefinite, the ground-state configuration would be the flat one $\mathbf{u} \equiv 0$. On the other hand, the disorder $V_N$ prefers certain random configurations; the interaction $-t_0\Delta$ prefers to keep these configurations from becoming too jagged; and the confinement $\mu_0$ prefers to keep them close to the origin. See Figure \ref{fig:elasticmanifold} for a graphical interpretation.

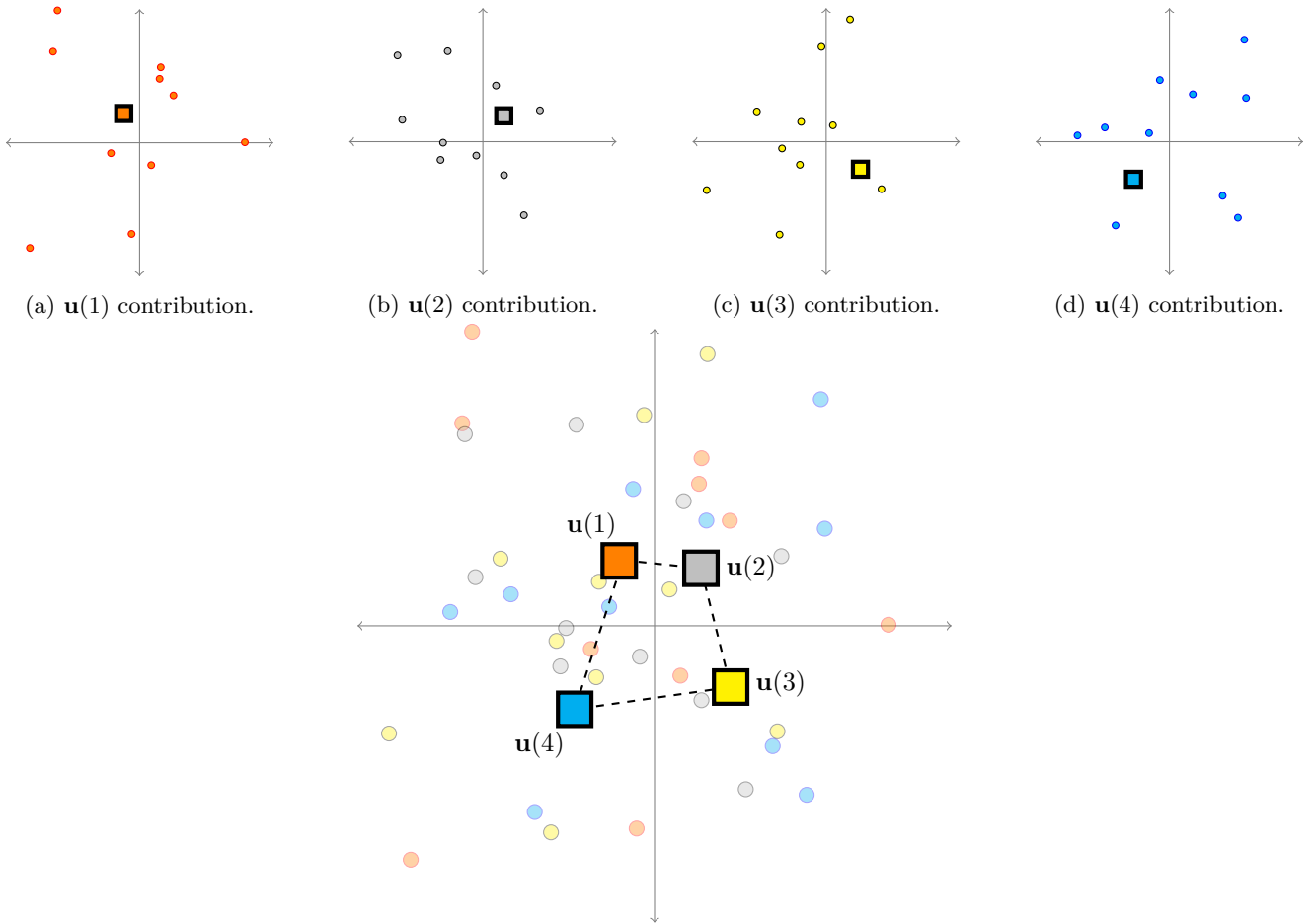
\begin{figure}[h!]
\centering
    \begin{subfigure}{0.22\textwidth}
        \centering
        \begin{tikzpicture}[scale=0.45]
            \draw[<->,gray] (-4,0) -- (4,0);
            \draw[<->,gray] (0,-4) -- (0,4);
            \draw [red, fill=orange] (-0.240559, -2.73191) circle [radius=0.1];
            \draw [red, fill=orange] (-3.28216, -3.15222) circle [radius=0.1];
            \draw [red, fill=orange] (-2.45617, 3.96398) circle [radius=0.1];
            \draw [red, fill=orange] (3.15055, 0.0142884) circle [radius=0.1];
            \draw [red, fill=orange] (-0.857254, -0.312903) circle [radius=0.1];
            \draw [red, fill=orange] (0.633215, 2.25918) circle [radius=0.1];
            \draw [red, fill=orange] (1.01345, 1.41645) circle [radius=0.1];
            \draw [red, fill=orange] (-2.58943, 2.72889) circle [radius=0.1];
            \draw [red, fill=orange] (0.598993, 1.91382) circle [radius=0.1];
            \draw [red, fill=orange] (0.348252, -0.6715046) circle [radius=0.1];
            \draw [black, ultra thick, fill=orange] (-0.7,1.1) rectangle (-0.25,0.65);
        \end{tikzpicture}
        \caption{$\mathbf{u}(1)$ contribution.}
    \end{subfigure}
    \hfill
    \begin{subfigure}{0.22\textwidth}
        \centering
        \begin{tikzpicture}[scale=0.45]
            \draw[<->,gray] (-4,0) -- (4,0);
            \draw[<->,gray] (0,-4) -- (0,4);
            \draw [black, fill=lightgray] (0.392661, 1.67969) circle [radius=0.1];
            \draw [black, fill=lightgray] (-1.19199, -0.0299144) circle [radius=0.1];
            \draw [black, fill=lightgray] (1.22834, -2.20258) circle [radius=0.1];
            \draw [black, fill=lightgray] (0.632006, -1.00243) circle [radius=0.1];
            \draw [black, fill=lightgray] (-2.41178, 0.655762) circle [radius=0.1];
            \draw [black, fill=lightgray] (-1.05262, 2.70988) circle [radius=0.1];
            \draw [black, fill=lightgray] (1.70791, 0.938488) circle [radius=0.1];
            \draw [black, fill=lightgray] (-2.55408, 2.58293) circle [radius=0.1];
            \draw [black, fill=lightgray] (-0.196709, -0.415952) circle [radius=0.1];
            \draw [black, fill=lightgray] (-1.26783, -0.546373) circle [radius=0.1];
            \draw [black, ultra thick, fill=lightgray] (0.4,1) rectangle (0.85,0.55);
        \end{tikzpicture}
        \caption{$\mathbf{u}(2)$ contribution.}
    \end{subfigure}
    \hfill
    \begin{subfigure}{0.22\textwidth}
        \centering
        \begin{tikzpicture}[scale=0.45]
            \draw[<->,gray] (-4,0) -- (4,0);
            \draw[<->,gray] (0,-4) -- (0,4);
            \draw [black, fill=yellow] (0.714484, 3.6619) circle [radius=0.1];
            \draw [black, fill=yellow] (-3.57512, -1.45147) circle [radius=0.1];
            \draw [black, fill=yellow] (-1.39413, -2.78396) circle [radius=0.1];
            \draw [black, fill=yellow] (0.202412, 0.489715) circle [radius=0.1];
            \draw [black, fill=yellow] (-0.749156, 0.593857) circle [radius=0.1];
            \draw [black, fill=yellow] (1.6556, -1.42229) circle [radius=0.1];
            \draw [black, fill=yellow] (-2.07423, 0.905042) circle [radius=0.1];
            \draw [black, fill=yellow] (-0.140773, 2.84032) circle [radius=0.1];
            \draw [black, fill=yellow] (-1.32014, -0.20337) circle [radius=0.1];
            \draw [black, fill=yellow] (-0.785649, -0.692477) circle [radius=0.1];
            \draw [black, ultra thick, fill = yellow] (0.8, -0.6) rectangle (1.25,-1.05);
        \end{tikzpicture}
        \caption{$\mathbf{u}(3)$ contribution.}
    \end{subfigure}
    \hfill
    \begin{subfigure}{0.22\textwidth}
        \centering 
        \begin{tikzpicture}[scale=0.45]
            \draw[<->,gray] (-4,0) -- (4,0);
            \draw[<->,gray] (0,-4) -- (0,4);
            \draw [blue, fill=cyan] (2.29236,1.30946) circle [radius=0.1];
            \draw [blue, fill=cyan] (-1.61325,-2.50894) circle [radius=0.1];
            \draw [blue, fill=cyan] (-0.288295,1.84407) circle [radius=0.1];
            \draw [blue, fill=cyan] (-0.610573, 0.257208) circle [radius=0.1];
            \draw [blue, fill=cyan] (2.23916, 3.05355) circle [radius=0.1];
            \draw [blue, fill=cyan] (0.698387, 1.4187) circle [radius=0.1];
            \draw [blue, fill=cyan] (1.59153, -1.62097) circle [radius=0.1];
            \draw [blue, fill=cyan] (-1.93492, 0.424784) circle [radius=0.1];
            \draw [blue, fill=cyan] (-2.75064, 0.18661) circle [radius=0.1];
            \draw [blue, fill=cyan] (2.04872, -2.2776) circle [radius=0.1];
            \draw [black, ultra thick, fill=cyan] (-1.3,-0.9) rectangle (-0.85, -1.35);
        \end{tikzpicture}
        \caption{$\mathbf{u}(4)$ contribution.}
    \end{subfigure}
    \\ \vfill
    \centering
    \begin{subfigure}{0.8\textwidth}
        \centering
        \begin{tikzpicture}
            \draw[<->,gray] (-4,0) -- (4,0);
            \draw[<->,gray] (0,-4) -- (0,4);
            \draw [blue, fill=cyan, opacity = 0.35] (2.29236,1.30946) circle [radius=0.1];
            \draw [blue, fill=cyan, opacity = 0.35] (-1.61325,-2.50894) circle [radius=0.1];
            \draw [blue, fill=cyan, opacity = 0.35] (-0.288295,1.84407) circle [radius=0.1];
            \draw [blue, fill=cyan, opacity = 0.35] (-0.610573, 0.257208) circle [radius=0.1];
            \draw [blue, fill=cyan, opacity = 0.35] (2.23916, 3.05355) circle [radius=0.1];
            \draw [blue, fill=cyan, opacity = 0.35] (0.698387, 1.4187) circle [radius=0.1];
            \draw [blue, fill=cyan, opacity = 0.35] (1.59153, -1.62097) circle [radius=0.1];
            \draw [blue, fill=cyan, opacity = 0.35] (-1.93492, 0.424784) circle [radius=0.1];
            \draw [blue, fill=cyan, opacity = 0.35] (-2.75064, 0.18661) circle [radius=0.1];
            \draw [blue, fill=cyan, opacity = 0.35] (2.04872, -2.2776) circle [radius=0.1];
    
            \draw [red, fill=orange, opacity = 0.35] (-0.240559, -2.73191) circle [radius=0.1];
            \draw [red, fill=orange, opacity = 0.35] (-3.28216, -3.15222) circle [radius=0.1];
            \draw [red, fill=orange, opacity = 0.35] (-2.45617, 3.96398) circle [radius=0.1];
            \draw [red, fill=orange, opacity = 0.35] (3.15055, 0.0142884) circle [radius=0.1];
            \draw [red, fill=orange, opacity = 0.35] (-0.857254, -0.312903) circle [radius=0.1];
            \draw [red, fill=orange, opacity = 0.35] (0.633215, 2.25918) circle [radius=0.1];
            \draw [red, fill=orange, opacity = 0.35] (1.01345, 1.41645) circle [radius=0.1];
            \draw [red, fill=orange, opacity = 0.35] (-2.58943, 2.72889) circle [radius=0.1];
            \draw [red, fill=orange, opacity = 0.35] (0.598993, 1.91382) circle [radius=0.1];
            \draw [red, fill=orange, opacity = 0.35] (0.348252, -0.6715046) circle [radius=0.1];
    
            \draw [black, fill=yellow, opacity = 0.35] (0.714484, 3.6619) circle [radius=0.1];
            \draw [black, fill=yellow, opacity = 0.35] (-3.57512, -1.45147) circle [radius=0.1];
            \draw [black, fill=yellow, opacity = 0.35] (-1.39413, -2.78396) circle [radius=0.1];
            \draw [black, fill=yellow, opacity = 0.35] (0.202412, 0.489715) circle [radius=0.1];
            \draw [black, fill=yellow, opacity = 0.35] (-0.749156, 0.593857) circle [radius=0.1];
            \draw [black, fill=yellow, opacity = 0.35] (1.6556, -1.42229) circle [radius=0.1];
            \draw [black, fill=yellow, opacity = 0.35] (-2.07423, 0.905042) circle [radius=0.1];
            \draw [black, fill=yellow, opacity = 0.35] (-0.140773, 2.84032) circle [radius=0.1];
            \draw [black, fill=yellow, opacity = 0.35] (-1.32014, -0.20337) circle [radius=0.1];
            \draw [black, fill=yellow, opacity = 0.35] (-0.785649, -0.692477) circle [radius=0.1];
            
            \draw [black, fill=lightgray, opacity = 0.35] (0.392661, 1.67969) circle [radius=0.1];
            \draw [black, fill=lightgray, opacity = 0.35] (-1.19199, -0.0299144) circle [radius=0.1];
            \draw [black, fill=lightgray, opacity = 0.35] (1.22834, -2.20258) circle [radius=0.1];
            \draw [black, fill=lightgray, opacity = 0.35] (0.632006, -1.00243) circle [radius=0.1];
            \draw [black, fill=lightgray, opacity = 0.35] (-2.41178, 0.655762) circle [radius=0.1];
            \draw [black, fill=lightgray, opacity = 0.35] (-1.05262, 2.70988) circle [radius=0.1];
            \draw [black, fill=lightgray, opacity = 0.35] (1.70791, 0.938488) circle [radius=0.1];
            \draw [black, fill=lightgray, opacity = 0.35] (-2.55408, 2.58293) circle [radius=0.1];
            \draw [black, fill=lightgray, opacity = 0.35] (-0.196709, -0.415952) circle [radius=0.1];
            \draw [black, fill=lightgray, opacity = 0.35] (-1.26783, -0.546373) circle [radius=0.1];

            \draw [black, dashed, thick] (-0.425,0.885) -- (0.625, 0.775) -- (1.025,-0.825) -- (-1.075,-1.125) -- (-0.425,0.885);
            \draw [black, ultra thick, fill=orange] (-0.7,1.1) rectangle (-0.25,0.65);
            \draw [black, ultra thick, fill=lightgray] (0.4,1) rectangle (0.85,0.55);
            \draw [black, ultra thick, fill = yellow] (0.8, -0.6) rectangle (1.25,-1.05);
            \draw [black, ultra thick, fill=cyan] (-1.3,-0.9) rectangle (-0.85, -1.35);
            \node [black] at (-0.85,1.35) {$\mathbf{u}(1)$};
            \node [black] at (1.3,0.775) {$\mathbf{u}(2)$};
            \node [black] at (1.7,-0.775) {$\mathbf{u}(3)$};
            \node [black] at (-1.55,-1.6) {$\mathbf{u}(4)$};
        \end{tikzpicture}
        \caption{The elastic contribution to the configuration energy is small, because the sum $\ip{\mathbf{u}(1), \mathbf{u}(2)} + \ip{\mathbf{u}(2), \mathbf{u}(3)} + \ip{\mathbf{u}(3), \mathbf{u}(4)} + \ip{\mathbf{u}(4), \mathbf{u}(1)}$ of inner products between points whose indices are nearest neighbors in the underlying lattice is small.}
    \end{subfigure}
\caption{Informal schematic of one low-energy elastic manifold configuration when $d = 1$, $L = 4$, and $N = 2$. The four manifold points $\mathbf{u}(1)$, $\mathbf{u}(2)$, $\mathbf{u}(3)$, and $\mathbf{u}(4)$ are indicated by squares. In the top four subfigures, which indicate the contributions to the total energy made by each manifold point on its own, each manifold point sees its own (independent) Gaussian environment and tries to avoid points of high energy cost (represented by circles of the same color) while staying close to the origin. In the bottom subfigure, these environments are overlaid, showing that the points have achieved their separate goals while also keeping their lattice-neighbor inner products small. The inner products $\ip{\mathbf{u}(1),\mathbf{u}(3)}$ and $\ip{\mathbf{u}(2),\mathbf{u}(4)}$ do not contribute, because $\{1,3\}$ and $\{2,4\}$ are not lattice neighbors. Perhaps this configuration is a local minimum, meaning the energy $\mc{H}[\mathbf{u}]$ increases if we slightly perturb any of the images $\mathbf{u}(i)$. We are trying to count such minima (and total critical points) in the $N \to +\infty$ limit, when these four points are immersed not in the plane but in a high-dimensional space. This figure is inspired by \cite[Figure 2]{Gia2009}.}
\label{fig:elasticmanifold}
\end{figure}

%%%%%%%%%%%%%%%%%%%%%%%%%%%%%%%%%%%%%%%%%%%%%%%%%%%%%%%%%%%%
%%%%%%%%             Subsubsection: History
%%%%%%%%%%%%%%%%%%%%%%%%%%%%%%%%%%%%%%%%%%%%%%%%%%%%%%%%%%%%

\subsubsection*{History.}\
Hamiltonians of this flavor have been used to model a wide variety of problems featuring surfaces with self-interactions in disordered media. For example, when $d = 1$, the model is a polymer, related to the KPZ universality class; when $N = d+1$, the model is an interface, such as that between regions of opposite magnetization in a ferromagnet. We direct readers to \cite{Gia2009} and \cite{GiaLeD1997} for a review of disordered elastic media in general and to \cite{FyoLeD2020A} for a review of this specific Hamiltonian, which we summarize briefly here.

Two phenomena are of primary interest: the \emph{depinning threshold $f_c$} and the \emph{wandering (or roughness) exponent $\zeta$}. The former refers to the manifold's nonlinear response to an applied force $f$, a consequence of the impurities in the potential $V$: at zero temperature, it moves from its preferred position only if the force is above the depinning threshold $f > f_c = f_c(L,d,t_0,N)$, whereas if $f \leq f_c$ it does not move at all and is said to be \emph{pinned}. (Depinning is typically discussed in the massless limit $\mu_0 \downarrow 0$, but restricting the manifold points to lie in a finite box. At positive temperature, the manifold can move when $f < f_c$, but the movement is typically slow and is called \emph{creep}; the movement above $f_c$ is faster.) Depinning is related to complexity: Adding a force changes the Hamiltonian, and the landscape is supposed to simplify as $f$ increases; then $f_c$ can be defined as the smallest $f$ for which the resulting (quenched) complexity vanishes. We do not study this connection further, but refer readers to a discussion in \cite{FyoLeD2020}.

The wandering exponent $\zeta$, which depends on $d$ and $N$, is defined by
\[
	\E[(\mathbf{u}_0(x) - \mathbf{u}_0(y))^2] \sim \|x-y\|^{2\zeta}
\]
where $\mathbf{u}_0$ is the ground state. It is generally believed that $\zeta = 0$, i.e. that the manifold is flat, for $d \geq 4$. Larkin proposed a simplification of the Hamiltonian \eqref{eqn:fld_hamiltonian}, replacing the terms $V_N(\mathbf{u}(x),x)$ with their linearizations $V_N(0,x) + \partial_y V_N(0,x)|_{y = 0} \mathbf{u}(x)$. This so-called \emph{Larkin model} is solvable and gives $\zeta = \left(\frac{4-d}{2}\right)_+$; note also that the Larkin model is quadratic in $\mathbf{u}$, hence only has one local minimum, i.e., is necessarily zero-complexity. Physicists believe that the Larkin model is a good approximation for the elastic manifold when $L$ is below the \emph{Larkin length} $L_c$, with $L_c \sim (B''(0))^{-1/(4-d)}$ for weak disorder. Above the Larkin length the approximation is supposed to break down, and describing the physics of the elastic manifold (in particular finding $\zeta$) is more challenging. This regime inspired early technical developments of Fisher in functional renormalization group methods \cite{Fis1986} and of M{\'e}zard and Parisi in the replica method \cite{MezPar1991}; the latter paper suggested that the system exhibits zero-temperature replica symmetry breaking for small $\mu_0$ in the $N \to +\infty$ limit. (This is the same limit we will consider, although of course one is ultimately interested in finite-$N$ results.) Increasing the ``mass'' $\mu_0$ has the effect of simplifying the landscape, and for $\mu_0$ larger than a \emph{Larkin mass} $\mu_c$ (related to the Larkin length $L_c$), the system is believed to be replica symmetric. In fact the Larkin mass is central to our results; we are making rigorous a result of Fyodorov and Le Doussal suggesting that, for all other parameters fixed, $\mu_c$ is precisely the boundary between zero complexity (for $\mu_0 \geq 2\sqrt{B''(0)}\mu_c$) and positive complexity (for $\mu_0 \leq 2\sqrt{B''(0)}\mu_c$). The same $\mu_c$ serves as the boundary both for total critical points and for local minima.

There are some previous complexity results for special cases. When $d = 0$, the system is interpreted by convention as being a single point, i.e., it reduces to the Hamiltonian \eqref{eqn:fyodorov2004model}. Fyodorov computed the complexity of \eqref{eqn:fyodorov2004model} and found a continuous phase transition in $\mu$: For $\mu \geq \mu_c$, the annealed complexity (of the total number of critical points) is zero and the landscape is ``simple,'' but for $\mu < \mu_c$ the annealed complexity is positive and the landscape is ``complex'' or ``glassy'' \cite{Fyo2004}. Later, Fyodorov and Williams showed that this phase transition matches that of replica-symmetry/replica-symmetry-breaking at zero temperature \cite{FyoWil2007}, interpreting replica-symmetry-breaking as ``a replica-symmetric computation of the free energy becomes unstable in the zero-temperature limit.'' For more discussion of the $d = 0$ case, see Section \ref{subsec:softspins} below. When $d = 1$, the model is an elastic line, with complexity studied in the case of $N = 1$ and $L \to +\infty$ in \cite{FyoLeDRosTex2018}.

%%%%%%%%%%%%%%%%%%%%%%%%%%%%%%%%%%%%%%%%%%%%%%%%%%%%%%%%%%%%
%%%%%%%%             Subsubsection: Results
%%%%%%%%%%%%%%%%%%%%%%%%%%%%%%%%%%%%%%%%%%%%%%%%%%%%%%%%%%%%

\subsubsection*{Results.}\
Let $\mc{N}_{\text{tot}}$ be the random number of stationary points of the Hamiltonian, i.e., of functions $\mathbf{u} : \Omega \to \R^N$ such that $\partial_{\mathbf{u}_i(x)} \mc{H}[\mathbf{u}] = 0$ for every $x \in \Omega$ and every $i = 1, \ldots, N$. Let $\mc{N}_{\text{st}}$ be the number of local minima.

\begin{defn}
For any $\mu_0, t_0, b > 0$, define
\begin{align}
\label{eqn:FLD_onedimensional}
\begin{split}
	&\Sigma(\mu_0, t_0, b) = \Sigma(\mu_0, t_0, b, L, d) \\
	&= - \frac{1}{L^d} \log(\det(\mu_0 \Id_{L^d \times L^d} - t_0 \Delta)) + \sup_{u \in \R} \left\{ \int_\R \log\abs{\lambda - u} (\rho_{\text{sc},b} \boxplus \hat{\mu}_{-t_0\Delta + \mu_0\Id})(\lambda) \diff \lambda - \frac{u^2}{2b} \right\}, \\
	&\Sigma_{\textup{st}}(\mu_0, t_0, b)  = \Sigma_{\textup{st}}(\mu_0, t_0, b, L, d) \\
	&= - \frac{1}{L^d} \log(\det(\mu_0 \Id_{L^d \times L^d} - t_0 \Delta)) + \sup_{u \leq \mathtt{l}(\rho_{\text{sc},b} \boxplus \hat{\mu}_{-t_0\Delta + \mu_0\Id})}  \left\{ \int_\R \log\abs{\lambda - u} (\rho_{\text{sc},b} \boxplus \hat{\mu}_{-t_0\Delta + \mu_0\Id})(\lambda) \diff \lambda - \frac{u^2}{2b} \right\}.
\end{split}
\end{align}
\end{defn}

\begin{thm}
\label{thm:variationalFLDcomplexity}
We have 
\begin{align}
\label{eqn:variationalFLDcomplexity}
\begin{split}
	\lim_{N \to \infty} \frac{1}{NL^d} \log\E[\mc{N}_{\textup{tot}}] &= \Sigma(\mu_0,t_0,4B''(0)), \\
	\lim_{N \to \infty} \frac{1}{NL^d} \log\E[\mc{N}_{\textup{st}}] &= \Sigma_{\textup{st}}(\mu_0,t_0,4B''(0)).
\end{split}
\end{align}
\end{thm}

\begin{defn}
For any $t_0, b > 0$, let the \emph{Larkin mass} $\mu_c = \mu_c(t_0, b, L, d)$ be the unique positive solution to
\begin{equation}
\label{eqn:larkin}
	\int_\R \frac{\hat{\mu}_{-t_0\Delta}(\diff \lambda)}{(\mu_c+\lambda)^2} = \frac{1}{b}.
\end{equation}
It will also be useful to define, for any $\mu_0, t_0 > 0$, the critical noise parameter
\[
	b_c = b_c(\mu_0, t_0, L, d) = \left( \int_\R \frac{\hat{\mu}_{-t_0\Delta}(\diff \lambda)}{(\mu_0 + \lambda)^2} \right)^{-1}.
\]
\end{defn}

For $\mu_0 < \mu_c(t_0, b, L, d)$, we write $c = c(\mu_0, t_0, b, L, d)$ for the unique positive value satisfying
\[
	\int_\R \frac{\hat{\mu}_{-t_0\Delta}(\diff \lambda)}{(\mu_0 + \lambda)^2 + b^2c} = \frac{1}{b} 
\]
and use this to define
\[
	v = v(\mu_0, t_0, b, L, d) = -b \int_\R \frac{\mu_0 + \lambda}{(\mu_0 + \lambda)^2 + b^2c} \, \hat{\mu}_{-t_0\Delta}(\diff \lambda).
\]
Finally, we need the positive numbers
\[
	c_{\textup{tot}}(\mu_0, t_0, L, d) = \frac{\left( \int_\R \frac{\hat{\mu}_{-t_0\Delta}(\diff \lambda)}{(\mu_0 + \lambda)^2} \right)^4}{4\left( \int_\R \frac{\hat{\mu}_{-t_0\Delta}(\diff \lambda)}{(\mu_0 + \lambda)^4} \right)}, \qquad c_{\textup{min}}(\mu_0, t_0, L, d) = \frac{\left( \int_\R \frac{\hat{\mu}_{-t_0\Delta}(\diff \lambda)}{(\mu_0 + \lambda)^2} \right)^6}{24 \left( \int_\R \frac{\hat{\mu}_{-t_0\Delta}(\diff \lambda)}{(\mu_0 + \lambda)^3}\right)^2}.
\]

\begin{thm}
\label{thm:betterFLDcomplexity}
For each $t_0$ and $B''(0)$, the Larkin mass $\mu_c$ separates the phases of positive and zero complexity, both for total critical points (whose complexity exhibits quadratic near-critical behavior) and for local minima (whose complexity exhibits cubic near-critical behavior).

More precisely, the complexity functions satisfy the following, where $b=4B''(0)$:
\begin{enumerate}[label=(\roman*)]
\item if $\mu_0 \geq \mu_c(t_0, b, L, d)$, then $\Sigma(\mu_0, t_0, b) = \Sigma_{\textup{st}}(\mu_0, t_0, b) = 0$;
\item if $\mu_0 < \mu_c(t_0, b, L, d)$, then $\Sigma(\mu_0, t_0, b) > \Sigma_{\textup{st}}(\mu_0, t_0, b) > 0$, and these are given by
\begin{align*}
	\Sigma(\mu_0, t_0, b) &= -\frac{1}{L^d} \log(\det(\mu_0 \Id - t_0\Delta)) + \int_\R \log\abs{\lambda - v} (\rho_{\text{sc}, b} \boxplus \hat{\mu}_{-t_0 \Delta+\mu_0\Id })(\lambda)(\diff \lambda) - \frac{v^2}{2b}, \\
	\Sigma_{\textup{st}}(\mu_0, t_0, b) &= -\frac{1}{L^d} \log(\det(\mu_0 \Id - t_0\Delta)) + \int_\R \log\abs{\lambda - \ell} (\rho_{\text{sc}, b} \boxplus \hat{\mu}_{-t_0 \Delta+\mu_0\Id })(\lambda)(\diff \lambda) - \frac{\ell^2}{2b}
\end{align*}
where $\ell = \mathtt{l}(\rho_{\text{sc}, b} \boxplus \hat{\mu}_{-t_0 \Delta + \mu_0 \Id})$ and $v$ is as above; and
\item for fixed $\mu_0$ and $t_0$, and supercritical $b$, we have 
\begin{align*}
	\Sigma(\mu_0, t_0, b) = c_{\text{tot}}(\mu_0, t_0, L, d) \cdot (b - b_c)^2 + O((b-b_c)^3), \\ 
	\Sigma_{\textup{st}}(\mu_0, t_0, b) = c_{\textup{min}}(\mu_0, t_0, L, d) \cdot (b - b_c)^3 + O((b-b_c)^4).
\end{align*}
\end{enumerate}
\end{thm}
For the proof of this theorem, we use determinant asymptotics from our companion paper \cite{BenBouMcK2021I} to give the complexity as a variational problem over $\R^{L^d}$. Using a remarkable MDE-induced convexity property, we reduce this to a variational problem over $\R$, namely \eqref{eqn:FLD_onedimensional}. We analyze this one-dimensional variational problem with a dynamic approach, varying $B''(0)$ for fixed $\mu_0$ and $t_0$.

We remark that Fyodorov and Le Doussal also exhibited a quadratic/cubic near-critical behavior for this model but in a different scaling, varying $\mu_0$ for fixed $B''(0)$ and $t_0$ \cite{FyoLeD2020}. 

%%%%%%%%%%%%%%%%%%%%%%%%%%%%%%%%%%%%%%%%%%%%%%%%%%%%%%%%%%%%
%%%%%%%%             Subsection: Soft spins in an anisotropic well
%%%%%%%%%%%%%%%%%%%%%%%%%%%%%%%%%%%%%%%%%%%%%%%%%%%%%%%%%%%%

\subsection{Soft spins in an anisotropic well.}\
\label{subsec:softspins}
We consider the random Hamiltonian $\mc{H}_N : \R^N \to \R$ given by
\[
    \mc{H}_N(x) = \frac{\ip{x, D_Nx}}{2} + V_N(x),
\]
where $D_N$ is a real symmetric matrix satisfying conditions below, and where $V_N$ is an isotropic centered Gaussian field with covariance
\[
    \E[V_N(x_1)V_N(x_2)] = NB\left( \frac{\|x_1 - x_2\|^2}{2N} \right)
\]
with $B : \R_+ \to \R_+$ a correlator function (meaning it has the representation \eqref{eqn:schoenberg}). As in Section \ref{subsec:elasticmanifold}, we assume that $B$ is four times differentiable at zero to ensure twice-differentiability of the field, and we assume 
\[
    0 < |B^{(i)}(0)| \qquad \text{for } i = 0, 1, 2,
\]
for nondegeneracy of the field and its first two derivatives.

We suppose that $(D_N)_{N=1}^\infty$ is a sequence of real symmetric matrices, $D_N \in \R^{N \times N}$, and that there exists some compactly supported measure $\mu_D$ such that, for some $\epsilon > 0$, we have
\begin{equation}
\label{eqn:speed_of_environment}
    d_{\textup{BL}}(\hat{\mu}_{D_N}, \mu_D) \leq N^{-\epsilon}
\end{equation}
and the eigenvalues are uniformly gapped away from zero and from infinity, in that
\[
	\epsilon \leq \inf_N \lambda_{\textup{min}}(D_N) \leq \sup_N \lambda_{\textup{max}}(D_N) \leq \frac{1}{\epsilon}.
\]

Although our results are for the $N \to +\infty$ limit, Figure \ref{fig:softspins} displays how changing $D_N$ can qualitatively change the count of critical points when $N = 2$.

\begin{figure}[h!]
\centering
    \begin{subfigure}{0.5\textwidth}
        \includegraphics[scale=0.22]{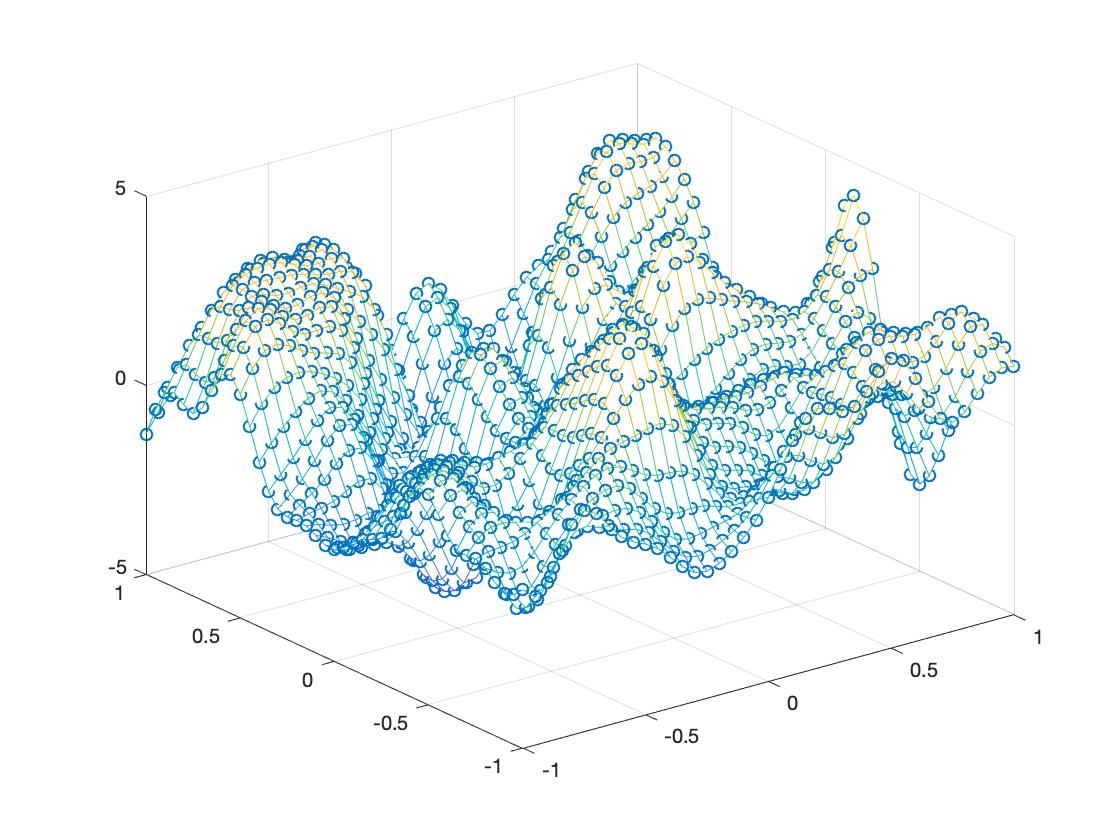}
        \caption{The landscape appears rugged when $D_2 = 0$.}
    \end{subfigure}%
    \begin{subfigure}{0.5\textwidth}
        \includegraphics[scale=0.22]{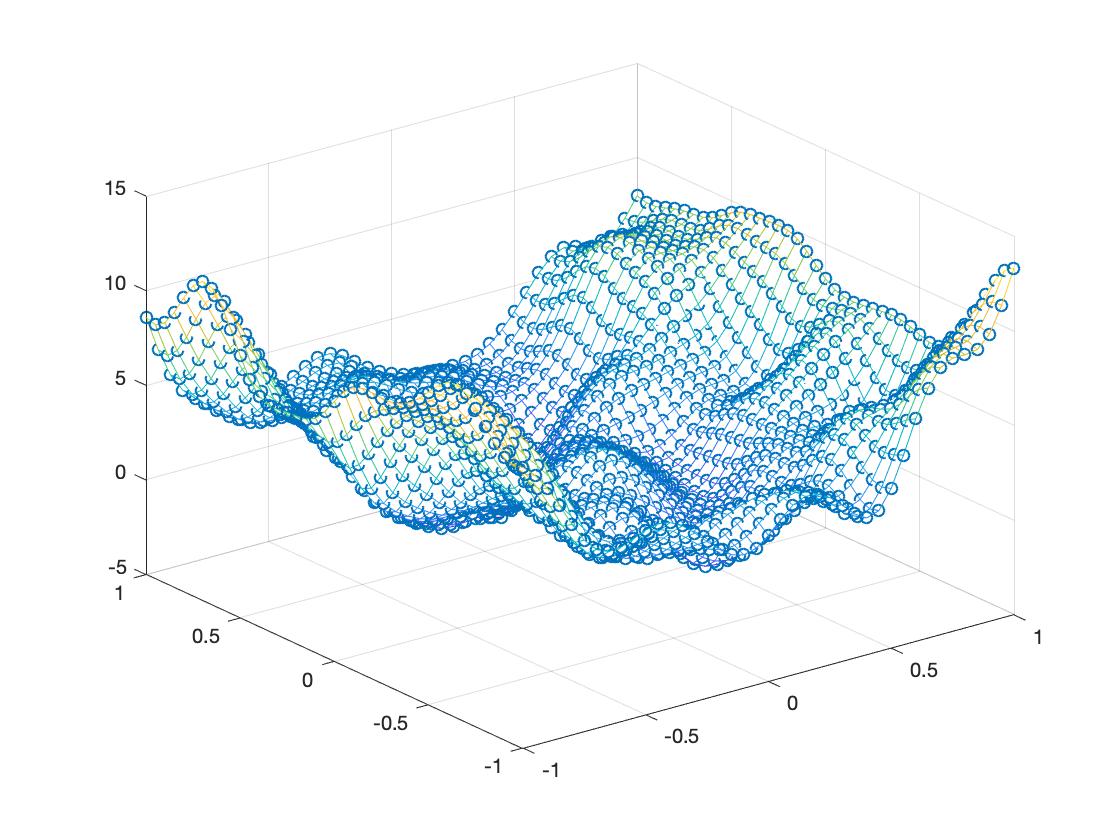}
        \caption{There are fewer critical points when $D_2 = 3 \cdot \left(\begin{smallmatrix} 6 & 0 \\ 0 & 1 \end{smallmatrix}\right)$.}
    \end{subfigure}%
    \\
    \begin{subfigure}{0.5\textwidth}
        \includegraphics[scale=0.22]{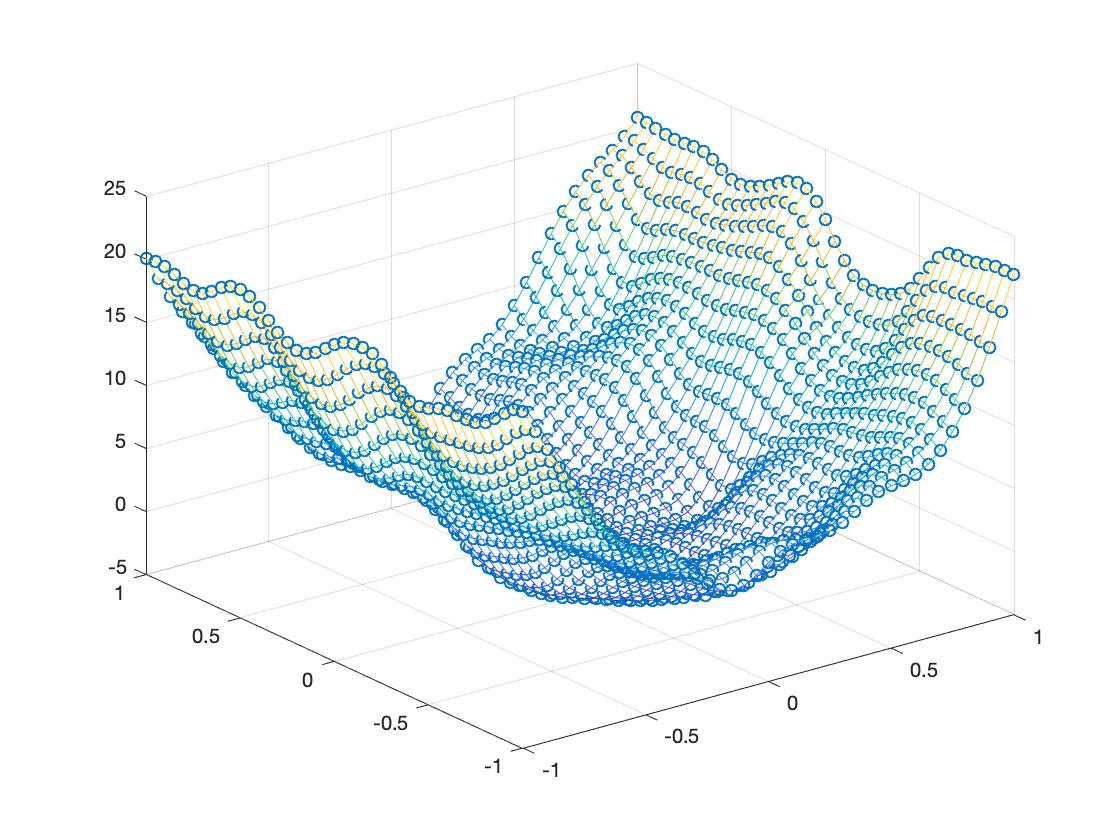}
        \caption{There are many fewer critical points when $D_2 = 6 \cdot \left(\begin{smallmatrix} 6 & 0 \\ 0 & 1 \end{smallmatrix}\right)$.}
    \end{subfigure}%
    \begin{subfigure}{0.5\textwidth}
        \includegraphics[scale=0.22]{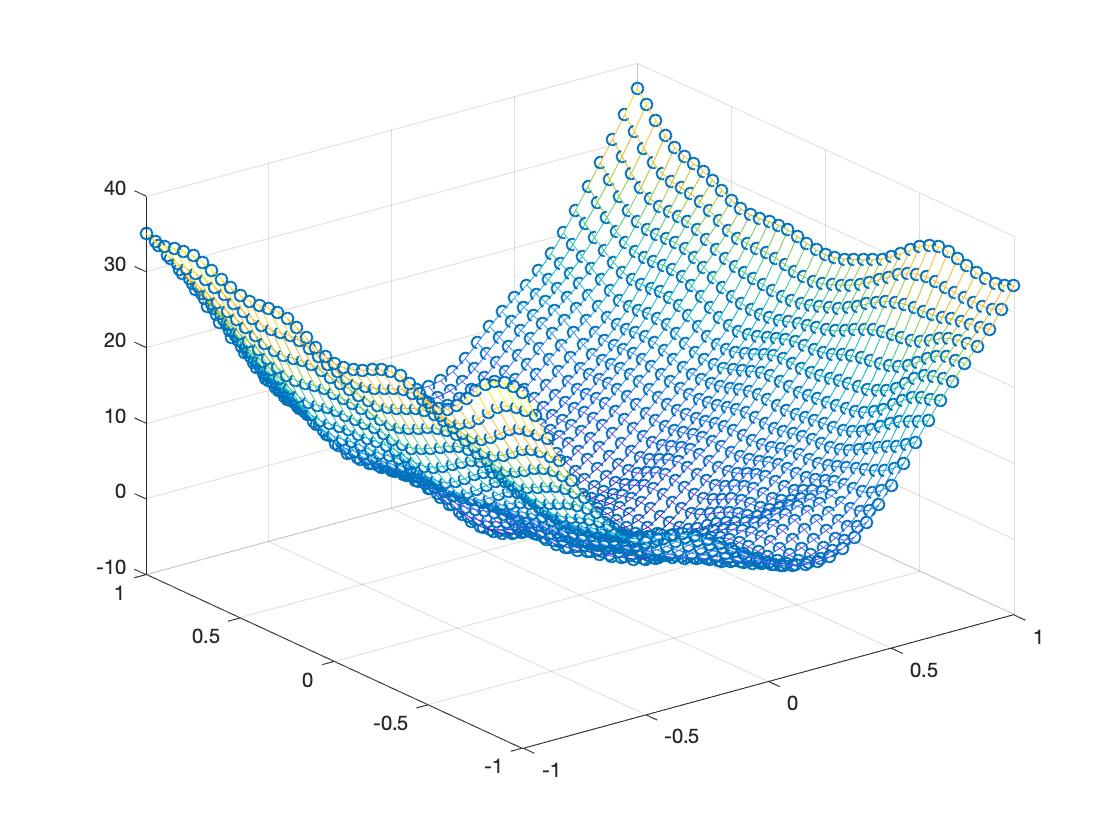}
        \caption{There are almost no critical points when $D_2 = 10 \cdot \left(\begin{smallmatrix} 6 & 0 \\ 0 & 1 \end{smallmatrix}\right)$.}
    \end{subfigure}%
\caption{Numerical (discretized) samples of $\mc{H}_2$ on $[-1,1]^2$ with the same (resampled) noise and four different choices of signal $D_N$. Precisely, these are scatterplots of $\mc{H}_2(x)$ values for $x$ on a $41 \times 41$ lattice, with an overlaid mesh fit, made with Matlab. Here $B(r) = \exp(-80r)$, meaning $\E[V_2(x)V_2(y)] = 2\exp(-20\|x-y\|^2)$.}
\label{fig:softspins}
\end{figure}

%%%%%%%%%%%%%%%%%%%%%%%%%%%%%%%%%%%%%%%%%%%%%%%%%%%%%%%%%%%%
%%%%%%%%             Subsubsection: History
%%%%%%%%%%%%%%%%%%%%%%%%%%%%%%%%%%%%%%%%%%%%%%%%%%%%%%%%%%%%

\subsubsection*{History.}\
Models of the form $V_N(x) + \frac{\mu}{2}\|x\|^2$ (recall \eqref{eqn:fyodorov2004model}), with various choices of randomness, have been considered in a wide variety of contexts. There are nice overviews of the literature in \cite{Fyo2004, FyoWil2007, AufZen2020}. In the early 1990s, the model was studied by M\'{e}zard-Parisi \cite{MezPar1992} and by Engel \cite{Eng1993} as a zero-dimensional case of the elastic manifold. The complexity was computed by Fyodorov \cite{Fyo2004} for total critical points and Fyodorov-Williams \cite{FyoWil2007} for minima, finding a phase transition between positive and zero complexity at an explicit $\mu_c$. Fyodorov and Nadal found that the complexity of minima for $\mu$ near $\mu_c$, scaled appropriately, tends to a limiting shape related to the Tracy-Widom distribution \cite{FyoNad2012}.  

There is also a long history of generalizing the model, as we do: Fyodorov and Williams actually studied the complexity after replacing the quadratic confinement $\frac{\mu}{2}\|x\|^2$ with a general radial confinement $NU(\frac{\|x\|^2}{2N})$ for some function $U : \R \to \R$ which is increasing and convex \cite{FyoWil2007}. In some sense our extension is orthogonal to theirs: they let the confinement be non-quadratic, whereas we let it be non-radial. As another generalization, if $V_N(x)$ is not isotropic but merely has isotropic \emph{increments} (meaning $\E[(V_N(x)-V_N(y))^2]$ depends only $\|x-y\|$), then the model can admit long-range correlations; this was studied in the physics literature by Fyodorov and co-authors \cite{FyoSom2007, FyoBou2008}, and its complexity was recently computed by Auffinger and Zeng \cite{AufZen2020}.

Our generalization is reminiscent of the work of Fan, Mei, and Montanari on an upper bound for the complexity of the TAP free energy of the Sherrington-Kirkpatrick model \cite{FanMeiMon2021}. Indeed, via the Kac-Rice formula, the random matrix that appears in our problem is a full-rank deformation of GOE (see \eqref{eqn:parabolastandardform}). A similar random matrix, in fact with an additional low-rank deformation, appears in \cite{FanMeiMon2021}.

%%%%%%%%%%%%%%%%%%%%%%%%%%%%%%%%%%%%%%%%%%%%%%%%%%%%%%%%%%%%
%%%%%%%%             Subsubsection: Results
%%%%%%%%%%%%%%%%%%%%%%%%%%%%%%%%%%%%%%%%%%%%%%%%%%%%%%%%%%%%

\subsubsection*{Results.}\
Let $\Crt_N^{\text{tot}}(\mc{H}_N)$ be the total number of critical points of $\mc{H}_N$ and $\Crt_N^{\min{}}(\mc{H}_N)$ be the total number of local minima.

\begin{defn}
For any $t > 0$ and any $\mu_D$ compactly supported in $(0,\infty)$, define
\begin{align}
    \Sigma^{\textup{tot}}(\mu_D,t) &= -\int_\R \log(\lambda) \mu_D(\diff \lambda) + \sup_{u \in \R} \left\{ \int_\R \log\abs{\lambda - u} (\rho_{\text{sc},t} \boxplus \mu_D)(\lambda) \diff \lambda - \frac{u^2}{2t}\right\}, \label{eqn:defsigmatot} \\
    \Sigma^{\textup{min}}(\mu_D,t) &= -\int_\R \log(\lambda) \mu_D(\diff \lambda) + \sup_{u \leq \mathtt{l}(\rho_{\text{sc},t} \boxplus \mu_D)} \left\{ \int_\R \log\abs{\lambda - u} (\rho_{\text{sc},t} \boxplus \mu_D)(\lambda) \diff \lambda - \frac{u^2}{2t}\right\}. \label{eqn:defsigmamin}
\end{align}
We will show that these suprema are achieved, possibly not uniquely.
\end{defn}

Theorem \ref{thm:asywell} below shows the relevance of these functions for complexity, and Theorem \ref{thm:softspins_threshold} analyzes the variational problems from \eqref{eqn:defsigmatot} and \eqref{eqn:defsigmamin} to describe the phase portrait in $\mu_D$ and $t$. In particular, the regimes of positive complexity for the total number of critical points and local minima coincide for any $\mu_D$, and the exponents describing near-critical behavior are universal in $\mu_D$. 

\begin{thm}
\label{thm:asywell}
We have
\[
    \lim_{N \to \infty} \frac{1}{N}\log \E[\Crt_N^{\textup{tot}}(\mc{H}_N)] = \Sigma^{\textup{tot}}(\mu_D,B''(0)).
\]
If in addition $D_N$ has no external outliers, in the sense that 
\begin{align*}
    \lim_{N \to \infty} \lambda_{\min{}}(D_N) = \mathtt{l}(\mu_D) \qquad \text{and} \qquad \lim_{N \to \infty} \lambda_{\max{}}(D_N) = \mathtt{r}(\mu_D),
\end{align*}
then
\[
    \limsup_{N \to \infty} \frac{1}{N}\log \E[\Crt_N^{\min{}}(\mc{H}_N)] = \Sigma^{\min{}}(\mu_D,B''(0)).
\]
\end{thm}

\begin{rem}
We emphasize that Theorem \ref{thm:asywell} shows that special directions in the environment (meaning outliers in $D_N$) have no effect on the total number of critical points at exponential scale, as long as there are $o(N)$ many of them. We leave open the effect of special directions on minima.
\end{rem}

We define the important threshold 
\begin{equation}
\label{eqn:softspins_threshold}
	t_c = t_c(\mu_D) = \left( \int_\R \frac{\mu_D(\diff \lambda)}{\lambda^2} \right)^{-1}.
\end{equation}
For $t > t_c$, we write $c = c(t,\mu_D)$ for the unique positive value satisfying 
\[
	\frac{1}{t} = \int_\R \frac{1}{\lambda^2 + t^2 c} \, \mu_D(\diff \lambda)
\]
and use this to define
\[
	v = v(t,\mu_D) = -t \int_\R \frac{\lambda}{\lambda^2+t^2 c(t,\mu_D)}  \, \mu_D(\diff \lambda).
\]
We also need the positive numbers
\begin{equation}
\label{eqn:singularityconstants}
    c_{\textup{tot}}(\mu_D) = \frac{\left(\int_\R \frac{\mu_D(\diff \lambda)}{\lambda^2}\right)^4}{4 \left( \int_\R \frac{\mu_D(\diff \lambda)}{\lambda^4}\right)}, \qquad c_{\textup{min}}(\mu_D) = \frac{\left(\int_\R \frac{\mu_D(\diff \lambda)}{\lambda^2}\right)^6}{24\left(\int_\R \frac{\mu_D(\diff \lambda)}{\lambda^3} \right)^2}.
\end{equation}

\begin{thm}
\label{thm:softspins_threshold}
For every $t > 0$ and every probability measure $\mu_D$ compactly supported in $(0,\infty)$, 
\begin{enumerate}[label=(\roman*)]
\item if $t \leq t_c$, then $\Sigma^{\textup{tot}}(\mu_D, t) = \Sigma^{\textup{min}}(\mu_D, t) = 0$;
\item if $t > t_c$, then $\Sigma^{\textup{tot}}(\mu_D, t) > \Sigma^{\textup{min}}(\mu_D, t) > 0$, and these are given by
\begin{align}
    \Sigma^{\textup{min}}(\mu_D,t) &= -\int_\R \log(\lambda) \mu_D(\diff \lambda) + \int_\R \log\abs{\lambda - \ell} (\rho_{\text{sc},t} \boxplus \mu_D)(\lambda) \diff \lambda - \frac{\ell^2}{2t}, \label{eqn:softspins_minima_formula}  \\
    \Sigma^{\textup{tot}}(\mu_D,t) &= -\int_\R \log(\lambda) \mu_D(\diff \lambda) + \int_\R \log\abs{\lambda - v} (\rho_{\text{sc},t} \boxplus \mu_D)(\lambda) \diff \lambda - \frac{v^2}{2t}, \label{eqn:softspins_total_formula}
\end{align}
where $\ell = \mathtt{l}(\rho_{\text{sc},t} \boxplus \mu_D)$ and $v$ is as above; and
\item for supercritical $t$, we have
\begin{align*}
    \Sigma^{\textup{tot}}(\mu_D,t) &= c_{\textup{tot}}(\mu_D) \cdot (t-t_c)^2 + O((t-t_c)^3), \\
    \Sigma^{\textup{min}}(\mu_D,t) &= c_{\textup{min}}(\mu_D) \cdot (t-t_c)^3 + O((t-t_c)^4),
\end{align*}
with $c_{\textup{tot}}(\mu_D)$, $c_{\textup{min}}(\mu_D)$ as in \eqref{eqn:singularityconstants}.
\end{enumerate}
\end{thm}

The proof of this theorem relies on a dynamic approach, like for results in Section \ref{subsec:elasticmanifold}. We also use two important inputs: (i) the Burgers' equation satisfied by the Stieltjes transform of the semicircle distribution, and (ii) an inequality from free probability, due to Guionnet and Ma{\"i}da, regarding the subordination function of the free convolution at the edge. We also need a new result in free probability, possibly of independent interest, which we prove in Appendix \ref{sec:freeconvolution}: The free convolution of any measure with semicircle decays at least as fast as a square-root at its extremal edges.

We remark that it is not obvious that the same threshold $t_c$ should work both for total critical points and for local minima, and the analogue is false in closely related models. For example, consider the Hamiltonian \eqref{eqn:fyodorov2004model}, i.e. $H_N(x) = \frac{\mu}{2}\|x\|^2 + V_N(x)$, but defined over $\{x \in \R^N : \|x\| \leq R\sqrt{N}\}$ for some fixed $R > 0$ rather than over the whole space. Fyodorov, Sommers, and Williams \cite{FyoSomWil2007} showed that, for some choices of $R$, the complexity of total critical points is positive but the complexity of minima vanishes. (See \cite{BraDea2007} for related independent work.) But \cite{FyoWil2007} proved the analogue of point 1 for their model, discussed above, which is defined on the full space. See \cite[Section 2.4]{FyoWil2007} for further discussion of the differences between the full-space models like ours with ``smooth confining potentials'' and the ``hard-wall confining potentials'' of \cite{FyoSomWil2007}. 

\begin{example}
The model \eqref{eqn:fyodorov2004model} is a special case when $D_N = \mu \Id$ for some scalar $\mu > 0$. In our notation, this corresponds to $\mu_D = \delta_\mu$. Theorem \ref{thm:asywell} yields
\begin{align}
\label{eqn:fyodorov_formulas}
\begin{split}
    \Sigma^{\textup{tot}}(\delta_\mu,B''(0)) &= \begin{cases} \frac{1}{2}\left(\frac{\mu^2}{B''(0)} - 1 \right) - \log\left(\frac{\mu}{\sqrt{B''(0)}}\right) & \text{if } \mu \leq \mu_c := \sqrt{B''(0)} \quad \text{(equivalently, if $\int \frac{\mu_D(\diff t)}{t^2} \geq \frac{1}{B''(0)}$)}, \\ 0 & \text{if } \mu \geq \mu_c, \end{cases} \\
    \Sigma^{\textup{min}}(\delta_\mu,B''(0)) &= \begin{cases} \frac{1}{2}\left[ -3 - \log\left( \frac{\mu^2}{B''(0)}\right) + 4 \cdot \frac{\mu}{\sqrt{B''(0)}} - \frac{\mu^2}{B''(0)} \right] & \text{if } \mu \leq \mu_c, \\
    0 & \text{if } \mu \geq \mu_c. \end{cases}
\end{split}
\end{align}
These recover results of \cite[Equations (18-19)]{Fyo2004} and \cite[Equation (81)]{FyoWil2007}, respectively. We also recover their results on decay near criticality, as one can check by hand that the behavior predicted by Theorem \ref{thm:softspins_threshold} (which gives $c_{\textup{tot}}(\delta_\mu) = \frac{1}{4\mu^4}$ and $c_{\textup{min}}(\delta_\mu) = \frac{1}{24\mu^6}$ here) is correct.
\end{example}

\begin{example}
We give one more explicit example, namely when
\[
    \mu_D(\diff x) = \rho_{\text{sc}}^{(m,\sigma^2)}(\diff x) = \frac{\sqrt{(4\sigma^2 - (x-m)^2)_+}}{2\pi\sigma^2} \diff x
\]
is the semicircle law of mean $m$ and variance $\sigma^2$. (Notice we need $\mu_D(\diff x)$ supported in $(0,\infty)$, equivalently $m - 2\sigma > 0$, for the model to be well-defined.) In this case we have 
\begin{align*}
    \Sigma^{\textup{tot}}(\mu_D,B''(0)) &= \begin{cases} \frac{m}{4\sigma^2} \left( \sqrt{m^2-4\sigma^2} - \frac{m}{1+2\frac{\sigma^2}{B''(0)}} \right) - \log \left( \frac{m+\sqrt{m^2-4\sigma^2}}{2\sqrt{B''(0) + \sigma^2}} \right) & \text{if } \int \frac{\mu_D(\diff t)}{t^2} = \frac{-1+\frac{m}{\sqrt{m^2-4\sigma^2}}}{2\sigma^2} \geq \frac{1}{B''(0)}, \\ 0 & \text{if } \int \frac{\mu_D(\diff t)}{t^2} \leq \frac{1}{B''(0)}, \end{cases} \\
    \Sigma^{\textup{min}}(\mu_D,B''(0)) &= \begin{cases}
    -1 + \frac{m(-m+\sqrt{m^2-4\sigma^2})}{4\sigma^2} - \frac{m^2+4\sigma^2-4m\sqrt{B''(0)+\sigma^2}}{2B''(0)} - \log \left( \frac{m+\sqrt{m^2-4\sigma^2}}{2\sqrt{B''(0)+\sigma^2}} \right) & \text{if } \int \frac{\mu_D(\diff t)}{t^2} \geq \frac{1}{B''(0)}, \\
    0 & \text{if } \int \frac{\mu_D(\diff t)}{t^2} \leq \frac{1}{B''(0)}. \end{cases}
\end{align*}
As a consistency check, in the limit $\sigma \downarrow 0$ we obtain exactly the formulas \eqref{eqn:fyodorov_formulas} with $\mu$ replaced by $m$. 
\end{example}

%%%%%%%%%%%%%%%%%%%%%%%%%%%%%%%%%%%%%%%%%%%%%%%%%%%%%%%%%%%%
%%%%%%%%%%%%%%%%%%%%%%%%%%%%%%%%%%%%%%%%%%%%%%%%%%%%%%%%%%%%
%%%%%%%%
%%%%%%%%             Section: Stability of the Matrix Dyson Equation
%%%%%%%%
%%%%%%%%%%%%%%%%%%%%%%%%%%%%%%%%%%%%%%%%%%%%%%%%%%%%%%%%%%%%
%%%%%%%%%%%%%%%%%%%%%%%%%%%%%%%%%%%%%%%%%%%%%%%%%%%%%%%%%%%%

\section{Stability of the Matrix Dyson Equation}
\label{sec:stability}

In this section, our goal is to give general results on the stability of the Matrix Dyson Equation. For example, the MDE for GOE matrices is
\[
	\Id + \left(z\Id + \frac{1}{N}\Tr(M_N(z)) + \frac{1}{N}M_N(z)^T\right)M_N(z) = 0, \quad \im M_N(z) > 0,
\]
but $\frac{1}{N}M_N(z)^T$ should be thought of as an error, and it is more convenient to consider the unique solution $M'_N(z)$ to
\[
	\Id + \left(z\Id + \frac{1}{N}\Tr(M'_N(z))\right)M'_N(z) = 0, \quad \im M'_N(z) > 0.
\]
In this section we prove stability of such MDEs to conclude $\frac{1}{N}\Tr M_N(z) \approx \frac{1}{N}\Tr M'_N(z)$ for their respective unique solutions. Similar arguments have appeared in papers of Erd{\H o}s and collaborators, for example \cite{AltErdKruNem2019}, but in more involved contexts where an exact deterministic solution of the MDE is compared to a random near-solution with small (random) error. Since we are interested in slightly different perturbations of the MDE, and only in the deterministic case, we adapt their arguments to give a short self-contained proof here.

Fix a sequence $(P_N)_{N=1}^\infty$ of positive integers (typically we take $P_N = N$ or $P_N$ independent of $N$). It is known \cite{HelFarSpe2007} that, whenever $\mc{S} : \C^{P_N \times P_N} \to \C^{P_N \times P_N}$ is a linear operator that is self-adjoint with respect to the inner product $\ip{R,T} = \Tr(R^\ast T)$ and that preserves the cone of positive-semi-definite matrices, and whenever $a(u) \in \R^{P_N \times P_N}$ is symmetric, the problem
\begin{equation}
\label{eqn:generic_mde}
    -M^{-1}(u,z) = z\Id - a(u) + \mc{S}[M(u,z)] \quad \text{subject to} \quad \im M(u,z) > 0
\end{equation}
has a unique solution $M(u,z) \in \C^{P_N \times P_N}$ for each $z \in \mathbb{H}$ and $u \in \R^m$, and 
\begin{equation}
\label{eqn:mde_trivialbd}
    \|M(u,z)\| \leq \frac{1}{\eta}.
\end{equation}
Fix two sequences $(\mc{S}_N)_{N=1}^\infty$ and $(\mc{S}'_N)_{N=1}^\infty$ of such operators and two sequences $(a_N(u))_{N=1}^\infty$ and $(a'_N(u))_{N=1}^\infty$ of such matrices (i.e., $\mc{S}_N$ and $\mc{S}'_N$ act on $\C^{P_N \times P_N}$, and $a_N(u), a'_N(u) \in \R^{P_N \times P_N}$), and consider the associated solutions:
\[
    \mc{S}_N \text{ and } a_N(u) \quad \text{induce} \quad M_N(u,z), \qquad \mc{S}'_N \text{ and } a'_N(u) \quad \text{induce} \quad M'_N(u,z).
\]
In this section, our goal is to show that $M_N$ and $M'_N$ are close if $\mc{S}_N$ and $\mc{S}'_N$ are close and $a_N(u)$ and $a'_N(u)$ are close; we will use this to help identify $\mu_\infty$ for both of our models.

\begin{lem}
\label{lem:generic_distance_stability}
Suppose that, for some $\kappa > 0$,
\begin{align}
    \sup_N \max(\|a_N(u)\|,\|a'_N(u)\|) &\leq \kappa\max(1,\|u\|), \label{eqn:mde_bddmean} \\
    \|\mc{S}'_N\|_{\text{hs} \to \|\cdot\|} &\leq \kappa, \label{eqn:mde_bddoperator} \\
    \|\mc{S}_N - \mc{S}'_N\|_{\|\cdot\| \to \|\cdot\|} &\leq \frac{\kappa}{N}, \label{eqn:mde_closeoperators} \\
    \|a_N(u) - a'_N(u)\| &\leq \frac{\kappa\max(1,\|u\|)}{N}. \label{eqn:mde_closemeans}
\end{align}

If $0 < \gamma < \frac{1}{50}$, then for each $R$ and each $A$ there exists $\delta > 0$ with 
\[
    \sup_{u \in B_R} \frac{1}{N} \int_{-A}^A \abs{\Tr(M_N(u,E+\ii N^{-\gamma})) - \Tr(M'_N(u,E+\ii N^{-\gamma}))} \diff E \leq \frac{1}{\delta} N^{-\delta}.
\]
\end{lem}
\begin{proof}
Notice that $M_N(u,z)$ almost solves the MDE \eqref{eqn:generic_mde} with $\mc{S} = \mc{S}'_N$ and $a(u) = a'_N(u)$; in fact,
\[
    -M_N(u,z)^{-1} = z\Id - a'_N(u) + \mc{S}'_N[M_N(u,z)] + \underbrace{(\mc{S}_N-\mc{S}'_N)[M_N(u,z)] + a'_N(u) - a_N(u)}_{=:d_N(u,z)},
\]
and $d_N(u,z)$ is an error term in the sense that, if $u \in B_R$ (we take $R \geq 1$ without loss of generality), from \eqref{eqn:mde_trivialbd}, \eqref{eqn:mde_closeoperators}, and \eqref{eqn:mde_closemeans} we have
\begin{equation}
\label{eqn:genericmde_dsmall}
    \|d_N(u,z)\| \leq \frac{\kappa}{N\eta} + \frac{\kappa\max(1,\|u\|)}{N} \leq \frac{\kappa(1+\eta R)}{N\eta}.
\end{equation}
We will apply standard stability theory of the MDE, which lets us conclude from this that $M_N$ is close to $M'_N$. In fact, our goal is significantly easier than that in the literature, because our approximate solution to the MDE is deterministic. In the generality we need, this theory has been developed in \cite{AltErdKruNem2019}, and manipulations exactly like those preceding (4.25) there let us conclude
\begin{equation}
\label{eqn:generic_first_stability_bd}
    \|M_N(u,z) - M'_N(u,z)\| \leq \|\ms{L}_N^{-1}(u,z)\| \|M'_N(u,z)\| \bigg( \|d_N(u,z)\| \|M_N(u,z)\| + \|\mc{S}'_N \| \|M_N(u,z) - M'_N(u,z)\|^2 \bigg).
\end{equation}
Here $\ms{L}_N(u,z) : \C^{P_N \times P_N} \to \C^{P_N \times P_N}$ is the ``stability operator''
\[
    \ms{L}_N(u,z)[T] = T - M'_N(u,z)\mc{S}'_N[T]M'_N(u,z),
\]
which is invertible for every $u$ and every $z$ by \cite[Lemma 3.7(i)]{AltErdKruNem2019}. Inserting the estimates \eqref{eqn:mde_trivialbd}, \eqref{eqn:mde_bddoperator}, and \eqref{eqn:genericmde_dsmall} into \eqref{eqn:generic_first_stability_bd} yields
\begin{equation}
\label{eqn:genericmde_quadraticineq}
    \|M_N(u,z) - M'_N(u,z)\| \leq \frac{\kappa}{\eta} \|\ms{L}^{-1}_N(u,z)\| \left( \frac{1+\eta R}{N\eta^2} + \|M_N(u,z) - M'_N(u,z)\|^2 \right).
\end{equation}
As usual, this quadratic inequality is fundamental to our strategy: We use it to show that $\|M_N-M'_N\|$ is small for very large $\eta$, then fix $E$ and decrease $\eta$ with a continuity argument. To make this bound useful, we import the following estimate on $\|\ms{L}^{-1}(u,z)\|$ from \cite[(3.23), (3.22), Convention 3.5]{AltErdKruNem2019} combined with \eqref{eqn:mde_trivialbd}: There exists a constant $C$ such that, for all $u$ and $z$, we have
\begin{equation}
\label{eqn:generic_mde_versatileLestimate}
    \|\ms{L}^{-1}(u,z)\| \leq C\left(1+\frac{1}{\eta^2} + \frac{\|M'_N(u,z)^{-1}\|^9}{\eta^{13}}\right).
\end{equation}
We use this estimate differently for $\eta \geq 1$ and $\eta \leq 1$, which are the two steps in the remainder of our argument.

\bigskip

\noindent \textbf{Step 1 ($\eta \geq 1$):} If $u \in B_R$ for some $R$ (we take $R \geq 1$ without loss of generality), then taking norms directly in the MDE \eqref{eqn:generic_mde} and applying \eqref{eqn:mde_trivialbd} and \eqref{eqn:mde_bddmean} yields
\[
    \|M'_N(u,z)^{-1}\| \leq \abs{z} + \|a_N(u)\| + \|M'_N(u,z)\| \leq \abs{z} + \kappa R + 1.
\]
If $\eta \geq 1$, then $\abs{z} \leq \eta \sqrt{1+E^2}$, so for any choice of $E_{\max{}}$ there exists a constant $C_{R,E_{\max{}}} = C_{R,E_{\max{}},\kappa_1}$ such that
\[
    \sup_{\abs{E} \leq E_{\max{}}, \, \eta \geq 1, \, u \in B_R} \frac{\|M'_N(u,z)^{-1}\|}{\eta} \leq C_{R,E_{\max{}}}. 
\]
Inserting this into \eqref{eqn:generic_mde_versatileLestimate} gives, for a new constant $\widetilde{C}_{R,E_{\max{}}} = \widetilde{C}_{R,E_{\max{}},\kappa_1}$,
\begin{equation}
\label{eqn:genericmde_crudeLestimate}
    \sup_{\abs{E} \leq E_{\max{}}, \, \eta \geq 1, \, u \in B_R} \|\ms{L}^{-1}(u,z)\| \leq \widetilde{C}_{R,E_{\max{}}}.
\end{equation}

Now fix $\abs{E} \leq E_{\max{}}$, and consider the functions $f_N : (0,\infty) \to \R$ and $g^{\pm}_N : [1,\infty) \to \R$ defined by
\begin{align*}
    f_N(\eta) &= f_{N,u}(\eta) = \|M_N(u,E+\ii \eta) - M'_N(u,E+\ii \eta)\|, \\ 
    g^{\pm}_N(\eta) &= \frac{\eta}{2\kappa\widetilde{C}_{R,E_{\max{}}}} \left( 1\pm\sqrt{1-\frac{4\kappa^2(\widetilde{C}_{R,E_{\max{}}})^2(1+\eta R)}{N\eta^4}} \right).
\end{align*}
(The functions $g^{\pm}_N(\eta)$ are well-defined if $N \geq 4(\widetilde{C}_{R,E_{\max{}}})^2(1+R)$.) The quadratic inequality \eqref{eqn:genericmde_quadraticineq} with the estimate \eqref{eqn:genericmde_crudeLestimate} inserted give, for all $N \geq 4(\widetilde{C}_{R,E_{\max{}}})^2$ and all $\eta \geq 1$, 
\[
    f_N(\eta) \in [0,g_N^{-}(\eta)] \cup [g_N^{+}(\eta),\infty).
\]
If $\eta \geq  \max\left\{1,\sqrt{4\kappa\widetilde{C}_{R,E_{\max{}}}}\right\}$, then the crude bound \eqref{eqn:mde_trivialbd} yields
\[
    f_N(\eta) \leq \frac{\eta}{2\widetilde{C}_{R,E_{\max{}}}} < g_N^{+}(\eta),
\]
so that $f_N(\eta) \leq g_N^{-}(\eta)$. But since $M_N(u,z)$ and $M'_N(u,z)$ are both holomorphic matrix-valued functions of $z$ \cite{HelFarSpe2007}, we know that $f_N(\eta)$ is a continuous function of $\eta$. Since $g_N^{-}(\eta) < g_N^{+}(\eta)$ for all $\eta > 1$, we have $f_N(\eta) \leq g_N^{-}(\eta)$ down to $\eta = 1$. Notice that this is uniform in $u \in B_R$.

\bigskip

\noindent \textbf{Step 2 ($\eta \leq 1$):} Now we estimate
\[
    \|M'_N(u,z)^{-1}\| \leq \abs{z}+\kappa R + \frac{1}{\eta} \leq \frac{C'_{R,E_{\max{}}}}{\eta}
\]
for some $C'_{R,E_{\max{}}} = C'_{R,E_{\max{}},\kappa_1}$. Inserting this and \eqref{eqn:mde_trivialbd} shows that, for some $C''_{R,E_{\max{}}} = C''_{R,E_{\max{}},\kappa_1}$, we have
\begin{equation}
\label{eqn:genericmde_fineLestimate}
    \sup_{\abs{E} \leq E_{\max{}}, \, \eta \leq 1, \, u \in B_R} \frac{\|\ms{L}^{-1}(u,z)\|}{\eta^{-22}} \leq C''_{R,E_{\max{}}}.
\end{equation}
Now fix $\abs{E} \leq E_{\max{}}$ and consider the functions $h_N^{\pm} : [N^{-1/50},1] \to \R$ defined by
\[
    h_N^{\pm}(\eta) = \frac{\eta^{23}}{2\kappa C''_{R,E_{\max{}}}} \left(1\pm \sqrt{1-\frac{4\kappa^2(C''_{R,E_{\max{}}})^2(1+\eta R)}{N\eta^{48}}} \right).
\]
As above, the quadratic inequality \eqref{eqn:genericmde_quadraticineq} with the estimate \eqref{eqn:genericmde_fineLestimate} inserted give, for all $N$ and all $\eta \leq 1$, 
\[
    f_N(\eta) \in [0,h_N^{-}(\eta)] \cup [h_N^{+}(\eta),\infty).
\]
But when $\eta = 1$ and $N \geq 4\kappa^2C''_{R,E_{\max{}}}\widetilde{C}_{R,E_{\max{}}}(1+R)$ we have (using $1-\sqrt{1-x} \leq x$)
\[
    f_N(1) \leq g_N^{-}(1) \leq \frac{2\kappa\widetilde{C}_{R,E_{\max{}}}(1+R)}{N} \leq \frac{1}{2\kappa C''_{R,E_{\max{}}}} < h_N^{+}(1),
\]
so the same continuity argument as above gives
\begin{equation}
\label{eqn:genericmde_resultofcty}
    f_N(\eta) \leq h_N^{-}(\eta) \leq \frac{2\kappa C''_{R,E_{\max{}}}(1+R)}{N\eta^{25}}.
\end{equation}
Again, this is uniform over $u \in B_R$ and $\abs{E} \leq E_{\max{}}$.

\bigskip

To show the statement of the lemma, given $R$, $0 < \gamma < \frac{1}{50}$, and $A$, we choose $E_{\max{}} = A$ above; then applying \eqref{eqn:genericmde_resultofcty} yields
\[
    \sup_{u \in B_R} \frac{1}{N} \int_{-A}^A \abs{\Tr(M_N(u,E+\ii N^{-\gamma})) - \Tr(M'_N(u,E+\ii N^{-\gamma}))} \diff E \leq (4A\kappa C''_{R,A}(1+R))N^{25\gamma-1}.
\]
This holds for $N$ large enough depending on $\kappa, R$, and $A$.
\end{proof}

%%%%%%%%%%%%%%%%%%%%%%%%%%%%%%%%%%%%%%%%%%%%%%%%%%%%%%%%%%%%
%%%%%%%%%%%%%%%%%%%%%%%%%%%%%%%%%%%%%%%%%%%%%%%%%%%%%%%%%%%%
%%%%%%%%
%%%%%%%%             Section: Elastic manifold
%%%%%%%%
%%%%%%%%%%%%%%%%%%%%%%%%%%%%%%%%%%%%%%%%%%%%%%%%%%%%%%%%%%%%
%%%%%%%%%%%%%%%%%%%%%%%%%%%%%%%%%%%%%%%%%%%%%%%%%%%%%%%%%%%%

\section{Elastic manifold}

%%%%%%%%%%%%%%%%%%%%%%%%%%%%%%%%%%%%%%%%%%%%%%%%%%%%%%%%%%%%
%%%%%%%%             Subsection: Establishing the variational formula
%%%%%%%%%%%%%%%%%%%%%%%%%%%%%%%%%%%%%%%%%%%%%%%%%%%%%%%%%%%%

\subsection{Establishing the variational formula.}\ In this subsection we establish a variational formula for complexity, given over $\R^{L^d}$. In the next subsection we analyze this variational problem, first by reducing it to a variational problem over $\R$ and then by relating it to the variational problem we analyze in depth for the soft-spins model.

In this subsection, we frequently reference notation and results in the companion paper \cite{BenBouMcK2021I}. Let
\[
	J = 2\sqrt{B''(0)}
\]
which will be an important scaling factor. For each $u \in \R^{L^d}$, define
\begin{equation}
\label{eqn:FLDdefau}
    a(u) = (-t_0\Delta + \diag(u) + \mu_0\Id_{L^d \times L^d}) \in \R^{L^d \times L^d}, 
\end{equation}
and for each $z \in \mathbb{H}$ let $m_\infty(u,z) = (m_\infty(u,z)_1, \ldots, m_\infty(u,z)_{L^d}) \in \C^{L^d}$ be the unique solution to
\begin{equation}
\label{eqn:kroneckermde_diagonal}
    m_\infty(u,z) = \diag[(a(u)-J^2m_\infty(u,z)-z\Id)^{-1}] \quad \text{such that} \quad \im m_\infty(u,z) > 0 \text{ componentwise}.
\end{equation}
(Recall we identify vectors with diagonal matrices; we will prove existence and uniqueness during the proof using the methods of Erd\H{o}s and co-authors.) Let $\mu_\infty(u)$ (which also depends on $L$, $d$, $t_0$, and $\mu_0$) be the measure whose Stieltjes transform is given by
\[
    \int \frac{\mu_\infty(u,\diff s)}{s-z} = \frac{1}{L^d} \sum_{i=1}^{L^d} m_\infty(u,z)_i.
\]

\begin{thm}
\label{thm:FLDcomplexity}
The probability measure $\mu_\infty(u)$ admits a bounded, compactly supported density $\mu_\infty(u,\cdot)$ with respect to Lebesgue measure, and
\begin{equation}
\label{eqn:FLD_tot}
    \Sigma(\mu_0) = \Sigma(\mu_0, L, d, t_0) := \lim_{N \to \infty} \frac{1}{NL^d}\log \E[\mc{N}_{\text{tot}}]
    = -\frac{1}{L^d} \log( \det(\mu_0 \Id - t_0\Delta)) + \sup_{u \in \R^{L^d}} \left\{\int \log\abs{\cdot} \rd\mu_\infty(u,\cdot) - \frac{\|u\|^2}{2J^2L^d} \right\}.
\end{equation}
Furthermore, if we define the set 
\begin{equation}
\label{eqn:FLD_G}
    \mc{G} = \{u \in \R^{L^d} : \mu_\infty(u)((-\infty,0)) = 0\}
\end{equation}
of $u$ values whose corresponding measures $\mu_\infty(u)$ are supported in the right half-line, we have
\begin{equation}
\label{eqn:FLD_min}
    \Sigma_{\text{st}}(\mu_0) = \Sigma(\mu, L, d, t_0) := \limsup_{N \to \infty} \frac{1}{NL^d} \log \E[\mc{N}_{\text{st}}] 
    = -\frac{1}{L^d} \log(\det(\mu_0 \Id - t_0\Delta)) + \sup_{u \in \mc{G}} \left\{\int \log\abs{\cdot}\rd \mu_\infty(u,\cdot) - \frac{\|u\|^2}{2J^2L^d} \right\}.
\end{equation}
The suprema in \eqref{eqn:FLD_tot} and \eqref{eqn:FLD_min} are achieved (possibly not uniquely).
\end{thm}

We first build the relevant block matrix. With $a(u)$ as in \eqref{eqn:FLDdefau}, let
\[
    A_N(u) = a(u) \otimes \Id_{N \times N}.
\]
For each $N$, let $(X_i)_{i=1}^{L^d}$ be a collection of $L^d$ independent $N \times N$ matrices, each distributed as $J$ times a GOE matrix, with the normalization $\E[((X_i)_{jk})^2] = J^2 \frac{1+\delta_{jk}}{N}$. Let
\begin{align*}
    W_N &= \sum_{i=1}^{L^d} E_{ii} \otimes X_i, \\
    H_N(u) &= A_N(u) + W_N.
\end{align*}
This matrix is in the class of ``block-diagonal Gaussian matrices'' studied in \cite[Corollary 1.9]{BenBouMcK2021I}.  It appears naturally in the Kac-Rice formula, but we also introduce a slight modification that is easier to work with. Let
\begin{align*}
    \widetilde{W_N} &= \left(\mathbf{1} - \frac{1}{\sqrt{2}}\Id\right) \odot W_N, \\
    \widetilde{H_N(u)} &= A_N(u) + \widetilde{W_N},
\end{align*}
where $\mathbf{1}$ is the matrix of all ones and $\odot$ is the entrywise product, i.e., $\widetilde{W_N}$ is just $W_N$ rescaled to make the variances $J^2/N$ both on and off the diagonal, coupled appropriately with $W_N$.

Now we simplify the MDE. It is known \cite{HelFarSpe2007} that, whenever $\mc{S} : \C^{L^d \times L^d} \to \C^{L^d \times L^d}$ is a linear operator that is self-adjoint with respect to the inner product $\ip{R,T} = \Tr(R^\ast T)$ and that preserves the cone of positive-semi-definite matrices, the problem
\begin{equation}
\label{eqn:genericmde}
    -M^{-1}(u,z) = z\Id - a(u) + \mc{S}[M(u,z)] \quad \text{subject to} \quad \im M(u,z) > 0
\end{equation}
has a unique solution $M(u,z) \in \C^{L^d \times L^d}$ for each $z \in \mathbb{H}$ and $u \in \R^{L^d}$. We will consider this problem with two choices of operator $\mc{S}$:
\begin{align*}
    \mc{S}_N[T] = J^2 \frac{N+1}{N} \diag(T) \quad \text{induces} \quad M_N(u,z), \\
    \mc{S}_\infty[T] = J^2 \diag(T) \quad \text{induces} \quad M_\infty(u,z).
\end{align*}
Write $\ms{S}_i$ (respectively, $\widetilde{\ms{S}_i}$) for the ``stability operators'' of \cite[(1.15)]{BenBouMcK2021I} corresponding to the matrix $H_N(u)$ (respectively, $\widetilde{H_N(u)}$). Then we can compute
\[
    \ms{S}_i[\mathbf{r}] = J^2 \sum_{k=1}^N \frac{1+\delta_{ik}}{N}\diag(r_k), \qquad \widetilde{\ms{S}_i[\mathbf{r}]} = \frac{J^2}{N} \sum_{k=1}^N \diag(r_k).
\]
Thus the choices $\mathbf{m}(u,z) = (M_N(u,z), \ldots, M_N(u,z))$ and $\widetilde{\mathbf{m}}(u,z) = (M_\infty(u,z), \ldots, M_\infty(u,z))$ exhibit solutions to the block MDE \cite[(1.16)]{BenBouMcK2021I} for the matrices $H_N(u)$ and $\widetilde{H_N(u)}$, respectively. That is, the measure $\mu_N(u)$ that appears in the local laws for $H_N(u)$ has Stieltjes transform
\[
    \int \frac{\mu_N(u,\diff s)}{s - z} = \frac{1}{L^d} \Tr(M_N(u,z)),
\]
and the measure that appears in the local laws for $\widetilde{H_N(u)}$ is actually independent of $N$: We call it $\mu_\infty(u)$, and its Stieltjes transform is given by
\[
    \int \frac{\mu_\infty(u,\diff s)}{s - z} = \frac{1}{L^d} \Tr(M_\infty(u,z)).
\]

\begin{lem}
\label{lem:kroneckermde_Ldbound}
Both $\mu_N(u)$ and $\mu_\infty(u)$ admit densities $\mu_N(u,\cdot)$ and $\mu_\infty(u,\cdot)$ with respect to Lebesgue measure, and
\[
    \sup_{u \in \R^m, z \in \mathbb{H}, N \in \N} \max \left\{ \|M_N(u,z)\|, \|M_\infty(u,z)\|, \|\mu_N(u,\cdot)\|_{L^\infty}, \|\mu_\infty(u,\cdot)\|_{L^\infty} \right\} \leq \sqrt{L^d}/J.
\]
\end{lem}
\begin{proof}
The following arguments are due to L\'{a}szl\'{o} Erd\H{o}s and Torben Kr\"{u}ger.
We prove the result for $M_N$ and $\mu_N$; the proofs for $M_\infty$ and $\mu_\infty$ are similar. By taking the imaginary part of \eqref{eqn:genericmde} and multiplying on the left by $M_N(u,z)^\ast$ and on the right by $M_N(u,z)$, then taking the diagonal of both sides, we obtain
\begin{align}
\label{eqn:im_mde}
\begin{split}
    \im \diag(M_N(u,z)) &= \diag\left(M_N(u,z)^\ast \left(\eta + J^2 \frac{N+1}{N}\im(\diag(M_N(u,z)))\right) M_N(u,z) \right) \\
    &= F_N(u,z)\left( \eta+ J^2 \frac{N+1}{N}\im(\diag(M_N(u,z))) \right),
\end{split}
\end{align}
where $F_N(u,z) \in \R^{L^d \times L^d}$ is given elementwise by $F_N(u,z)_{ij} = \abs{(M_N(u,z))_{ij}}^2$. By transposing the MDE \eqref{eqn:genericmde} and using the fact that $a(u)$ is symmetric, we see that $M_N(u,z)$ is symmetric (but not Hermitian) as well. Hence $F_N(u,z)$ is a real symmetric matrix with nonnegative entries. The inner product of \eqref{eqn:im_mde} with the Perron-Frobenius eigenvector of $F_N$ then gives $\|F_N\| \leq \frac{1}{J^2} \frac{N}{N+1}$, since $\im(\diag(M_N(u,z))$ has all positive components. Thus
\[
    \|M_N(u,z)\|^2 \leq \Tr(M_N(u,z)^\ast M_N(u,z)) = \ip{1, F_N(u,z) 1} \leq \frac{1}{J^2} \frac{N}{N+1}L^d.
\]
Now for any interval $[a,b]$ we have
\[
    \mu_N(u)([a,b]) + \frac{\mu_N(u)(\{a\}) + \mu_N(u)(\{b\})}{2} = \lim_{\eta \downarrow 0} \frac{1}{\pi} \int_a^b \im \frac{1}{L^d} \Tr(M_N(u,E+\ii \eta)) \diff E \leq (b-a)\frac{\sqrt{L^d}}{J\pi}.
\]
By standard continuity arguments we extend this to $\mu_N(u)(A) \leq \abs{A} \frac{\sqrt{L^d}}{J\pi}$ for any Borel set $A$; this implies that $\mu_N$ is absolutely continuous with respect to Lebesgue measure with a density that is pointwise bounded by $\frac{\sqrt{L^d}}{J\pi}$. 
\end{proof}

\begin{lem}
\label{lem:fld_reference_muinfty}
For every $R$, there exists $\epsilon > 0$ such that
\begin{equation}
\label{eqn:FLD_Wasserstein}
    \sup_{u \in B_R} {\rm W}_1(\E[\hat{\mu}_{H_N(u)}], \mu_\infty(u)) \leq N^{-\epsilon}.
\end{equation}
\end{lem}
\begin{proof}
First, note that
\begin{equation}
\label{eqn:FLD_bddmean}
    \sup_N\|A_N(u)\| = \|a(u) \otimes \Id\| = \|a(u)\| \leq t\|\Delta\| + \|u\| + \mu < \infty.
\end{equation}
Along with Lemma \ref{lem:kroneckermde_Ldbound}, this verifies the assumptions of \cite[Corollary 1.9]{BenBouMcK2021I}, the proof of which shows that \eqref{eqn:FLD_Wasserstein} holds with $\mu_\infty(u)$ replaced by $\mu_N(u)$ (the result is locally uniform in $u$ since all the assumptions are). To compare $\mu_N(u)$ and $\mu_\infty(u)$, we use the result of Lemma \ref{lem:generic_distance_stability} (with the choices $P_N \equiv L^d$, $\mc{S}'_N = \mc{S}_\infty$ as above) and follow the proof of \cite[Proposition 3.1]{BenBouMcK2021I}.
\end{proof}

\begin{lem}
\label{lem:fld_detsublinear}
There exists $C > 0$ with
\[
    \E[\abs{\det(H_N(u))}] \leq (C\max(\|u\|,1))^N.
\]
\end{lem}
\begin{proof}
Deterministically,
\[
    \abs{\det(H_N(u))} \leq \|H_N(u)\|^N \leq (\|W_N\| + \|A_N(u)\|)^N \leq (2\|W_N\|)^N + (2\|A_N(u)\|)^N.
\]
In the proof of \cite[Corollary 1.9]{BenBouMcK2021I}, we showed $\P(\|W_N\| \geq t) \leq e^{-cN\max(0,t-C)}$ for some constants $c, C$, which implies $\E[\|W_N\|^N] \leq e^{CN}$. With the estimate on $\|A_N(u)\|$ from \eqref{eqn:FLD_bddmean}, this suffices.
\end{proof}

\begin{lem}
\label{lem:fld_NlogN}
For every $R$ and every $\epsilon > 0$, we have
\[
    \lim_{N \to \infty} \frac{1}{N\log N} \log \left[ \sup_{u \in B_R} \P(d_{\textup{BL}}(\hat{\mu}_{H_N(u)}, \mu_\infty(u)) > \epsilon) \right] = -\infty.
\]
\end{lem}
\begin{proof}
The laws of the entries of $\sqrt{N}H_N(u)$ satisfy the log-Sobolev inequality with a uniform constant, since they are Gaussians. (If they are degenerate, we recall that a delta mass satisfies log-Sobolev with any constant.) This is true uniformly over $u \in \R^{L^d}$, since $u$ only affects the mean, and translating measures preserves log-Sobolev with the same constant. Hence \cite[Theorem 1.5]{GuiZei2000} give, for some constants $C_1$ and $C_2$, 
\[
    \sup_{u \in \R^{L^d}} \P(d_{\textup{BL}}(\hat{\mu}_{H_N(u)}, \E[\hat{\mu}_{H_N(u)}]) > \epsilon) \leq \frac{C_1}{\epsilon^{3/2}} \exp\left(-\frac{C_2}{2L^d}N^2\epsilon^5\right).
\]
To relate $\E[\hat{\mu}_{H_N(u)}]$ to $\mu_\infty(u)$, we use Lemma \ref{lem:fld_reference_muinfty}. 
\end{proof}

\begin{lem}
\label{lem:FLD_nooutliers}
For every $\epsilon > 0$ and $R > 0$, we have
\begin{equation}
\label{eqn:FLD_nooutliers}
    \lim_{N \to \infty} \inf_{u \in B_R} \P(\Spec(H_N(u)) \subset [\mathtt{l}(\mu_\infty(u))-\epsilon, \mathtt{r}(\mu_\infty(u)) + \epsilon]) = 1.
\end{equation}
and in fact the extreme eigenvalues of $H_N(u)$ converge in probability to the endpoints of $\mu_\infty(u)$.
\end{lem}
\begin{proof}
The local law \cite[Theorem 2.4, Remark 2.5(v)]{AltErdKruNem2019} tells us that, for every $\epsilon$ and $R$, there exists $C_{\epsilon,R}$ such that
\begin{equation}
\label{eqn:FLD_Htilde}
    \inf_{u \in B_R} \P\left(\Spec(\widetilde{H_N(u)}) \subset \left[\mathtt{l}(\mu_\infty(u))-\frac{\epsilon}{2}, \mathtt{r}(\mu_\infty(u))+\frac{\epsilon}{2}\right]\right) \geq 1 - \frac{C_{\epsilon,R}}{N^{100}}.
\end{equation}
(We can take the infimum over $u \in B_R$ because the local-law estimates are uniform over all models possessing the same ``model parameters,'' see the remarks just before Theorem 2.4 there. Our model parameters depend on $u$ but can be taken uniformly over $u \in B_R$, for example because $\sup_{u \in B_R} \|A_N(u)\| < \infty$.)

Notice that 
\[
    \Delta_N = H_N(u) - \widetilde{H_N(u)} = W_N - \widetilde{W_N}
\]
is a diagonal matrix with independent Gaussian entries of variance $J^2/N$ that does not depend on $u$. Thus
\begin{equation}
\label{eqn:FLD_DeltaN}
    \sup_{u \in \R^{L^d}} \P\left(\abs{\lambda_{\min{}}(H_N(u)) - \lambda_{\min{}}(\widetilde{H_N(u)})} \geq \frac{\epsilon}{2}\right) \leq \P\left(\|\Delta_N\| \geq \frac{\epsilon}{2}\right) \leq \frac{2J\sqrt{N}L^d}{\epsilon}e^{-N\epsilon^2/(8J^2)}
\end{equation}
and similarly for $\lambda_{\min{}}$. Since 
\[
    \P(\lambda_{\max{}}(H_N(u)) \geq \mathtt{r}(\mu_\infty(u)) + \epsilon) \leq \P\left(\lambda_{\max{}}(\widetilde{H_N(u)}) \geq \mathtt{r}(\mu_\infty(u))+\frac{\epsilon}{2}\right) + \P\left(\abs{\lambda_{\max{}}(H_N(u)) - \lambda_{\max{}}(\widetilde{H_N(u)})} \geq \frac{\epsilon}{2}\right),
\]
and similarly for $\lambda_{\min{}}$, \eqref{eqn:FLD_Htilde} and \eqref{eqn:FLD_DeltaN} imply \eqref{eqn:FLD_nooutliers}.

For the other inequality, namely that $\liminf_{N \to \infty} \lambda_{\max{}}(H_N(u)) \geq \mathtt{r}(\mu_\infty(u))$ in probability, we note that $\hat{\mu}_{H_N(u)}$ concentrates about $\mu_\infty(u)$ in the sense of Lemma \ref{lem:fld_NlogN}. The smallest eigenvalue is handled similarly.%, so $\liminf_{N \to \infty} \lambda_{\max{}}(H_N(u)) \geq \mathtt{r}(\mu_\infty(u))$ in probability; the reverse inequality follows from \eqref{eqn:FLD_nooutliers}, and similarly for the left endpoint. 
\end{proof}

\begin{lem}
\label{lem:FLD_topology}
With $\mc{G}_{+\epsilon}$ as defined in \cite[(4.5)]{BenBouMcK2021I} and $\mc{G}$ as defined in \eqref{eqn:FLD_G}, we have that each $\mc{G}_{+\epsilon}$ is convex, that $\mc{G}_{+1}$ has positive measure, and that 
\[
    \overline{\bigcup_{\epsilon > 0} \mc{G}_{+\epsilon}} = \mc{G}.
\]
\end{lem}
\begin{proof}
Whenever $u, v \in \R^{L^d}$ and $t \in [0,1]$, one can check $H_N(tu+(1-t)v) = tH_N(u) + (1-t)H_N(v)$; thus
\[
    \lambda_{\min{}}(H_N(tu+(1-t)v)) \geq t\lambda_{\min{}}(H_N(u)) + (1-t)\lambda_{\min{}}(H_N(v))
\]
almost surely, and by letting $N \to \infty$ and applying the convergence in probability of Lemma \ref{lem:FLD_nooutliers}  we obtain $
    \mathtt{l}(\mu_\infty(tu+(1-t)v)) \geq t\mathtt{l}(\mu_\infty(u)) + (1-t)\mathtt{l}(\mu_\infty(v)).
$
Hence each $\mc{G}_{+\epsilon}$ is convex. 

Since $-t\Delta$ and $\mu\Id$ are positive semidefinite, 
\[
    \lambda_{\min{}}(A_N(u)) = \lambda_{\min{}}(a(u) \otimes \Id) = \lambda_{\min{}}(a(u)) = \lambda_{\min{}}(-t\Delta + \diag(u) + \mu\Id) \geq \min (u_1, \ldots, u_{L^d}).
\]
On the other hand,
\[
    \lambda_{\min{}}(W_N) = \lambda_{\min{}}\left(\sum_{i=1}^{L^d} E_{ii} \otimes X_i\right) = \min_{i=1}^{L^d}(\lambda_{\min{}}(X_i))
\]
which tends almost surely to $-2J$ in our normalization. Thus
\[
    \liminf_{N \to \infty} \lambda_{\min{}}(H_N(u)) \geq \min(u_1, \ldots, u_{L^d}) - 2J.
\]
which, combined with the convergence in probability of Lemma \ref{lem:FLD_nooutliers}, shows that $\mc{G}_{+1}$ has positive measure.

Finally, we note that the inclusion $\cup_{\epsilon > 0} \mc{G}_{+\epsilon} \subset \mc{G}$ is clear, and that $\mc{G}$ is closed by \cite[Lemma 4.6]{BenBouMcK2021I}. To show the reverse inclusion, write $\vec{1} \in \R^{L^d}$ for the vector of all ones; then it is easy to check $H_N(u + \delta \vec{1}) = H_N(u) + \delta \Id$, so by the convergence in probability of Lemma \ref{lem:FLD_nooutliers} we have
$
    \mathtt{l}(\mu_\infty(u+\delta \vec{1})) = \mathtt{l}(\mu_\infty(u)) + \delta.
$
This completes the proof.
\end{proof}

\begin{prop}
We have
\begin{equation}
\label{eqn:FLD_detcon}
    \lim_{N \to \infty} \frac{1}{NL^d} \log \int_{\R^{L^d}} e^{-N\frac{\|u\|^2}{2J^2}} \E[\abs{\det(H_N(u))}] \diff u = \sup_{u \in \R^{L^d}} \left\{ \int \log\abs{\lambda} \mu_\infty(u,\lambda) \diff \lambda - \frac{\|u\|^2}{2J^2L^d}\right\}
\end{equation}
and
\begin{equation}
\label{eqn:FLD_detconmin}
    \limsup_{N \to \infty} \frac{1}{NL^d} \log \int_{\R^{L^d}} e^{-N\frac{\|u\|^2}{2J^2}} \E[\abs{\det(H_N(u))}\mathbf{1}_{H_N(u) \geq 0}] \diff u = \sup_{u \in \mc{G}} \left\{ \int \log\abs{\lambda} \mu_\infty(u,\lambda) \diff \lambda - \frac{\|u\|^2}{2J^2L^d}\right\}
\end{equation}
where $\mc{G}$ is defined in \eqref{eqn:FLD_G}.
\end{prop}
\begin{proof}
For \eqref{eqn:FLD_detcon}, we apply \cite[Theorem 4.1]{BenBouMcK2021I} with $\alpha = \frac{1}{2J^2L^d}$, $p = 0$, and $\mf{D} = \R^{L^d}$. (Technically, we are applying this theorem with $N$ there replaced by $NL^d$ here, which is the size of $H_N$; this is why $\alpha$ is $\frac{1}{2J^2L^d}$ and not $\frac{1}{2J^2}$.) We checked the conditions of this theorem in \cite[Corollary 1.9]{BenBouMcK2021I} and Lemmas \ref{lem:fld_reference_muinfty} and \ref{lem:fld_detsublinear}. (All the results are locally uniform in $u$ because all the parameters of the random matrices are.) For \eqref{eqn:FLD_detconmin}, we apply \cite[Theorem 4.5]{BenBouMcK2021I} with $\alpha = \frac{1}{2J^2L^d}$, $p = 0$, and $\mf{D} = \R^{L^d}$. We checked the conditions for this result in Lemmas \ref{lem:fld_NlogN}, \ref{lem:FLD_nooutliers}, and \ref{lem:FLD_topology}. 
\end{proof}

\begin{proof}[Proof of Theorem \ref{thm:FLDcomplexity}]
Kac-Rice computations in \cite{FyoLeD2020} show,  exactly for finite $N$, 
\begin{align*}
    \frac{1}{NL^d} \log \E[\mc{N}_{\text{tot}}] &= -\frac{1}{L^d}\log(\det(u_0-t_0\Delta)) + \frac{1}{NL^d} \log \int_{\R^{L^d}} \frac{e^{-N\frac{\|u\|^2}{2J^2}}}{(\sqrt{2\pi J^2/N})^{L^d}} \E[\abs{\det(H_N(u))}] \diff u, \\
    \frac{1}{NL^d} \log \E[\mc{N}_{\text{st}}] &= -\frac{1}{L^d}\log(\det(u_0-t_0\Delta)) + \frac{1}{NL^d} \log \int_{\R^{L^d}} \frac{e^{-N\frac{\|u\|^2}{2J^2}}}{(\sqrt{2\pi J^2/N})^{L^d}} \E[\abs{\det(H_N(u))}\mathbf{1}_{H_N(u) \geq 0}] \diff u,
\end{align*}
where $H_N(u)$ is as above. Then we apply the above Proposition.
\end{proof}

%%%%%%%%%%%%%%%%%%%%%%%%%%%%%%%%%%%%%%%%%%%%%%%%%%%%%%%%%%%%
%%%%%%%%             Subsection: Analyzing the variational formula
%%%%%%%%%%%%%%%%%%%%%%%%%%%%%%%%%%%%%%%%%%%%%%%%%%%%%%%%%%%%

\subsection{Analyzing the variational formula.}\ The following  concavity property is the key reason the complexity thresholds can be calculated explicitly,  from the variational formulas appearing in the previous section.

\begin{prop}
\label{prop:FLDconcavity}
The function
\[
    \mc{S}[u] = \int_\R \log\abs{\lambda} \mu_\infty(u,\lambda) \diff \lambda - \frac{\|u\|^2}{2J^2L^d}
\]
is concave.
\end{prop}
\begin{proof}
We assume $d=1$ below, the general case requiring only notational change of $L$ into $L^d$. The MDE for our problem, namely \eqref{eqn:genericmde} with the choice $\mc{S}[T] = J^2\diag(T)$, has a matrix solution $M(u,z) = M_\infty(u,z)$ (we now drop the $\infty$ to save notation). The problem can be reduced to a vector MDE for $\vec{m}(u,z) = \diag(M(u,z)) =: (m_1(u,z), \ldots, m_L(u,z))$ by taking the diagonal of both sides of \eqref{eqn:genericmde}. (In fact, $M(u,z)$ can be reconstructed from knowledge of $\vec{m}(u,z)$ via \eqref{eqn:genericmde}.) The diagonal MDE takes the form
\begin{equation}\label{eqn:MDEdiag}
-\diag(m_1,\dots,m_L)=\diag[(z-\mu+t\Delta-\diag(u_1,\dots,u_L)+J^2\diag(m_1,\dots,m_L))^{-1}].
\end{equation}
We denote $\partial_k=\partial_{u_k}$, and write $m=\frac{1}{L}\sum_{1}^L m_i$ for the Stieltjes transform of $\mu_\infty$. The first essential observation is
\begin{equation}\label{eqn:diffMDE}
\frac{\rd}{\rd u_k} (-Lm)=\frac{\rd}{\rd z} m_k.
\end{equation}
Indeed, for any invertible matrix $B$, we have 
$$
(B^{-1})_{kk}={\rm Tr}(B^{-1}|e_k\rangle\langle e_k|)=\partial_{u=0}\log \det(\Id+u B^{-1}|e_k\rangle\langle e_k|)=\partial_{u=0}\log \det(B+u |e_k\rangle\langle e_k|),
$$
so that, denoting 
$B(z,u,\vec{m})=z-\mu+t\Delta-\diag(u_1,\dots,u_L)+J^2\diag(m_1,\dots,m_L)$, we have
\begin{align*}
\frac{\rd}{\rd u_k}\log\det B&=\partial_k\log\det B+J^2 \sum_j\partial_{m_j}\log\det B\cdot \partial_k m_j=m_k- J^2\sum_j m_j\partial_k m_j,\\
\frac{\rd}{\rd z}\log\det B&=\partial_z\log\det B+J^2\sum_j\partial_{m_j}\log\det B\cdot \partial_z m_j=-Lm-J^2\sum_j m_j\partial_z m_j,
\end{align*}
i.e.
\begin{align*}
m_k&=\frac{\rd}{\rd u_k}(\log\det B +\frac{J^2}{2}\sum_jm_j^2),\\
-Lm&=\frac{\rd}{\rd z}(\log\det B +\frac{J^2}{2}\sum_jm_j^2).
\end{align*}
We conclude that $\frac{\rd}{\rd u_k} (-Lm)=\frac{\rd}{\rd z} m_k$.

\medskip

With \eqref{eqn:diffMDE}, we obtain
\begin{equation}\label{eqn:key2}
\partial_k \int \log\abs{\lambda} \mu_\infty(u,\lambda)=\frac{1}{\pi L} \int \log\abs{\lambda} \partial_\lambda\im m_k(\lambda+\ii 0^+)=-\frac{1}{\pi L}\int \frac{1}{\lambda} \im m_k(\lambda+\ii 0^+)=-\frac{1}{L}\re m_k(\ii 0^+).
\end{equation}
Now, from \eqref{eqn:MDEdiag}, we obtain
$$
\partial_k (m_1,\ldots,m_L) =M(E_k+J^2\diag(\partial_k m_1,\dots,\partial_k m_L))M=R(J^2\partial_k(m_1,\dots,m_L)+e_k)^T
$$
where the matrix $E_k$ (respectively, the vector $e_k$) is $0$'s except $1$ at position $(k,k)$ (respectively, position $k$), and where $R = R(u,z)$ is the linear operator defined by $(Rv)_i=\sum (M_{ij})^2v_j$, with $M=M(u,z)$ the MDE solution matrix. Thus we have
$$
(1-J^2R)(\partial_k \vec{m})=R(e_{k})^T,
$$
from which
\[
	\partial_k \vec{m} = \frac{1}{J^2}(1-J^2R)^{-1}(J^2R-1+1)(e_k)^T = \frac{1}{J^2}(-\Id + (1-J^2R)^{-1})(e_k)^T.
\]
Taking the $j$th component of both sides, we get the scalar equation
\[
	\partial_k m_j = \frac{1}{J^2}(-\Id+(1-J^2R)^{-1})_{jk}.
\]
Together with \eqref{eqn:key2}, this gives
\begin{equation}\label{eqn:Erd}
	\nabla^2_u \int \log\abs{\lambda-\mu} \mu_\infty(u,\lambda) = \frac{1}{J^2L}(\Id - \re[(1-J^2R)^{-1}]).
\end{equation}
Lemma \ref{lem:erdoskruger} below, due to L\'{a}szl\'{o} Erd\H{o}s and Torben Kr\"{u}ger,  shows that $\re[(1-J^2R)^{-1}] \geq 0$ in the sense of quadratic forms. Along with \eqref{eqn:Erd}, this gives concavity of $\int_\R \log\abs{\lambda} \mu_\infty(u,\lambda) \diff \lambda - \frac{\|u\|^2}{2J^2L^d}$ in $u$. 
\end{proof}

\medskip

\begin{lem}
\label{lem:erdoskruger}
For each $u \in \R^{L^d}$ and each $z \in \mathbb{H}$, let $R(u,z) \in \C^{L^d \times L^d}$ be defined elementwise by 
\[
    R(u,z)_{jk} = (M(u,z)_{jk})^2 =: e^{\ii \theta(u,z)_{jk}} \abs{M(u,z)_{jk}}^2,
\]
where $M(u,z) = M_\infty(u,z)$. Then for every $u \in \R^{L^d}$, every $z \in \mathbb{H}$, and every nonzero vector $v$ we have
\[
    \re\ip{v,(\Id-J^2R(u,z))v} > 0.
\]
In particular, for any $w$, written as $(1-J^2R)v$, we have
\[
    \re \langle w, (1-J^2R)^{-1}w \rangle=\re \langle (1-J^2R)v, v \rangle > 0.
\]
\end{lem}
\begin{proof}
Consider the matrix $F(u,z) \in \R^{L^d \times L^d}$ defined elementwise by $F(u,z)_{jk} = \abs{M(u,z)_{jk}}^2$. The proof of Lemma \ref{lem:kroneckermde_Ldbound} shows that $\sup_{u \in \R^{L^d}, z \in \mathbb{H}} \|F(u,z)\| \leq \frac{1}{J^2}$. Given $v \in \C^{L^d}$, write $\abs{v} = (\abs{v_1}, \ldots, \abs{v_{L^d}})$; then 
\begin{align*}
    \re\ip{v,(\Id-J^2R(u,z))v} &\geq \re \sum_{j=1}^{L^d} \left(1-J^2\abs{M(u,z)_{jj}}^2e^{\ii\theta(\xi,z)_{jj}}\right)\abs{v_j}^2 - J^2\sum_{j \neq k} \abs{M(u,z)_{jk}}^2\abs{v_jv_k} \\
    &= \sum_{j=1}^{L^d} \left(1- J^2\abs{M(u,z)_{jj}}^2 \cos(\theta(u,z)_{jj})\right) \abs{v_j}^2 - J^2\sum_{j \neq k} \abs{M(u,z)_{jk}}^2\abs{v_jv_k} \\
    &= \ip{\abs{v},\left(\Id-J^2F(u,z)\right)\abs{v}} + J^2\sum_{j=1}^{L^d} (1-\cos(\theta(u,z)_{jj})) \abs{M(u,z)_{jj}}^2 \abs{v_j}^2 \\
    &\geq \ip{\abs{v},\left(\Id- J^2F(u,z)\right)\abs{v}} + J^2\sum_{j=1}^{L^d} 2(\im (M(u,z)_{jj}))^2 \abs{v_j}^2.
\end{align*}
The first term is nonnegative since $\|F(u,z)\| \leq \frac{1}{J^2}$, so we have
\[
    \re\ip{v,(\Id-J^2R(u,z))v} \geq 2J^2\left(\min_{j=1}^{L^d} \im (M(u,z)_{jj})\right)^2 \|v\|^2.
\]
But $\im (M(u,z)_{jj})$ is the $(j,j)$ entry of the matrix $\im M(u,z)$, which is (strictly) positive definite by the definition of the MDE.
\end{proof}

\begin{prop}
\label{prop:FLDrestricttodiagonal}
$\mc{S}$ is maximized on the diagonal of $\R^d$, i.e., 
\[
    \sup_{u \in \R^{L^d}} \mc{S}[u] = \sup_{u \in \R} \mc{S}[(u,\ldots,u)].
\]
\end{prop}
\begin{proof}
It suffices to show that the set of maximizers
\[
    \mc{M} = \left\{u \in \R^{L^d} : \mc{S}[u] = \sup_{v \in \R^{L^d}} \mc{S}[v]\right\}
\]
intersects the diagonal. First, $\mc{M}$ is nonempty, since (by \cite[Lemma 4.4]{BenBouMcK2021I}) $\lim_{\|u\| \to +\infty} \mc{S}[u] = -\infty$ and $\mc{S}$ is continuous. Furthermore, $\mc{M}$ is closed under the operation ``permute the coordinates (which are indexed by lattice points) with a permutation that is also a translation of the periodic lattice,'' since such permutations preserve $a(u)$ in \eqref{eqn:FLDdefau} and thus $\mu_\infty(u)$. Finally, $\mc{M}$ is convex, since $\mc{S}[u]$ is concave.

Given $u \in \mc{M}$, its images under all possible lattice translations are thus all in $\mc{M}$, so the average of all these points (which is in their convex hull) is in $\mc{M}$. Since the lattice is periodic (i.e., translations are in bijection with lattice sites), this average is on the diagonal.
\end{proof}

\begin{proof}[Proof of Theorem \ref{thm:variationalFLDcomplexity}]
Using Proposition \ref{prop:FLDrestricttodiagonal} to restrict the variational problem from Theorem \ref{thm:FLDcomplexity} to the diagonal, we have
\[
    \Sigma(\mu_0) = -\frac{1}{L^d} \log(\det(\mu_0 \Id_{L^d \times L^d} - t_0\Delta)) + \sup_{u \in \R} \left\{ \int_\R \log\abs{\lambda} \mu_\infty((u,\ldots,u),\lambda) \diff \lambda - \frac{u^2}{2J^2} \right\}
\]
and similarly for minima. One can check directly from the MDE that
\[
    \mu_\infty((u,\ldots,u),\lambda) = \mu_\infty((0,\ldots,0),\lambda-u),
\]
and in fact we have $\mu_\infty((0,\ldots,0)) = \rho_{\text{sc},J^2} \boxplus \hat{\mu}_{-t_0\Delta+\mu_0\Id}$. Indeed, by symmetry all the entries of $m_\infty(0,z)$ must be equal. If we denote by $m(z)$ their shared value (which is also the Stieltjes transform of $\mu_\infty((0,\ldots,0))$), then by taking the normalized trace in \eqref{eqn:kroneckermde_diagonal} we find that $m(z)$ satisfies the self-consistent equation
\[
    m(z) = \int \frac{\hat{\mu}_{-t_0\Delta + \mu_0\Id}(\diff s)}{s-z-J^2m(z)}.
\]
This \emph{Pastur relation} characterizes \cite{Pas1972} the Stieltjes transform $m(z)$ of $\rho_{\text{sc},J^2} \boxplus \hat{\mu}_{-t_0\Delta+\mu_0\Id}$. Exchanging $u$ and $-u$ gives \eqref{eqn:variationalFLDcomplexity}. 
\end{proof}

\begin{proof}[Proof of Theorem \ref{thm:betterFLDcomplexity}]
Since $-L^{-d}\log\det(\mu_0 - t_0\Delta) = -\int\log(\lambda) \hat{\mu}_{-t_0\Delta + \mu_0\Id}$, the variational problems given in \eqref{eqn:FLD_onedimensional} and \eqref{eqn:variationalFLDcomplexity} are exactly the variational problems analyzed for the soft spins model in \eqref{eqn:defsigmatot} and \eqref{eqn:defsigmamin}, identifying $\mu_D$ there with $\hat{\mu}_{-t_0\Delta + \mu_0\Id}$ here (which is gapped from zero since $\mu_0 > 0$) and $B''(0)$ there with $J^2$ here. The statement of Theorem \ref{thm:betterFLDcomplexity} follows from our analysis of that variational problem in the next section, since
\[
    \int_\R \frac{\hat{\mu}_{-t_0\Delta+\mu_0\Id}(\diff \lambda)}{\lambda^2} = \int_\R \frac{\hat{\mu}_{-t_0\Delta}(\diff \lambda)}{(\lambda+\mu_0)^2}
\]
is a strictly decreasing function of $\mu_0$, tending to $0$ as $\mu_0 \to +\infty$ and tending to $+\infty$ (since the Laplacian is singular) as $\mu_0 \downarrow 0$. This proves existence and uniqueness of the Larkin mass as claimed.
\end{proof}

\begin{rem}
Here we take $\Delta$ to be the lattice Laplacian, which is the classic choice in the elastic manifold, but as suggested in \cite{FyoLeD2020A} the same methods and proofs allow us to replace $\Delta$ everywhere with any symmetric negative semidefinite $L^d \times L^d$ matrix. For example, this allows for interactions beyond pairwise. 
\end{rem}

%%%%%%%%%%%%%%%%%%%%%%%%%%%%%%%%%%%%%%%%%%%%%%%%%%%%%%%%%%%%
%%%%%%%%%%%%%%%%%%%%%%%%%%%%%%%%%%%%%%%%%%%%%%%%%%%%%%%%%%%%
%%%%%%%%
%%%%%%%%             Section: Soft spins in an anisotropic well
%%%%%%%%
%%%%%%%%%%%%%%%%%%%%%%%%%%%%%%%%%%%%%%%%%%%%%%%%%%%%%%%%%%%%
%%%%%%%%%%%%%%%%%%%%%%%%%%%%%%%%%%%%%%%%%%%%%%%%%%%%%%%%%%%%

\section{Soft spins in an anisotropic well}

%%%%%%%%%%%%%%%%%%%%%%%%%%%%%%%%%%%%%%%%%%%%%%%%%%%%%%%%%%%%
%%%%%%%%             Subsection: Establishing the variational formula
%%%%%%%%%%%%%%%%%%%%%%%%%%%%%%%%%%%%%%%%%%%%%%%%%%%%%%%%%%%%

\subsection{Establishing the variational formula.}\
In this subsection we prove Theorem \ref{thm:asywell}, which establishes a variational formula for complexity. In the next subsection we analyze it to prove Theorem \ref{thm:softspins_threshold}.

The Kac-Rice formula \cite[Theorem 11.2.1]{AdlTay2007} gives
\begin{align}
\label{eqn:parabola_AT}
\begin{split}
    \E[\Crt_N^{\text{tot}}(\mc{H})] &= \int_{\R^N} \E[\abs{\det(\nabla^2 \mc{H}(\sigma))} \mid \nabla \mc{H}(\sigma) = 0] \phi_\sigma(0) \diff \sigma, \\
    \E[\Crt_N^{\text{min}}(\mc{H})] &= \int_{\R^N} \E[\abs{\det(\nabla^2 \mc{H}(\sigma))} \mathbf{1}_{\nabla^2 \mc{H}(\sigma) \geq 0} \mid \nabla \mc{H}(\sigma) = 0] \phi_\sigma(0) \diff \sigma,
\end{split}
\end{align}
where 
\[
    \phi_\sigma(0) = \frac{1}{(2\pi B'(0))^{N/2}} \exp\left(-\frac{1}{2B'(0)} \|D_N\sigma\|^2 \right)
\]
is the density of $\nabla \mc{H}(\sigma)$ at $0 \in \R^N$. (As stated, the Kac-Rice formula actually counts the mean number of critical points, not in all of $\R^N$, but in a compact subset $T$ of $\R^N$ satisfying some regularity assumptions; then the right-hand integrals in \eqref{eqn:parabola_AT} are over $T$ instead of $\R^N$. To obtain \eqref{eqn:parabola_AT} as written, we use this version of Kac-Rice for some nested sequence $(T_N)_{N=1}^\infty$ of compact sets whose union is $\R^N$ and apply monotone convergence on both sides.)

Since $V$ is isotropic, for each $\sigma$ we have that $(\nabla^2 V(\sigma),V(\sigma))$ is independent of $\nabla V(\sigma)$; hence for each $\sigma$ we also have that $(\nabla^2 \mc{H}(\sigma),\mc{H}(\sigma))$ is independent of $\nabla \mc{H}(\sigma)$. In fact, since $V$ is isotropic the distribution of $\nabla^2 V(\sigma)$ is independent of $\sigma$; and by computation
\[
    \nabla^2 \frac{1}{2}\ip{\sigma,D_N\sigma} = D_N
\]
is independent of $\sigma$ as well. Thus 
\begin{equation}
\label{eqn:parabolakacrice}
\begin{split}
    \E[\Crt_N^{\text{tot}}(\mc{H})] &= \int_{\R^N} \E[\abs{\det(\nabla^2 \mc{H}(\sigma))} \mid \nabla \mc{H}(\sigma) = 0] \phi_\sigma(0) \diff \sigma = \E\left[\abs{\det(\nabla^2 \mc{H}(\mathbf{0}))}\right] \int_{\R^N} \phi_\sigma(0) \diff \sigma \\
    &= \frac{1}{\det(D_N)} \E\left[\abs{\det(\nabla^2 \mc{H}(\mathbf{0}))}\right], \\
    \E[\Crt^{\min{}}_N(\mc{H})] &= \frac{1}{\det(D_N)} \E\left[ \abs{\det(\nabla^2 \mc{H}(\mathbf{0}))} \mathbf{1}_{\nabla^2 \mc{H}(\mathbf{0}) \geq 0} \right].
\end{split}
\end{equation}
Since the eigenvalues of $D_N$ are gapped away from zero and from infinity, uniformly in $N$, we have 
\[
    \lim_{N \to \infty} \frac{1}{N} \log \left(\frac{1}{\det(D_N)}\right) = -\int \log(\lambda) \mu_D(\diff \lambda).
\]
Thus it remains only to study the Hessian.

Classical Gaussian computations (e.g., \cite[Section 5.5]{AdlTay2007}) yield
\[
    \nabla^2 \mc{H}(\mathbf{0}) \overset{(d)}{=} W_N + \xi \Id + D_N,
\]
where $W_N$ is distributed according to $\sqrt{B''(0)}$ times the GOE and $\xi \sim \mc{N}(0,B''(0)/N)$ is independent of $W_N$. In fact, since the law of $W_N + \xi \Id$ is invariant under conjugation by orthogonal matrices, we can assume without loss of generality that $D_N$ is diagonal. If we define
\[
    A_N(u) = u \Id + D_N
\]
and $H_N(u) = W_N + A_N(u)$, then we have
\begin{align}
\label{eqn:parabolastandardform}
\begin{split}
    \E\left[\abs{\det(\nabla^2 \mc{H}(\mathbf{0}))}\right] &= \frac{1}{\sqrt{2\pi/N}} \int_\R e^{-N\frac{u^2}{2B''(0)}} \E[\abs{\det(H_N(u))}] \diff u, \\
    \E\left[\abs{\det(\nabla^2 \mc{H}(\mathbf{0}))} \mathbf{1}_{\nabla^2 \mc{H}(\mathbf{0}) \geq 0} \right] &= \frac{1}{\sqrt{2\pi/N}} \int_\R e^{-N\frac{u^2}{2B''(0)}} \E[\abs{\det(H_N(u))} \mathbf{1}_{H_N(u) \geq 0}] \diff u.
\end{split}
\end{align}

Now we study the relevant MDE. Given a linear operator $\mc{S} : \C^{N \times N} \to \C^{N \times N}$ that is self-adjoint with respect to the inner product $\ip{R,T} = \Tr(R^\ast T)$ and that preserves the cone of positive-semi-definite matrices, the problem
\begin{equation}
\label{eqn:parabola_genericmde}
    -M^{-1}(u,z) = z\Id - A_N(u) + \mc{S}[M(u,z)] \quad \text{subject to} \quad \im M(u,z) > 0
\end{equation}
has a unique solution $M(u,z) \in \C^{N \times N}$ for each $z \in \mathbb{H}$ and $u \in \R$. We will consider this problem with two choices of operator $\mc{S}$:
\begin{align*}
    \mc{S}_N[T] = \frac{B''(0)}{N}\Tr(T) + B''(0) \frac{T^{\text{tr}}}{N} \quad \text{induces} \quad M_N(u,z), \\
    \mc{S}'_N[T] = \frac{B''(0)}{N}\Tr(T) \quad \text{induces} \quad M'_N(u,z).
\end{align*}

Let $\mu_N(u)$ and $\mu'_N(u)$ be the probability measures whose Stieltjes transforms are, respectively, $\frac{1}{N}\Tr(M_N(u,z))$ and $\frac{1}{N}\Tr(M'_N(u,z))$. Recall the notation $\rho_{\text{sc},t}$ for the semicircle law of variance $t$.

\begin{lem}
We recognize
\[
    \mu'_N(u) = \rho_{\text{sc},B''(0)} \boxplus \hat{\mu}_{A_N(u)}.
\]
\end{lem}
\begin{proof}
Write $m'_N(u,z)$ for the Stieltjes transform of $\rho_{\text{sc},B''(0)} \boxplus \hat{\mu}_{A_N(u)}$. The Pastur relation \cite{Pas1972}, which characterizes the Stieltjes transform of the free convolution of the semicircle law with another measure,  states that $m'_N(u,z)$ satisfies the self-consistent equation
\[
    m'_N(u,z) = \int \frac{(\hat{\mu}_{D_N + u \Id})(\diff \lambda)}{\lambda-z-B''(0)m'_N(u,z)} = \frac{1}{N} \sum_{i=1}^N \frac{1}{(D_N)_{ii} + u - z - B''(0)m'_N(u,z)}.
\]
(Recall we changed variables so that $D_N$ is diagonal.) If we define
\[
    \mc{M}'_N(u,z) = \diag\left( \frac{1}{(D_N)_{11} + u - z - B''(0)m'_N(u,z)}, \ldots, \frac{1}{(D_N)_{NN} + u - z - B''(0)m'_N(u,z)}\right),
\]
this Pastur relation then gives $m'_N(u,z) = \frac{1}{N}\Tr(\mc{M}'_N(u,z))$, which means that $\mc{M}'_N(u,z)$ exhibits a solution to the MDE \eqref{eqn:parabola_genericmde} with $\mc{S} = \mc{S}'_N$. (Since $\im m'_N(u,z) > 0$ when $z \in \mathbb{H}$, one can check that $\im \mc{M}'_N(u,z) > 0$.) Thus $m'_N(u,z)$, which we defined as the Stieltjes transform of $\rho_{\text{sc},B''(0)} \boxplus \hat{\mu}_{A_N(u)}$, is also the Stieltjes transform of $\mu'_N(u)$. 
\end{proof}

Let $\tau_u$ be the translation $\tau_u(x) = x+u$, and write $(\tau_u)_\ast \mu$ for the pushforward of a probability measure $\mu$ under $\tau_u$ (i.e., the translation of $\mu$ by $u$).
\begin{lem}
\label{lem:parabola_regular}
The measures $\mu'_N(u)$ and
\[
    \mu_\infty(u) = \rho_{\text{sc},B''(0)} \boxplus ((\tau_u)_\ast \mu_D)
\]
admit bounded and compactly supported densities on $\R$, locally uniformly in $u$ (meaning the bound and the compact set can be taken uniform on compact sets of $u$).
\end{lem}
\begin{proof}
These are standard consequences of the regularity of free convolution with the semicircle law, studied in depth by \cite{Bia1997}. For a compactly supported measure $\mu$ and $t > 0$, we have \cite[Corollaries 2, 5]{Bia1997} that $\rho_{\text{sc},t} \boxplus \mu$ admits a density $(\rho_{\text{sc},t} \boxplus \mu)(\cdot)$ with 
\[
    (\rho_{\text{sc},t} \boxplus \mu)(x) \leq \left(\frac{3}{4\pi^3t^2}(4+\abs{\mathtt{r}(\mu) - \mathtt{l}(\mu)})\right)^{1/3} \mathbf{1}_{\mathtt{l}(\mu) - 2\sqrt{t} \leq x \leq \mathtt{r}(\mu)+2\sqrt{t}}.
\]
To study $\mu'_N(u)$, we apply this with $\mu = \hat{\mu}_{A_N(u)}$. Since $\mathtt{r}(\hat{\mu}_{A_N(u)}) \leq u + \sup_N\lambda_{\max{}}(D_N)$ and $\mathtt{l}(\hat{\mu}_{A_N(u)}) \geq u$, both of which are uniformly bounded over $u \in B_R$, we obtain the claim for $\mu'_N(u)$. The proof for $\mu_\infty(u)$ is similar.
\end{proof}

\begin{prop}
\label{prop:parabolalaplace}
We have
\begin{align}
\label{eqn:asy_tot}
&\frac{1}{N} \log \int_\R e^{-\frac{Nu^2}{2B''(0)}} \E[\abs{\det(H_N(u))}] \diff u \underset{N\to\infty}{\longrightarrow} \sup_{u \in \R} \left\{ \int_\R \log\abs{\cdot-u} \rd(\rho_{\text{sc},B''(0)} \boxplus \mu_D)- \frac{u^2}{2B''(0)} \right\},\\
\label{eqn:asy_min}
 &\frac{1}{N} \log \int_\R e^{-\frac{Nu^2}{2B''(0)}} \E[\abs{\det(H_N(u))}\mathbf{1}_{H_N(u) \geq 0}] \diff u \underset{N\to\infty}{\longrightarrow} \sup_{u \leq \mathtt{l}(\rho_{\text{sc},B''(0)} \boxplus \mu_D)} \left\{ \int_\R \log\abs{\cdot-u} \rd(\rho_{\text{sc},B''(0)} \boxplus \mu_D)- \frac{u^2}{2B''(0)} \right\}.
\end{align}
\end{prop}
\begin{proof}
For \eqref{eqn:asy_tot}, we wish to apply \cite[Theorem 4.1]{BenBouMcK2021I} with $\alpha = \frac{1}{2B''(0)}$, $p = 0$, $\mf{D} = \R$, and $\mu_\infty(u)$ as above. To do this, we will consider $H_N(u)$ as a sequence of ``Gaussian matrices with a (co)variance profile,'' in the language of \cite[Corollary 1.8.A]{BenBouMcK2021I}. So we verify the assumptions of that corollary. 

By assumption we have $\sup_N \lambda_{\max{}}(D_N) < \infty$; thus
\[
    \sup_N\|A_N(u)\| \leq \abs{u} + \sup_N \lambda_{\max{}}(D_N) < \infty.
\]
Since $W_N$ is $\sqrt{B''(0)}$ times a GOE matrix, the proof of Lemma \ref{lem:fld_detsublinear} gives us $\E[\abs{\det(H_N(u))}] \leq (C\max(\|u\|,1))^N$ for some $C$. For the same reason, we can compute directly
\[
    \E[W_NTW_N] = \frac{B''(0)}{N}\Tr(T)\Id + \frac{B''(0)}{N}T
\]
which verifies the flatness condition. Since everything is locally uniform in $u$, it remains only to show 
\begin{equation}
\label{eqn:softspins_wasserstein}
    {\rm W}_1(\mu_N(u),\mu_\infty(u)) \leq N^{-\kappa}
\end{equation}
for some $\kappa > 0$. Since all of these measures are compactly supported, locally uniformly in $u$, the Wasserstein-$1$ and bounded-Lipschitz distances are equivalent, so we will work with $d_{\textup{BL}}$. 

First we relate $\mu_N$ to $\mu'_N$, using Lemma \ref{lem:generic_distance_stability} (with $P_N = N$) to estimate the difference between their Stieltjes transforms and then following the proof of \cite[Proposition 3.1]{BenBouMcK2021I}, using the regularity we established in Lemma \ref{lem:parabola_regular}. To relate $\mu'_N$ and $\mu_\infty$, we recall the notation $d_{\textup{L}}$ for the L\'{e}vy distance between probability measures, then combine the translation-invariance of bounded-Lipschitz distance, \cite[Corollary 11.6.5, Theorem 11.3.3]{Dud2002}, and \cite[Proposition 4.13]{BerVoi1993} to obtain
\begin{align*}
    d_{\textup{BL}}(\mu'_N(u), \mu_\infty(u)) &= d_{\textup{BL}}(\rho_{\text{sc},B''(0)} \boxplus \hat{\mu}_{A_N(u)}, \rho_{\text{sc},B''(0)} \boxplus ((\tau_u)_\ast \mu_D)) \\
    &= d_{\textup{BL}}((\tau_u)_\ast (\rho_{\text{sc},B''(0)} \boxplus \hat{\mu}_{D_N}), (\tau_u)_\ast (\rho_{\text{sc},B''(0)} \boxplus \mu_D)) \\
    &= d_{\textup{BL}}(\rho_{\text{sc},B''(0)} \boxplus \hat{\mu}_{D_N}, \rho_{\text{sc},B''(0)} \boxplus \mu_D) \leq 2d_{\textup{L}}(\rho_{\text{sc},B''(0)} \boxplus \hat{\mu}_{D_N}, \rho_{\text{sc},B''(0)} \boxplus \mu_D) \\
    &\leq 2d_{\textup{L}}(\hat{\mu}_{D_N}, \mu_D) \leq 4\sqrt{d_{\textup{BL}}(\hat{\mu}_{D_N}, \mu_D)} \leq N^{-\epsilon},
\end{align*}
uniformly over $u \in \R$, where the last inequality is by assumption \eqref{eqn:speed_of_environment}. This verifies \eqref{eqn:softspins_wasserstein}, and thus \cite[Theorem 4.1]{BenBouMcK2021I} yields
\[
    \lim_{N \to \infty} \frac{1}{N} \log \int_\R e^{-N\frac{u^2}{2B''(0)}} \E[\abs{\det(H_N(u))}] \diff u = \sup_{u \in \R} \left\{ \int_\R \log\abs{\lambda} \mu_\infty(u, \lambda) \diff \lambda - \frac{u^2}{2B''(0)} \right\}.
\]
To obtain \eqref{eqn:asy_tot}, we notice that 
\begin{equation}
\label{eqn:asy_recognize_rhoinfty}
    \mu_\infty(u, \lambda) = (\rho_{\text{sc},B''(0)} \boxplus ((\tau_u)_\ast \mu_D))(\lambda) = ((\tau_u)_\ast(\rho_{\text{sc},B''(0)} \boxplus \mu_D))(\lambda) = (\rho_{\text{sc},B''(0)} \boxplus \mu_D)(\lambda - u)
\end{equation}
and change variables twice (exchanging $u$ and $-u$). This completes the proof of \eqref{eqn:asy_tot}.

For \eqref{eqn:asy_min}, we wish to apply \cite[Theorem 4.5]{BenBouMcK2021I} with $\alpha = \frac{1}{2B''(0)}$, $p = 0$, $\mf{D} = \R$, and $\mu_\infty(u)$ as above. Now we verify its conditions. Arguments as in the elastic-manifold case, specifically Lemma \ref{lem:fld_NlogN}, give \cite[(4.6)]{BenBouMcK2021I}. By \eqref{eqn:asy_recognize_rhoinfty}, $\P(\Spec(H_N(u)) \subset [\mathtt{l}(\mu_\infty(u))-\epsilon, \mathtt{r}(\mu_\infty(u)) + \epsilon])$ is actually independent of $u$, and when $u = 0$ it takes the form
\[
    \P(\Spec(W_N + D_N) \subset [\mathtt{l}(\rho_{\text{sc},B''(0)} \boxplus \mu_D) - \epsilon, \mathtt{r}(\rho_{\text{sc},B''(0)} \boxplus \mu_D) + \epsilon]).
\]
Estimates showing that this tends to one are classical, since $W_N$ is $\sqrt{B''(0)}$ times a GOE matrix and $D_N$ has no outliers by assumption (recall that we made this assumption only for counting local minima, not for counting total critical points). In the generality we need (i.e., with the fewest assumptions on $D_N$), this estimate follows from the large-deviations result \cite[(2.5)]{McK2019} (written for GOE, not $\sqrt{B''(0)}$ times GOE, but clearly goes through in this generality); this verifies \cite[(4.7)]{BenBouMcK2021I}. Finally, the topological requirement \cite[(4.8)]{BenBouMcK2021I} follows immediately after noticing that (in the notation there)
\begin{align*}
    \mc{G} &= \{u : \mu_\infty(u)((-\infty,0)) = 0\} = \{u : u \geq -\mathtt{l}(\rho_{\text{sc},B''(0)} \boxplus \mu_D)\}, \\
    \mc{G}_{+\epsilon} &= \{u : \mathtt{l}(\mu_\infty(u)) \geq 2\epsilon\} = \{u : u \geq 2\epsilon - \mathtt{l}(\rho_{\text{sc},B''(0)} \boxplus \mu_D)\}.
\end{align*}
Having checked all the conditions, we can apply \cite[Theorem 4.5]{BenBouMcK2021I} to complete the proof.
\end{proof}

\begin{proof}[Proof of Theorem \ref{thm:asywell}]
This follows immediately from \eqref{eqn:parabolakacrice}, \eqref{eqn:parabolastandardform}, and Proposition \ref{prop:parabolalaplace}.
\end{proof}

%%%%%%%%%%%%%%%%%%%%%%%%%%%%%%%%%%%%%%%%%%%%%%%%%%%%%%%%%%%%
%%%%%%%%             Subsection: Analyzing the variational formula
%%%%%%%%%%%%%%%%%%%%%%%%%%%%%%%%%%%%%%%%%%%%%%%%%%%%%%%%%%%%

\subsection{Analyzing the variational formula.}
\label{subsec:softspins_variational_analysis}\ 
The key idea presented here is a dynamical analysis of the variational formulas appearing in the previous section,  increasing the noise parameter $B''(0)$.  Important ingredients are the Burgers' equation \eqref{eqn:freeHeatBM} and the square root edge behavior of the relevant free convolutions, as proved in the Appendix.

\begin{proof}[Proof of Theorem \ref{thm:softspins_threshold}]
In this proof, we state several claims as lemmas, postponing their proofs.

We think of the variational problem as dynamic in the parameter $t$, which corresponds to the noise parameter $B''(0)$ in the complexity problem, for fixed $\mu_D$. That is, at ``time $0$'' we have a pure signal with zero complexity, and as ``time'' (meaning noise) increases we find a threshold at which complexity becomes positive. Precisely, write 
\begin{align*}
    \mu_t &= \rho_{\text{sc},t} \boxplus \mu_D, \\
    \ell_t &= \mathtt{l}(\mu_t), \\
    r_t &= \mathtt{r}(\mu_t),
\end{align*}
for the free convolution of $\mu_D$ with the semi-circular distribution of variance $t$ (which has density $\mu_t(\cdot)$) and its left and right edges, respectively. Let 
\[
    F(u,t) = -\int_\R \log(\lambda) \mu_D(\diff \lambda) + \int_\R \log\abs{\lambda - u} \mu_t(\lambda) \diff \lambda - \frac{u^2}{2t}
\]
and recall that we are interested in 
\begin{align*}
    \Sigma^{\text{tot}}(\mu_D,t) &= \sup_{u \in \R} F(u,t), \\
    \Sigma^{\text{min}}(\mu_D,t) &= \sup_{u \leq \ell_t} F(u,t).
\end{align*}
Let
\begin{equation}
\label{eqn:defut}
    u_t = -t\int \frac{\mu_D(\diff \lambda)}{\lambda},
\end{equation}
and consider the thresholds
\begin{align*}
    t_0 &= \inf\{t > 0 : u_t = \ell_t\}, \\
    t_c &= \left(\int \frac{\mu_D(\diff \lambda)}{\lambda^2} \right)^{-1}.
\end{align*}
Later we will show that $t_0 = t_c$, but for now we distinguish between them. In particular we do not yet assume that $t_0$ is finite. Since $\mu_D$ is supported in $(0,\infty)$, we have $u_0 = 0 < \ell_0$, and by continuity we have $u_t < \ell_t$ for all $t < t_0$.

Let $m_t$ be the Stieltjes transform of $\mu_t$, with the sign convention $m_t(z) = \int_\R \frac{\mu_t(\diff \lambda)}{\lambda - z}$. It is known (see for example \cite{Voi1986, Bia1997}, noting their opposite sign convention $m_t(z) = \int_\R \frac{\cdots}{z-\lambda}$) that for any $z$ outside the support of $\mu_t$, we have
\begin{equation}\label{eqn:freeHeatBM}
\partial_t m_t(z)-m_t(z)\partial_z m_t(z)=0.
\end{equation}
For $t < t_0$,  $u_t$ is not in the support of $\mu_t$, so \eqref{eqn:freeHeatBM} gives
$$
\frac{\rd}{\rd t}m_t(u_t)=\partial_u m_t(u)\partial_t u_t+\partial_t m_t(u)=\partial_u m_t(u)(\partial_t u_t+ m_t(u_t))=\partial_u m_t(u)(-m_0(u_0) + m_t(u_t)).
$$
The (unique) solution to this differential equation is clearly
$
m_t(u_t)=m_0(u_0),
$
so that 
\begin{equation}\label{eqn:stieltjesut}
    -m_t(u_t) - \frac{u_t}{t} = 0,
\end{equation}
i.e.
\begin{equation}\label{eqn:diffOptimumnew}
    \left(\frac{\partial}{\partial u} F(u,t) \right)_{u = u_t} = 0
\end{equation}
for $t < t_0$.

\begin{lem}
\label{lem:difflogpotential}
For any $u \in \R$, we have
\[
    \frac{\diff}{\diff t} \int_\R \log\abs{\lambda - u} \mu_t(\diff \lambda) = \frac{\im(m_t(u))^2 - \re(m_t(u))^2}{2}.
\]
\end{lem}

For $t < t_0$ (when $\im(m_t(u_t)) = 0$), we can then use \eqref{eqn:diffOptimumnew}, Lemma \ref{lem:difflogpotential}, and \eqref{eqn:stieltjesut} to obtain 
\[
    \frac{\diff}{\diff t} F(u_t,t) = \left( \frac{\partial}{\partial t} F(u,t) \right)_{u = u_t} + \left(\frac{\partial}{\partial u} F(u,t) \right)_{u = u_t} \partial_t u_t = \left( \frac{\partial}{\partial t} F(u,t) \right)_{u = u_t} = -\frac{m_t(u_t)^2}{2} + \frac{(u_t)^2}{2t^2} = 0.
\]
Together with $F(u_t,t)\to 0$ as $t\to 0$, the above equation gives
\begin{equation}\label{eqn:valueOptimumnew}
F(u_t,t)=0
\end{equation}
for $t < t_0$.

\begin{lem}
\label{lem:softspinsConcavity}
For every measure $\mu_D$ and every $t$, the function $F(u,t)$ is concave in $u$ (possibly not strictly).
\end{lem}

From \eqref{eqn:diffOptimumnew}, \eqref{eqn:valueOptimumnew}, and Lemma \ref{lem:softspinsConcavity} we conclude that
\begin{equation}
\label{eqn:smalltimecomplexity}
    \Sigma^{\textup{tot}}(\mu_D,t) = \Sigma^{\text{min}}(\mu_D,t) = F(u_t,t) = 0 \quad \text{for all} \quad t < t_0.
\end{equation}

Now we study the phase $t > t_0$, showing $t_0 < \infty$ along the way, by considering the evolution of $\ell_t$. 

\begin{lem}
\label{lem:leftedgeevolution}
For all $t > 0$ we have
\[
    \partial_t \ell_t = -\re(m_t(\ell_t)).
\]
\end{lem}

Since the density of $\mu_t$ decays to zero at its edges (in fact at least as quickly as a cube root \cite{Bia1997}), we have $\im(m_t(\ell_t)) = 0$ for all $t$. From Lemmas \ref{lem:leftedgeevolution} and \ref{lem:difflogpotential} we therefore obtain
\begin{align}
\label{eqn:diffFellt}
\begin{split}
    \frac{\diff}{\diff t} F(\ell_t,t) &= \left(\frac{\partial}{\partial u} F(u,t) \right)_{u = \ell_t} \partial_t \ell_t + \left(\frac{\partial}{\partial t} F(u,t) \right)_{u = \ell_t} \\
    &= \left(-\re(m_t(\ell_t)) - \frac{\ell_t}{t}\right)(-\re(m_t(\ell_t))) + \frac{(\im(m_t(\ell_t)))^2}{2} + \frac{1}{2}\left[ \left(\frac{\ell_t}{t}\right)^2 - (\re(m_t(\ell_t)))^2 \right] \\
    &= \frac{1}{2}\left(\frac{\ell_t}{t} + \re(m_t(\ell_t))\right)^2 = \frac{1}{2}\left[\left(\frac{\partial}{\partial u} F(u,t) \right)_{u = \ell_t} \right]^2.
\end{split}
\end{align}

To analyze this, we use the following lemma.

\begin{lem}
\label{lem:GuionnetMaida}
We have
\[
    \left(\frac{\partial}{\partial u} F(u,t)\right)_{u = \ell_t} \begin{cases} < 0 & \text{if} \quad 0 < t < t_c, \\ = 0 & \text{if} \quad t = t_c, \\ > 0 & \text{if} \quad t > t_c. \end{cases}
\]
\end{lem}

Thus \eqref{eqn:diffFellt} is positive for $t \neq t_c$ and vanishes at $t = t_c$. This has two important consequences. First, $F(\ell_t,t)$ is a \emph{strictly} increasing function of $t$. Second, 
\begin{equation}
\label{eqn:t0tc}
    t_0 = t_c.
\end{equation}
Indeed, on the one hand, for $t < t_0$ and small $\epsilon > 0$, we have
\[
    F(\ell_t,t) < F(\ell_{t+\epsilon},t+\epsilon) \leq F(u_{t+\epsilon},t+\epsilon) = 0 = F(u_t,t),
\]
so that $\left(\frac{\partial}{\partial u} F(u,t)\right)_{u = \ell_t} \neq 0$ by concavity in $u$ (Lemma \ref{lem:softspinsConcavity}) and \eqref{eqn:diffOptimumnew}, and thus $t \neq t_c$. Hence $t_0 \leq t_c < \infty$.

On the other hand, if $t$ has the property that $\sup_{u \in \R} F(u,t) = F(\ell_t,t)$, then we have $\left(\frac{\partial}{\partial u} F(u,t)\right)_{u = \ell_t} = 0$, thus $t = t_c$. But $t_0$ has this property, now that we know it is finite, since by continuity we have
\[
    0 = \sup_{u \in \R} F(u,t_0) = F(u_{t_0},t_0) = F(\ell_{t_0},t_0).
\]
We have shown that $F(\ell_t,t)$ is a strictly increasing function which vanishes at $t_c$; thus
\begin{equation}
\label{eqn:largetimecomplexity}
    \Sigma^{\textup{tot}}(\mu_D,t) \geq \Sigma^{\textup{min}}(\mu_D,t) \geq F(\ell_t,t) > F(\ell_{t_c},t_c) = 0 \quad \text{for all} \quad t > t_c.
\end{equation}
The fact that both complexities vanish if and only if $t \leq t_c$ follows immediately from \eqref{eqn:smalltimecomplexity}, \eqref{eqn:t0tc}, and \eqref{eqn:largetimecomplexity} (the case $t = t_0$ follows from \eqref{eqn:smalltimecomplexity} by continuity).

Lemmas \ref{lem:softspinsConcavity} and \ref{lem:GuionnetMaida} combine to give \eqref{eqn:softspins_minima_formula}, as well as strict inequality in $\Sigma^{\textup{tot}}(\mu_D,t) > \Sigma^{\textup{min}}(\mu_D,t)$ for $t > t_c$. Now we prove \eqref{eqn:softspins_total_formula}. To do this, we will rely on 
Pastur's relation \cite{Pas1972}
\begin{equation}\label{eqn:PasturBM}
    m_t(z) = \int \frac{\mu_D(\diff \lambda)}{\lambda-z-tm_t(z)}.
\end{equation} 
By taking real and imaginary parts of \eqref{eqn:PasturBM}, we get for any $u \in \R$ the coupled system
\begin{align}
    \re(m_t(u)) &= \int \frac{\lambda - u - t\re(m_t(u))}{(\lambda - u - t\re(m_t(u)))^2 + t^2\im(m_t(u))^2} \, \mu_D(\diff \lambda), \label{eqn:PasturReal} \\
    \im(m_t(u)) &= t \int \frac{\im(m_t(u))}{(\lambda - u - t\re(m_t(u)))^2 + t^2\im(m_t(u))^2} \, \mu_D(\diff \lambda).\label{eqn:PasturImaginary}
\end{align}
If $v_t$ satisfies $F(v_t,t) = \sup_{u \in \R} F(u,t)$, then 
\[
    0 = \left(\frac{\partial}{\partial u} F(u,t)\right)_{u = v_t} = -\frac{v_t}{t} - \re(m_t(v_t)).
\]
We plug this into \eqref{eqn:PasturReal} and \eqref{eqn:PasturImaginary} to obtain, writing $y_t = \im(m_t(v_t))$,
\begin{align}
    -\frac{v_t}{t} &= \int \frac{\lambda}{\lambda^2 + t^2y_t^2} \, \mu_D(\diff \lambda), \label{eqn:findvt}\\
    y_t &= t \int \frac{y_t}{\lambda^2 + t^2y_t^2} \, \mu_D(\diff \lambda). \label{eqn:findyt}
\end{align}
From its definition, $y_t \geq 0$. For every $t > 0$, notice that $(u_t,0)$ is a solution to the coupled system \{\eqref{eqn:findvt}, \eqref{eqn:findyt}\}, where $u_t$ was defined in \eqref{eqn:defut}. We claim that this is the unique solution when $t \leq t_c$, but that for $t > t_c$ there is exactly one more solution, with $y_t > 0$, and that for such times this latter solution is the one corresponding to the optimizer (i.e., for $t > t_c$ the point $u_t$ is not an optimizer anymore). 

For existence and uniqueness of this second solution exactly when $t > t_c$, we note that the positive solutions $y_t$ to \eqref{eqn:findyt} are exactly the positive solutions to 
\begin{equation}
\label{eqn:softspins_yt}
    \frac{1}{t} = \int_\R \frac{\mu_D(\diff \lambda)}{\lambda^2 + t^2y_t^2},
\end{equation}
but the right-hand side of this equation is a strictly decreasing function of $y_t$, tending to zero as $y_t \to +\infty$ and tending to $\frac{1}{t_c}$ as $y_t \downarrow 0$ (which is bigger than $\frac{1}{t}$ precisely when $t > t_c$).

To verify the claim that $u_t$ is not an optimizer when $t > t_c$, it suffices to show
\begin{equation}
\label{eqn:utlesslt}
	u_t \leq \ell_t \quad \text{for all} \quad t.
\end{equation}
Indeed, since $F(u,t)$ is concave and $\left(\frac{\partial}{\partial u} F(u,t) \right))_{u = \ell_t} > 0$ in the regime $t > t_c$, \eqref{eqn:utlesslt} would imply that $u_t$ is not the optimizer of $F(\cdot,t)$ when $t > t_c$. 

To show \eqref{eqn:utlesslt}, we will show that $t \mapsto \ell_t - u_t$ is convex with a vanishing derivative at $t = t_c$, where it takes the value zero. First we claim
\begin{equation}
\label{eqn:elltconvex}
    \frac{\diff^2}{\diff t^2}(\ell_t - u_t) = \frac{\diff^2}{\diff t^2} \ell_t = \frac{\diff}{\diff t}(-m_t(\ell_t)) \geq 0
\end{equation}
for all $t$. Indeed, a simple calculation similar to the previous ones gives, for any $\epsilon > 0$,
\begin{equation}
\label{eqn:almostleftedge}
    \frac{\diff}{\diff t} m_t(\ell_t - \epsilon) = (m_t(\ell_t - \epsilon) - m_t(\ell_t))\partial_u m_t(\ell_t-\epsilon) = \int \frac{-\sqrt{\epsilon}\mu_t(\diff \lambda)}{(\lambda - (\ell_t-\epsilon))(\lambda - \ell_t)} \int \frac{\sqrt{\epsilon}\mu_t(\diff \lambda)}{(\lambda-(\ell_t-\epsilon))^2} < 0.
\end{equation}
As $\epsilon \downarrow 0$, each of the integrals on the right-hand side converges, since $\mu_t$ decays at its left edge at least as quickly as a square root by Proposition \ref{prop:freeconv_edge}, and since, for example,
\[
    \lim_{\epsilon \downarrow 0} \int_0^1 \frac{\sqrt{\epsilon}x^p}{(x+\epsilon)x} \diff x = \begin{cases} \pi & \text{if } p = 1/2, \\ 0 & \text{if } p > 1/2. \end{cases}
\]
(When $p = 1/2$, this can be integrated directly at each $\epsilon$; when $p > 1/2$, we use dominated convergence with dominating function $\frac{1}{2}x^{p-\frac{3}{2}}$.) Thus in the limit $\epsilon \downarrow 0$ we prove the existence of $\frac{\diff}{\diff t} m_t(\ell_t) \leq 0$, concluding the proof of \eqref{eqn:elltconvex}.
Since
\[
	\left[ \frac{\diff}{\diff t} (\ell_t - u_t) \right]_{t = t_c} = -\re(m_{t_c}(\ell_{t_c})) - \frac{u_{t_c}}{t_c} = -\re(m_{t_c}(\ell_{t_c})) - \frac{\ell_{t_c}}{t_c} = \left(\frac{\partial}{\partial u} F(u,t_c)\right)_{u = \ell_{t_c}} = 0
\]
with $\ell_{t_c} = u_{t_c}$, we conclude \eqref{eqn:utlesslt} and thus \eqref{eqn:softspins_total_formula}.

Next we study the degree of vanishing of $\Sigma^{\text{tot}}(\mu_D,t) = F(v_t,t)$ as $t \downarrow t_c$. First, note that $v_t$ and $y_t$ are $C^1$ functions of $t > t_c$ with the appropriate right-hand limits at criticality (namely $\lim_{t \downarrow t_c} y_t = 0$ and $\lim_{t \downarrow t_c} v_t = \ell_{t_c}$): this is proved, first by studying $y_t$ via \eqref{eqn:softspins_yt} and the implicit function theorem, then  studying $v_t$ via \eqref{eqn:findvt} using the knowledge of $y_t$. For $t > t_c$, Lemma \ref{lem:difflogpotential} gives
\[
    \frac{\diff}{\diff t} F(v_t,t) = \underbrace{\left(\frac{\partial}{\partial u} F(u,t)\right)_{u = v_t}}_{=0} \partial_t v_t + \left(\frac{\partial}{\partial t} F(u,t)\right)_{u = v_t} = \frac{\im(m_t(v_t))^2 - \re(m_t(v_t))^2}{2} + \frac{v_t^2}{2t^2} = \frac{\im(m_t(v_t))^2}{2} = \frac{y_t^2}{2}.
\]
As $t \downarrow t_c$, this tends to zero. Differentiating \eqref{eqn:softspins_yt} in $t$ to find an expression for $y_ty'_t$ and inserting it, we find 
\[
    \frac{\diff^2}{\diff t^2} F(v_t,t) = y_ty'_t = \frac{1}{2t^4\int_\R \frac{\mu_D(\diff \lambda)}{(\lambda^2+t^2y_t^2)^2}} - \frac{y_t^2}{t}.
\]
As $t \downarrow t_c$, this tends to $\frac{1}{2}t_c^{-4} \left(\int_\R \frac{\mu_D(\diff \lambda)}{\lambda^4}\right)^{-1} = \frac{1}{2} \left(\int_\R \frac{\mu_D(\diff \lambda)}{\lambda^2}\right)^{4}\left(\int_\R \frac{\mu_D(\diff \lambda)}{\lambda^4}\right)^{-1}$, which is positive. This gives us the quadratic decay and the prefactor.

Finally we study the degree of vanishing of $\Sigma^{\text{min}}(\mu_D,t) = F(\ell_t,t)$ as $t \downarrow t_c$. To do this, we first study regularity of $m_t(\ell_t)$ (we studied regularity of $\ell_t$ above, around \eqref{eqn:elltconvex}). Notice that $\im(m_t(\ell_t)) = 0$ but $\im(m_t(\ell_t+\epsilon)) > 0$ for all sufficiently small $\epsilon > 0$, since $\mu_t$ admits a density that vanishes at the endpoints and is analytic where positive \cite{Bia1997}; using this in the real and imaginary parts \eqref{eqn:PasturReal} and \eqref{eqn:PasturImaginary} of the Pastur relation, we obtain
\begin{align}
    m_t(\ell_t) &= \int \frac{1}{\lambda - \ell_t - tm_t(\ell_t)} \, \mu_D(\diff \lambda), \label{eqn:PasturEdgeReal} \\
    1 &= t \int \frac{1}{(\lambda - \ell_t - tm_t(\ell_t))^2} \, \mu_D(\diff \lambda). \label{eqn:PasturEdgeImaginary}
\end{align}
For $t > t_c$, we will show in the proof of Lemma \ref{lem:GuionnetMaida} that $\ell_t+tm_t(\ell_t) < 0$; thus $\int \frac{\mu_D(\diff \lambda)}{(\lambda-\ell_t-tm_t(\ell_t))^p} < \infty$ for all $p > 0$. This also implies, using the implicit function theorem, that $\ell_t+tm_t(\ell_t)$ is a $C^2$ function of $t > t_c$, hence so is $m_t(\ell_t)$. Differentiating \eqref{eqn:PasturEdgeImaginary} in $t$ and solving for $\frac{\diff}{\diff t}m_t(\ell_t)$, we find
\[
    \frac{\diff}{\diff t} m_t(\ell_t) = -\frac{1}{2t^3 \left(\int_\R \frac{\mu_D(\diff \lambda)}{(\lambda-\ell_t-tm_t(\ell_t))^3} \right)}.
\]
As $t \downarrow t_c$, this tends to $-\frac{1}{2}\left(\int_\R \frac{\mu_D(\diff \lambda)}{\lambda^2}\right)^3 \left(\int_\R \frac{\mu_D(\diff \lambda)}{\lambda^3}\right)^{-1}$. Now we compute derivatives: We have $F(\ell_{t_c},t_c) = 0$, by combining \eqref{eqn:diffFellt} and Lemma \ref{lem:GuionnetMaida} we find that the first derivative also vanishes at criticality. Next, from \eqref{eqn:diffFellt} and Lemma \ref{lem:leftedgeevolution} we have
\[
    \frac{\diff^2}{\diff t^2} F(\ell_t,t) = \left(\frac{\ell_t}{t} + m_t(\ell_t)\right) \left(-\frac{m_t(\ell_t)}{t} - \frac{\ell_t}{t^2} + \frac{\diff}{\diff t} m_t(\ell_t)\right).
\]
At $t = t_c$, this vanishes by Lemma \ref{lem:GuionnetMaida}. The third derivative is
\[
    \frac{\diff^3}{\diff t^3} F(\ell_t,t) = \left(\frac{\ell_t}{t} + m_t(\ell_t)\right)\left(-\frac{1}{t} \cdot \frac{\diff}{\diff t} m_t(\ell_t) + 2\frac{m_t(\ell_t)}{t^2} + 2\frac{\ell_t}{t^3} + \frac{\diff^2}{\diff t^2} m_t(\ell_t)\right) + \left(-\frac{m_t(\ell_t)}{t} - \frac{\ell_t}{t^2} + \frac{\diff}{\diff t} m_t(\ell_t)\right)^2.
\]
Since $\frac{\ell_{t_c}}{t_c} + m_{t_c}(\ell_{t_c}) = 0$, at $t = t_c$ this reduces to $\left[\left(\frac{\diff}{\diff t} m_t(\ell_t)\right)_{t = t_c}\right]^2$, which we computed above (and which is clearly nonzero). This gives the cubic decay and the prefactor, and completes the proof.
\end{proof}

\bigskip

\begin{proof}[Proof of Lemma \ref{lem:difflogpotential}]
For large (in absolute value) negative $A$ and small $\eta > 0$, by \eqref{eqn:freeHeatBM} we have
\begin{align}
\label{eqn:difflogpotential_etapositive}
\begin{split}
    &\frac{\diff}{\diff t} \int_\R (\log\abs{\lambda - (u+\ii\eta)} - \log\abs{\lambda - (A+\ii\eta)}) \mu_t(\lambda) \diff \lambda \\
    &= -\frac{\diff}{\diff t} \int_\R \int_A^u \frac{\lambda-x}{(\lambda-x)^2+\eta^2} \diff x \, \mu_t(\diff \lambda) = -\frac{\diff}{\diff t} \int_A^u \left[ \int_\R \re \frac{\mu_t(\diff \lambda)}{\lambda - (x+\ii\eta)} \right] \diff x = -\re \left[ \int_A^u \frac{\diff}{\diff t} m_t(x+\ii\eta) \diff x \right] \\
    &= -\frac{1}{2}\re\left[\int_A^u \partial_z(m_t(x+\ii\eta)^2) \diff x \right] = -\frac{\re(m_t(u+\ii\eta)^2) - \re(m_t(A+\ii\eta)^2)}{2}.
\end{split}
\end{align}
We will take $A \to -\infty$ and $\eta \downarrow 0$ in that order. After these two limits, the right-hand side of \eqref{eqn:difflogpotential_etapositive} reads
\[
    \frac{\im(m_t(u))^2 - \re(m_t(u))^2}{2}.
\]
Now we claim that
\begin{equation}
\label{eqn:softspins_dominatedconvergence}
    \lim_{A \to -\infty} \frac{\diff}{\diff t} \int_\R \log\abs{\lambda - (A+\ii\eta)} \mu_t(\lambda) \diff \lambda = 0
\end{equation}
for every $\eta > 0$. Indeed, since $\mu_t$ has mass one for all $t$ and $\mu_t(r_t) = \mu_t(\ell_t) = 0$, we have
\begin{align*}
     &\frac{\diff}{\diff t} \int_\R \log\abs{\lambda - (A+\ii\eta)} \mu_t(\lambda) \diff \lambda = \frac{\diff}{\diff t} \int_{\ell_t}^{r_t} \log\abs{\frac{-\lambda+\ii\eta}{A} + 1} \mu_t(\lambda) \diff \lambda = \int_{\ell_t}^{r_t} \log\abs{\frac{-\lambda+\ii\eta}{A} + 1} \partial_t \mu_t(\lambda)\diff \lambda.
\end{align*}
As $A \to -\infty$, the integrand on the right-hand side tends pointwise to zero, and it is bounded in absolute value by 
\[
    \abs{\partial_t\mu_t(\lambda)} \max\left\{\abs{\log\abs{\frac{-\ell_t+\ii\eta}{A_0}+1}},\abs{\log\abs{\frac{-r_t+\ii\eta}{A_0}+1}}\right\}
\]
for all $A \geq A_0 = A_0(t)$. This is integrable by Lemma \ref{lem:diffdensity} below and H\"{o}lder's inequality, so we can conclude the proof of \eqref{eqn:softspins_dominatedconvergence} by dominated convergence.

Thus as $A \to -\infty$ the left-hand side of \eqref{eqn:difflogpotential_etapositive} tends to 
\begin{align*}
    \frac{\diff}{\diff t}\int_\R \log\abs{\lambda - (u+\ii\eta)} \mu_t(\lambda) \diff \lambda = \int_{\ell_t}^{r_t} \log\abs{\lambda - (u+\ii\eta)} \partial_t \mu_t(\lambda) \diff \lambda.
\end{align*}
As $\eta \downarrow 0$, this tends to $\frac{\diff}{\diff t} \int \log\abs{\lambda-u} \mu_t(\lambda) \diff \lambda$ by dominated convergence, using for example the dominating function
\[
    \max\{-\log{\abs{\lambda-u}}, -\log(1/2), \log\sqrt{(\lambda-u)^2+1/2}\} \abs{\partial_t \mu_t(\lambda)} \mathbf{1}_{\lambda \in [\ell_t,r_t]}
\]
for $\eta^2 < 1/2$, which is integrable by Lemma \ref{lem:diffdensity} and H\"{o}lder's inequality.
\end{proof}

\begin{lem}
\label{lem:diffdensity}
The derivative $\partial_t \mu_t(\lambda)$ is in $L^p(\R)$, as a function of $\lambda$, for $1 < p < 3/2$.
\end{lem}
\begin{proof}
For $\eta > 0$, the Burgers equation \eqref{eqn:freeHeatBM} gives 
\begin{align}
\label{eqn:diffdensityetapos}
\begin{split}
    \partial_t \im(m_t(\lambda+\ii \eta)) &= \im\left(\partial_z\left(\frac{m_t(\lambda+\ii \eta)^2}{2}\right)\right) = \partial_z \big[ \re(m_t(\lambda+\ii \eta))\im(m_t(\lambda+\ii \eta)) \big] \\
    &= \big[ \partial_\lambda \re(m_t(\lambda+\ii \eta)) \big] \im(m_t(\lambda+\ii\eta)) + \big[ \partial_\lambda \im(m_t(\lambda+\ii\eta)) \big] \re(m_t(\lambda+\ii\eta)).
\end{split}
\end{align}
We now consider $\eta \downarrow 0$.  As $\mu_t$ is analytic on $\{\lambda : \mu_t(\lambda) > 0\}$ \cite[Corollary 4]{Bia1997}, if $\lambda$ is not an edge or cusp of $\mu_t$,
\[
    \lim_{\eta \downarrow 0} \partial_\lambda \im(m_t(\lambda+\ii\eta)) = \lim_{\eta \downarrow 0} \left( \pi \int_\R \frac{\eta}{(s-\lambda)^2+\eta^2} \mu'_t(s) \diff s \right) = \pi\mu'_t(\lambda).
\]
As $\mu'_t$ is compactly supported and analytic on the set where it does not vanish, this limit is locally uniform in $\lambda$. By the same argument, this local uniformity also holds for $\lim_{\eta \downarrow 0} \im(m_t(\lambda+\ii\eta)) = \pi\mu_t(\lambda)$. We argue similarly for the real part (noting that the interchange $\lim_{\eta \downarrow 0} \partial_\lambda \re(m_t(\lambda+\ii\eta)) = \partial_\lambda \re(m_t(\lambda+\ii 0^+))$ is simply a rephrasing of the fact that the Hilbert transform commutes with differentiation). Furthermore, \cite[Proposition 2, Lemma 3]{Bia1997} gives
\begin{equation}
\label{eqn:Bianeuniformbound}
    \sup_{z \in \mathbb{H}}\abs{m_t(z)} \leq \frac{1}{\sqrt{t}}.
\end{equation}
Thus the right-hand side of \eqref{eqn:diffdensityetapos} tends to $\pi\partial_\lambda[\re(m_t(\lambda+\ii 0^+))\mu_t(\lambda)]$ as $\eta \downarrow 0$, and this limit is locally uniform in $\lambda$. This justifies swapping the limit and derivative on the left-hand side of \eqref{eqn:diffdensityetapos}, and dividing through by $\pi$ we obtain
\begin{equation}
\label{eqn:diffdensity}
    \partial_t \mu_t(\lambda) =  \partial_\lambda\bigg[ \re(m_t(\lambda+\ii 0^+))\mu_t(\lambda)\bigg]
\end{equation}
for $\lambda$ not an edge or cusp of $\mu_t$. 

Now we prove the regularity claim. Since $\mu_t$ decays at most like a cube root near its edges and possible cusps \cite[Corollary 5]{Bia1997}, we have $\partial_\lambda \mu_t \in L^p(\rd\lambda)$, for any $1 \leq p < 3/2$. Since the Hilbert transform commutes with differentiation and is bounded on $L^p$ for $1 < p < \infty$, we also have $\partial_\lambda(\re(m_t(\lambda+\ii 0^+))) \in L^p(\rd\lambda)$,  for the same range of $p$ values. Expanding the derivative in \eqref{eqn:diffdensity} and using \eqref{eqn:Bianeuniformbound}, we conclude that $\partial_t \mu_t$ is in $L^p$ for $1 < p < 3/2$. 
\end{proof}

\begin{proof}[Proof of Lemma \ref{lem:softspinsConcavity}]
Assume first that $\supp(\mu_D)$ is connected.  By \cite[Proposition 3]{Bia1997}  $\supp(\mu_t)$ is connected for any $t > 0$.  By  \cite[Corollary 4]{Bia1997}  $\mu_t$ has a density that is analytic on $\{x : \mu_t(x) > 0\}$ (although it can have cusps).

Outside of $\supp(\mu_t)$, the function $F(\cdot,t)$ is concave as the sum of concave functions.  For $\mu_t(u) > 0$, we compute $\partial_{uu} F(u,t)$ below. For any $\eta=\im z>0$ by taking the imaginary part of \eqref{eqn:PasturBM} we have on the one hand
$$
\im m_t(z)= \int \frac{\mu_D(\diff \lambda)(\eta+t\im m_t(z))}{|\lambda-z-tm_t(z)|^2},
$$
i.e.
\begin{equation}\label{eqn:constantBM}
1= t\int \frac{\mu_D(\diff \lambda)}{|\lambda-z-tm_t(z)|^2},
\end{equation}
for $z=u+\ii \eta$ and $\eta=0^+$.
On the other hand, differentiation of \eqref{eqn:PasturBM} gives
$$
\re \partial_z m_t(z)=\re \frac{X}{1-tX}=\frac{1}{t}\re \frac{1}{1-tX}-\frac{1}{t}\ \ \ {\rm with}\ 
X:=\int \frac{\mu_D(\diff \lambda)}{(\lambda-z-tm_t(z))^2}.
$$
From \eqref{eqn:constantBM}, for $z=u+\ii 0^+$ we have $|tX|\leq1$ so that $\re \frac{1}{1-tX}\geq 0$.
Note that by analyticity, $\partial_z m=\partial_u \re m+\ii\partial_u\im m$, so we have proved
$$
\partial_u m_t(u)\geq -\frac{1}{t},
$$
so that
$$\frac{\partial^2}{\partial u^2} F(u,t) \leq \frac{1}{t}-\frac{1}{t}\leq 0.$$

Since $F(\cdot,t)$ is differentiable at $\ell_t$ (with derivative $-\re(m_t(\ell_t)) - \ell_t/t$) and similarly for $r_t$, this completes the proof if $\supp(\mu_D)$ is connected. In the general case, write $I$ for the convex hull of $\supp(\mu_D)$, which is necessarily an interval gapped away from zero, write $\nu_I$ for uniform measure on $I$, and consider the probability measures
$
    \mu_D^{(\epsilon)} = (1-\epsilon) \mu_D + \epsilon\nu_I.
$
We temporarily add the measure to the notation for $F(u,t)$, writing $F(u,t,\mu_D)$. We have $\mu_D^{(\epsilon)} \to \mu_D$ weakly as $\epsilon \to 0$; in particular, since $\lambda \mapsto \log(\lambda)$ is bounded and continuous on $I$, we have
\[
    \lim_{\epsilon \to 0} \int_\R \log(\lambda) \mu_D^{(\epsilon)}(\diff \lambda) = \int_\R \log(\lambda) \mu_D(\diff \lambda).
\]
Combined with Lemma \ref{lem:logpotentiallimit} below, this lets us conclude that $F(\cdot,t,\mu_D^{(\epsilon)}) \to F(\cdot,t,\mu_D)$ pointwise as $\epsilon \to 0$. Since each $\supp(\mu_D^{(\epsilon)}) = I$ is connected, $F(\cdot,t,\mu_D)$ is thus concave as the pointwise limit of concave functions.
\end{proof}

\begin{proof}[Proof of Lemma \ref{lem:leftedgeevolution}] 
Differentiating both sides of \eqref{eqn:PasturEdgeReal} in $t$ and using \eqref{eqn:PasturEdgeImaginary}, we obtain
\begin{align*}
    \frac{\diff}{\diff t} \re(m_t(\ell_t)) = \int \frac{\partial_t\ell_t + t\frac{\diff}{\diff t} \re(m_t(\ell_t)) + \re(m_t(\ell_t))}{(\lambda - \ell_t - t\re(m_t(\ell_t)))^2} \mu_D(\diff \lambda) = \frac{\partial_t\ell_t + \re(m_t(\ell_t))}{t} + \frac{\diff}{\diff t} \re(m_t(\ell_t)).
\end{align*}
(Differentiability of $m_t(\ell_t)$ was established using \eqref{eqn:almostleftedge}.) We note that the idea to study the evolution of the edge by differentiating a self-consistent equation that it satisfies also appears in the proof of \cite[Proposition 3.4]{AdhHua2020}.
\end{proof}

\begin{proof}[Proof of Lemma \ref{lem:GuionnetMaida}]
Notice that 
\[
    \left(\frac{\partial}{\partial u} F(u,t)\right)_{u = \ell_t} = -\frac{\ell_t}{t} - \re(m_t(\ell_t)).
\]
We work with the right-hand side. We claim that
\begin{equation}
\label{eqn:GuionnetMaidaMagic}
    \ell_t + t \re(m_t(\ell_t)) \leq \mathtt{l}(\mu_D).
\end{equation}
This is in fact a special case of an inequality established by Guionnet-Ma\"{i}da in the proof of \cite[Lemma 6.1]{GuiMai2020}, which says that if $\mu$ and $\nu$ are compactly supported probability measures and $\omega$ is the so-called \emph{subordination function} defined implicitly by 
\[
    \int \frac{(\mu \boxplus \nu)(\diff \lambda)}{\lambda - z} = \int \frac{\mu(\diff \lambda)}{\lambda - \omega(z)},
\]
then 
\[
    \omega(\mathtt{r}(\mu \boxplus \nu)) \geq \mathtt{r}(\mu).
\]
In our case, $\nu = \rho_{\text{sc},t}$ and $\mu = \mu_D$, so that $\mu \boxplus \nu = \mu_t$, and the Pastur relation \eqref{eqn:PasturBM} shows that the subordination function is $\omega(z) = z + tm_t(z)$. (In fact, these choices give us results about the right edge; to get \eqref{eqn:GuionnetMaidaMagic}, one should choose $\mu = -\mu_D$, the measure defined by $-\mu_D(A) = \mu_D(-A)$ for Borel $A$, then track the negative signs.) 

Combined with \eqref{eqn:PasturEdgeImaginary}, the result \eqref{eqn:GuionnetMaidaMagic} shows that $w_t = \ell_t + t\re(m_t(\ell_t))$ is a solution to the following constrained problem:
\begin{equation}
\label{eqn:wproblem}
    \begin{cases}
    \dfrac{1}{t} = {\displaystyle \int} \dfrac{\mu_D(\diff \lambda)}{(\lambda - w_t)^2}, & \\
    w_t \leq \mathtt{l}(\mu_D). &
    \end{cases}
\end{equation}
A short differential calculation shows that $f(w) = \int \frac{\mu_D(\diff \lambda)}{(\lambda - w)^2}$ is \emph{strictly} increasing for $w \leq \ell(\mu_D)$, so \eqref{eqn:wproblem} has at most one solution. Furthermore, $f(0) = \frac{1}{t_c}$; this means that the unique solution (which we showed is $\ell_t + t\re(m_t(\ell_t))$) must be positive if $0 < t < t_c$, must be zero if $t = t_c$, and must be negative if $t > t_c$.
\end{proof}

\begin{lem}
\label{lem:logpotentiallimit}
Suppose that $\mu_N$ is a sequence of probability measures, all supported on some $[a,b]$, tending weakly to some $\mu_\infty$ which is also supported on $[a,b]$. Then for every $t > 0$ and every $u \in \R$ we have 
\[
    \lim_{N \to \infty} \int_\R \log\abs{\lambda - u} (\rho_{\text{sc},t} \boxplus \mu_N)(\lambda) \diff \lambda = \int_\R \log\abs{\lambda - u} (\rho_{\text{sc},t} \boxplus \mu_\infty)(\lambda) \diff \lambda.
\]
\end{lem}
\begin{proof}
For small positive $\eta = \eta_N$ to be chosen, define $\log_\eta : \R \to \R$ by
$
    \log_\eta(x) = \log\abs{x + \ii\eta}.
$
For any probability measure $\mu$ supported on $[a,b]$, \cite[Corollaries 2, 5]{Bia1997} yields
\[
    (\rho_{\text{sc},t} \boxplus \mu)(\lambda) \leq \left(\frac{3}{4\pi^3t^2}(4+b-a)\right)^{1/3} \mathbf{1}_{a-2\sqrt{t} \leq \lambda \leq b+2\sqrt{t}}.
\]
Since$
    \int_\R \frac{\log_\eta(\lambda) - \log\abs{\lambda}}{\eta} \diff \lambda = \pi,$
we have
\[
    \abs{\int_\R \log\abs{\lambda - u} (\rho_{\text{sc},t} \boxplus \mu)(\lambda) \diff \lambda - \int_\R \log_\eta(\lambda - u) (\rho_{\text{sc},t} \boxplus \mu)(\lambda) \diff \lambda}  \leq \left(\frac{3}{4\pi^3t^2}(4+b-a)\right)^{1/3} \pi \eta
\]
which depends on $\mu$ only through $[a,b]$. On the other hand, the function $f_{u,\eta}(\lambda) = \log_\eta(\lambda - u)$ is $\frac{1}{2\eta}$-Lipschitz and bounded on $[a-2\sqrt{t},b+2\sqrt{t}]$ by 
\[
    \max\{\abs{\log(\eta)},\abs{\log_\eta(b-u+2\sqrt{t})},\abs{\log_\eta(a-u-2\sqrt{t})}\} = \abs{\log(\eta)}
\]
where the equality holds for $\eta$ sufficiently small depending on $a$, $b$, and $u$. Since combining \cite[Corollary 11.65, Theorem 11.3.3]{Dud2002} and \cite[Proposition 4.13]{BerVoi1993} gives
\[
    d_{\textup{BL}}(\rho_{\text{sc},t} \boxplus \mu_N, \rho_{\text{sc},t} \boxplus \mu_\infty) \leq 4\sqrt{d_{\textup{BL}}(\mu_N,\mu_\infty)},
\]
we bound $\abs{\int_\R \log\abs{\cdot - u} \rd(\rho_{\text{sc},t} \boxplus \mu_N)- \int_\R \log\abs{\cdot - u} \rd(\rho_{\text{sc},t} \boxplus \mu_\infty)}$ with
\begin{align*}
    & \sum_{\nu=\mu_N,\mu_\infty}\abs{\int_\R (\log\abs{\cdot - u}-\log_\eta(\cdot - u) )\rd(\rho_{\text{sc},t} \boxplus \nu)} + \abs{\int_\R \log_\eta(\cdot - u) \rd(\rho_{\text{sc},t} \boxplus \mu_N-\rho_{\text{sc},t} \boxplus \mu_\infty)}\\
    &\leq \left(\frac{3}{4\pi^3t^2}(4+b-a)\right)^{1/3} \pi \eta + \left(\frac{1}{2\eta} + \abs{\log(\eta)}\right)d_{\textup{BL}}(\rho_{\text{sc},t} \boxplus \mu_N, \rho_{\text{sc},t} \boxplus \mu_\infty) \\
    &\leq \left(\frac{3}{4\pi^3t^2}(4+b-a)\right)^{1/3} \pi \eta + \left(\frac{1}{2\eta} + \abs{\log(\eta)}\right)4\sqrt{d_{\textup{BL}}(\mu_N,\mu_\infty)},
\end{align*}
for $\eta$ sufficiently small depending on $u$. If we choose
$
    \eta = \eta_N = \big(d_{\textup{BL}}(\mu_N,\mu_\infty)\big)^{1/4},
$
which tends to zero as $N \to \infty$, this upper bound also tends to zero as $N \to \infty$.
\end{proof}

%%%%%%%%%%%%%%%%%%%%%%%%%%%%%%%%%%%%%%%%%%%%%%%%%%%%%%%%%%%%
%%%%%%%%%%%%%%%%%%%%%%%%%%%%%%%%%%%%%%%%%%%%%%%%%%%%%%%%%%%%
%%%%%%%%
%%%%%%%%             Section: Extensions to products of determinants
%%%%%%%%
%%%%%%%%%%%%%%%%%%%%%%%%%%%%%%%%%%%%%%%%%%%%%%%%%%%%%%%%%%%%
%%%%%%%%%%%%%%%%%%%%%%%%%%%%%%%%%%%%%%%%%%%%%%%%%%%%%%%%%%%%

\setcounter{equation}{0}
\setcounter{thm}{0}
\renewcommand{\theequation}{A.\arabic{equation}}
\renewcommand{\thethm}{A.\arabic{thm}}
\appendix
\setcounter{secnumdepth}{0}
\section[Appendix\ \ \ Edge behavior of general free convolutions with semicircle]
{Appendix\ \ \ Edge behavior of general free convolutions with semicircle}
\label{sec:freeconvolution}

Recall the notation of Section \ref{subsec:softspins_variational_analysis} for the free convolution of a measure $\mu_D$ with the semi-circular distribution of variance $t$, and for its left edge:
$$
    \mu_t = \rho_{\text{sc},t} \boxplus \mu_D, \ \ 
    \ell_t = \mathtt{l}(\mu_t), \ \ 
    m_t(z) = \int_\R \frac{\mu_t(\diff \lambda)}{\lambda - z}.
$$
Recall also the notation $\mu_t(\cdot)$ for the density of $\mu_t$. 
The following result might be of independent interest.

\begin{prop}
\label{prop:freeconv_edge}
Any free convolution with semicircle decays at least as fast as a square root at the extremal edges, in the following sense: 
For any compactly supported measure $\mu_D$ and any $t$, there exist $c, \epsilon > 0$ such that
\[
    \mu_t(x) \leq c \sqrt{x-\ell_t} \qquad \text{for } x \in [\ell_t, \ell_t+\epsilon].
\]
\end{prop}

On the one hand, square-root decay is of course achieved if $\mu_D = \delta_0$ (so that the free convolution is semicircle). On the other hand, Lee and Schnelli have presented a family of examples where decay at the edge is strictly faster than square root \cite[Lemma 2.7]{LeeSch2013}. Thus the above result cannot be improved. We also mention works providing sufficient conditions on $\mu_D$ for a matching lower bound, i.e., to ensure that extremal-edge decay is \emph{exactly} square root, such as \cite[Theorem 2.2]{BaoErdSch2020B} (which actually considers free convolution between two Jacobi measures, not our special case when one of them is semicircular).

This result also complements \cite[Corollary 5]{Bia1997}, which shows that decay near \emph{any} edge is at least as fast a \emph{cube} root. As Biane shows, this is in fact the correct power at a cusp when two connected components of the support merge. Thus the ``extremal'' restriction in the above proposition is necessary.

\begin{proof}
We adapt arguments of \cite{Bia1997} as follows. Biane considers the function $v_t(u) : \R \to [0,\infty)$ defined by
\[
    v_t(u) = \inf\left\{ v \geq 0 : \int_\R \frac{\mu_D(\diff x)}{(u-x)^2+v^2} \leq \frac{1}{t}\right\}
\]
and the open set $U_t = \{u \in \R : v_t(u) > 0\}$, then defines a certain homeomorphism $\psi_t : \R \to \R$ (whose exact form is not important to us now) and proves that, for all $u \in \R$, 
\[
    \mu_t(\psi_t(u)) = \frac{v_t(u)}{\pi t}.
\]

On the one hand, by \cite[Corollary 3]{Bia1997}, we have
\[
    u_t := \psi_t^{-1}(\ell_t) = \ell_t + tm_t(\ell_t).
\]
This is at most $\mathtt{l}(\mu_D)$ by \eqref{eqn:GuionnetMaidaMagic}, and in fact the inequality is strict since $m_t(\ell_t) > 0$:
\begin{equation}
\label{eqn:strictut}
    u_t < \mathtt{l}(\mu_D).
\end{equation}

On the other hand, let $x$ be such that $\mu_t(x) > 0$. Then $x = \psi_t(u)$ for some $u \in U_t$, and adapting the proofs of \cite[Proposition 4, Lemma 5]{Bia1997} we obtain
\begin{equation}
\label{eqn:freeconv_edgeestimate}
    \abs{\mu_t(x)\mu_t'(x)} \leq \frac{\abs{v_t(u)v'_t(u)}}{\pi^2t^2\psi'_t(u)} \leq \frac{\abs{v'_t(u)}}{2\pi^2 t \abs{v_t(u)}(1+v'_t(u)^2)} \leq \frac{1}{2\pi^2 t} \cdot \frac{1}{\abs{v_t(u)v'_t(u)}}.
\end{equation}
But the proof of \cite[Lemma 5]{Bia1997} shows that
\[
    v_t(u)v'_t(u) = \frac{\int_\R \frac{(x-u)}{((u-x)^2+v_t(u)^2)} \mu_D(\diff x)}{\int_\R \frac{1}{((u-x)^2+v_t(u)^2)} \mu_D(\diff x)} \geq \mathtt{l}(\mu_D) - u.
\]
For $u$ in some $[u_t,u_t+\epsilon]$ (corresponding via $\psi_t$ to $x$ in some $[\ell_t,\ell_t+\epsilon']$), this lower bound is strictly positive by \eqref{eqn:strictut}. By \eqref{eqn:freeconv_edgeestimate}, this suffices. 
\end{proof}

\addcontentsline{toc}{section}{References}
\bibliographystyle{alpha-abbrvsort}
\bibliography{complexitybib}

\end{document}